\documentclass[11pt, twoside]{article}

\usepackage{amsmath}
\usepackage{amssymb}
\usepackage{titletoc}
\usepackage{mathrsfs}
\usepackage{amsthm}

\usepackage{indentfirst}
\usepackage{color}
\usepackage{txfonts}

\allowdisplaybreaks

\def\bint{{\ifinner\rlap{\bf\kern.30em--}
\int\else\rlap{\bf\kern.35em--}\int\fi}\ignorespaces}

\def\sbint{{\ifinner\rlap{\bf\kern.32em--}
\hspace{0.078cm}\int\else\rlap{\bf\kern.45em--}\int\fi}\ignorespaces}

\def\red{\color{red}}

\def\rr{\mathbb{R}}
\def\rn{\mathbb{R}^n}

\def\nn{\mathbb{N}}
\def\zz{\mathbb{Z}}
\def\zn{\mathbb{Z}^n}

\def\lz{\lambda}

\def\dz{\delta}

\def\bz{\beta}

\def\vz{\varphi}

\def\tz{\theta}

\def\fg{\lceil\gamma\rceil}

\def\wz{\widetilde}

\def\ls{\lesssim}

\def\fz{\infty}
\def\az{\alpha}
\def\esup{\mathop\mathrm{\,ess\,sup\,}}
\def\einf{\mathop\mathrm{\,ess\,inf\,}}

\def\ca{{\mathcal A}}
\def\cb{{\mathcal B}}

\def\cf{{\mathcal F}}

\def\ck{{\mathcal K}}

\def\cm{{\mathcal M}}

\def\cq{{\mathcal Q}}

\def\cs{{\mathcal S}}

\def\lp{{L^p(\rn)}}

\def\r{\right}
\def\lf{\left}

\def\hs{\hspace{0.3cm}}

\def\r{\right}
\def\lf{\left}

\def\supp{{\mathop\mathrm{\,supp\,}}}

\def\loc{{\mathop\mathrm{\,loc\,}}}

\def\eqref#1{(\ref{#1})}

\def\func#1{\,\mathop{\mathrm{#1}}\,}

\def\supp{\func{supp}}

\newtheorem{theorem}{Theorem}[section]
\newtheorem{lemma}[theorem]{Lemma}
\newtheorem{corollary}[theorem]{Corollary}
\newtheorem{proposition}[theorem]{Proposition}

\theoremstyle{definition}
\newtheorem{remark}[theorem]{Remark}
\newtheorem{definition}[theorem]{Definition}
\newtheorem{assumption}[theorem]{Assumption}

\numberwithin{equation}{section}

\begin{document}

\title{\bf\Large Applications of
Hardy Spaces Associated with Ball Quasi-Banach Function Spaces
\footnotetext{\hspace{-0.35cm} 2010 {\it Mathematics Subject
Classification}. {Primary 42B30; Secondary 42B35, 42B25, 42B20, 47G30.}
\endgraf {\it Key words and phrases}. ball quasi-Banach
function space, Hardy space, $g$-function,
$g^\ast_\lambda$-function, atom, Calder\'on--Zygmund operator, pseudo-differential operator.
\endgraf This work is supported by the
National Natural Science Foundation  of China (Grant Nos.  11971058,
11761131002, 11671185, 11871254 and 11571289).}}
\date{ }
\author{Fan Wang, Dachun Yang\,\footnote{Corresponding author/{\red November 1,
2019}/Final version.}\, \ and Sibei Yang}
\maketitle

\vspace{-0.8cm}

\begin{center}
\begin{minipage}{13cm}
{\small {\bf Abstract}\quad Let $X$ be a ball quasi-Banach function space
satisfying some minor assumptions.
In this article, the authors establish the characterizations of $H_X(\mathbb{R}^n)$,
the Hardy space associated with $X$,
via the Littlewood--Paley $g$-functions and $g_\lambda^\ast$-functions.
Moreover, the authors obtain the boundedness of Calder\'on--Zygmund operators
on $H_X(\mathbb{R}^n)$. For the local Hardy-type space $h_X(\mathbb{R}^n)$ associated with $X$,
the authors also obtain the boundedness of $S^0_{1,0}(\mathbb{R}^n)$ pseudo-differential operators
on $h_X(\mathbb{R}^n)$ via first establishing the atomic characterization
of $h_X(\mathbb{R}^n)$. Furthermore, the
characterizations of $h_X(\mathbb{R}^n)$ by means of local molecules and
local Littlewood--Paley functions are also given.
The results obtained in this article
have a wide range of generality and can be applied to
the classical Hardy space, the weighted Hardy space,
the Herz--Hardy space, the Lorentz--Hardy space,
the Morrey--Hardy space, the variable Hardy space,
the Orlicz-slice Hardy space and their local versions.
Some special cases of these applications are even new and, particularly, in the case of
the variable Hardy space, the $g_\lambda^\ast$-function characterization obtained in this
article improves the known results via widening the range of $\lambda$.}
\end{minipage}
\end{center}

\vspace{0.3cm}

\section{Introduction}

It is well known that both the real-variable theory of function spaces and
the boundedness of operators are always two important topics of harmonic analysis.
Recall that the Lebesgue space $L^p(\rn)$ with $p\in(0,\fz)$ is defined by setting
$$L^p(\rn):=\left\{f \ \text{is measurable on} \ \rn :\
\|f\|_{\lp}:=\left[\int_{\rn}|f(x)|^p\,dx\right]^{1/p}<\infty\right\},$$
which is the most basic function space. When $p\in(1,\fz)$, the space $L^p(\rn)$ is
very useful in the study of the boundedness of operators.
However, when $p\in(0,1]$, the space $L^p(\rn)$ is not appropriate when considering
the boundedness of some operators; for instance,
the Riesz transform is not bounded on $L^1(\rn)$.

As natural generalizations and substitutes of $L^p(\rn)$ with $p\in(0,1]$,
a real-variable theory of classical Hardy spaces $H^p(\rn)$
was originally initiated  by Stein and Weiss \cite{sw} and then systematically developed by Fefferman and Stein \cite{fs}.
These celebrated articles \cite{sw} and \cite{fs} inspire many
new ideas for the real-variable theory of function spaces.
For instance, the characterizations of classical Hardy spaces reveal the important connections
among various notions in harmonic analysis, such as harmonic functions,
various maximal functions and various square functions.

On another hand, due to the need from applications for more inclusive classes of
function spaces than $L^p(\rn)$, many other function spaces are introduced;
for instance, weighted Lebesgue spaces,
Lorentz spaces, variable Lebesgue spaces, Orlicz spaces and  Morrey spaces.
These spaces and the Hardy-type spaces based on them have been investigated extensively.
Associated with these spaces, an important concept about function spaces is the
(quasi-)Banach function space, which is defined as follows
(see, for instance, \cite[Chapter 1]{br} for more details).

\begin{definition}\label{bfs}
Let $Y$ be a quasi-Banach space consisting of measurable functions on $\rn$.
Then it is called a \emph{quasi-Banach function space} if it satisfies
\begin{enumerate}
\item[(i)] $\|f\|_{Y}=0$ implies that $f=0$ almost everywhere;

\item[(ii)]  $|g|\le|f|$ in the sense of almost everywhere implies
that $\|g\|_{Y}\le\|f\|_{Y}$;

\item[(iii)]  $0\le f_{m}\uparrow f$ in the sense of almost
everywhere implies that $\|f_{m}\|_{Y}\uparrow\|f\|_{Y}$;

\item[(iv)]  $\mathbf{1}_E\in Y$ for any measurable set $E\subset\rn$ with finite measure.
Here and thereafter, $\mathbf{1}_E$ denotes the \emph{characteristic function} of $E$.
\end{enumerate}

Moreover, a Banach space $Y$, consisting of measurable functions on $\rn$, is called a
\emph{Banach function space} if it satisfies the above terms (i) through (iv) and

\begin{enumerate}
\item[(v)]  for any measurable set $E\subset\rn$ with finite measure,
there exists a positive constant $C_{(E)}$, depending on $E$,
such that, for any $f\in Y$,
\begin{equation*}
\int_{E} |f(x)|\,dx \le C_{(E)}\|f\|_{Y}.
\end{equation*}
\end{enumerate}
\end{definition}

One can show that Lebesgue spaces,
Lorentz spaces, variable Lebesgue spaces and Orlicz spaces  are (quasi-)Banach function spaces.
However, weighted Lebesgue spaces, Herz spaces, Morrey spaces and Musielak--Orlicz
spaces are not necessarily quasi-Banach
function spaces (see, for instance, \cite{st,wyyz,zwyy} for more details and examples).
Therefore, in this sense, the notion of
(quasi-)Banach function spaces is restrictive.
To extend it further so that weighted Lebesgue spaces, Herz spaces, Morrey spaces and Musielak--Orlicz
spaces are also included in the generalized framework,  Sawano et al. \cite{shyy} introduced
the following so-called ball quasi-Banach function spaces and ball Banach function spaces.

\begin{definition}
Let $Y$ be a quasi-Banach space consisting of measurable functions on $\rn$.
Then it is called a \emph{ball quasi-Banach function space} if it satisfies
(i), (ii) and (iii) of Definition \ref{bfs} and
\begin{enumerate}
\item[{\rm(vi)}] $\mathbf{1}_B\in Y$ for any ball $B\subset\rn$.
\end{enumerate}

A ball quasi-Banach function space $X$ is called a \emph{ball Banach function space} if the norm of $X$ satisfies
\begin{enumerate}
\item[{\rm(vii)}] for any $f,\ g\in X$,
$$\|f+g\|_X\leq\|f\|_X+\|g\|_X;$$
\item[{\rm(viii)}] for any ball $B\subset\rn$, there exists a positive constant $C_{(B)}$ such that,
for any $f\in X$,
\begin{equation*}
\int_B|f(x)|\,dx\leq C_{(B)}\|f\|_X.
\end{equation*}
\end{enumerate}
\end{definition}

Furthermore, for any given ball quasi-Banach function space $X$, the Hardy-type space $H_X(\rn)$
(see Definition \ref{H} below or \cite[Definition 6.17]{s18}) was introduced via the grand maximal function, and its
several equivalent characterizations, respectively, in terms of Lusin-area functions,
atoms and radial or non-tangential maximal functions, were established in \cite{shyy}.
Besides, the local Hardy-type $h_X(\rn)$ (see Definition \ref{lh} below) was also introduced
in \cite{shyy}, and the relation between $H_X(\rn)$ and $h_X(\rn)$ was established.

However, for the Hardy-type spaces $H_X(\rn)$ and $h_X(\rn)$
associated with the ball quasi-Banach function space $X$,
some natural questions could be asked.
For instance, could the spaces $H_X(\rn)$ and $h_X(\rn)$ be characterized
via Littlewood--Paley $g$-functions or $g_\lambda^\ast$-functions?
Are Calder\'on--Zygmund operators or pseudo-differential
operators bounded on the Hardy-type spaces $H_X(\rn)$ or $h_X(\rn)$?

The main targets of this article are to answer these questions under the condition that
a ball quasi-Banach function space $X$ satisfies some minor assumptions.
More precisely, let $X$ be a ball quasi-Banach function space satisfying Assumptions \ref{a} and
\ref{a2} with the same $s\in(0,1]$ as below. We establish the characterizations of $H_X(\mathbb{R}^n)$
via Littlewood--Paley $g$-functions and $g_\lambda^\ast$-functions.
Moreover, we obtain the boundedness of Calder\'on--Zygmund operators
on $H_X(\mathbb{R}^n)$. For the local Hardy-type space $h_X(\mathbb{R}^n)$,
we obtain the boundedness of $S^0_{1,0}(\rn)$ pseudo-differential operators
on it via first establishing its atomic characterization. Furthermore,
the characterizations of $h_X(\mathbb{R}^n)$ are also given by means of local molecules,
local Lusin-area functions, local Littlewood--Paley $g$-functions and local Littlewood--Paley $g_\lambda^\ast$-functions.

It is worth pointing out that the results obtained in this article
have a wide range of generality. More precisely, the Hardy
type space $H_X(\rn)$ or its local version $h_X(\rn)$
considered in this article includes the classical (local) Hardy space,
the (local) weighted Hardy space, the (local)  Herz--Hardy space,
the (local) Lorentz--Hardy space, the (local)  Morrey--Hardy space,
the (local) variable Hardy space and the (local) Orlicz-slice Hardy space.
Even for the Morrey--Hardy space, the characterizations of the Littlewood--Paley $g$-function
and  $g_\lambda^\ast$-function obtained in this article are new [see Remark \ref{r2.3}(v) below] and,
in the  case of the variable Hardy space,
the Littlewood--Paley $g_\lambda^\ast$-function characterization of the Hardy space
$H_X(\rn)$ obtained in this article improves the known results [see Remark
\ref{r2.3}(vi) below for the details]. Moreover,
for both the local Lorentz--Hardy and the local Orlicz-slice Hardy space,
the boundedness of $S^0_{1,0}(\rn)$ pseudo-differential operators established
in this article is new (see Remark \ref{r4.2} below for the details).

The layout of this article is as follows. In Section \ref{s2}, we establish the
real-variable characterizations of the Hardy space
$H_X(\rn)$ via Littlewood-Paley $g$-functions and $g_\lambda^\ast$-functions
(see Theorem \ref{g} below). It is worth pointing out that,
when $X:=L^p(\rn)$ with $p\in(0,1]$, the range of the index $\lambda$
in Theorem \ref{g} below coincides with the classical case of $\lambda\in(2/p,\infty)$
[see Remark \ref{r2.3}(i) below]. Denote by $L^{p(\cdot)}(\rn)$
the variable Lebesgue space on $\rn$. When $X:=L^{p(\cdot)}(\rn)$ with
$$0<\einf_{x\in\rn}p(x)\le1,$$
Theorem \ref{g} improves the results obtained in \cite{yyyz}
via \emph{widening the range of the parameter
$\lambda$} [see Remark \ref{r2.3}(vi) below for the details].

To characterize the Hardy space $H_X(\rn)$ via the Littlewood--Paley $g$-function,
we borrow some ideas from \cite{u} and \cite{yyyz}. More precisely,
we use a maximal function $(\phi_t^\ast f)_a$ [see \eqref{p} below for its definition],
which is similar to the maximal function of Peetre type, to control the $g$-function.
The main difficulty appeared in this case is how to estimate the norm
of $(\phi_t^\ast f)_a$ in the space $X$. Differently from the case of variable Lebesgue
spaces as in \cite{yyyz}, we do not have the concrete form of the norm of the space $X$,
so that we can not compute $\|(\phi_t^\ast f)_a\|_X$ directly as in the proof
of \cite[Theorem 6.2]{yyyz}. To overcome this difficulty, we introduce the assumption
that a Fefferman--Stein type vector-valued maximal inequality holds true for the space
$X$ [see \eqref{ma2} below for the details], which indeed holds true for many function spaces
(see Remark \ref{r2.1} below).

Moreover, to characterize the Hardy space $H_X(\rn)$ via the
Littlewood--Paley $g_\lambda^\ast$-function, we use a different
method from the case of $g$-functions. More precisely,
motivated by \cite[Theorem 1.1]{a11},
we first establish the precise quantitative control of
the (quasi-)norm of the Lusin-area function $\ca_\alpha$, with the aperture
$\alpha\in[1,\fz)$, in $X$ via the aperture $\alpha$ and
the (quasi-)norm of the Lusin-area function $\ca_1$ in $X$ (see
Lemma \ref{af} below for the details), which was
obtained in \cite[Theorem 1.1]{a11} when $X:=L^p(\rn)$ with $p\in(0,\fz)$.
By this and the relation of the Lusin-area function with aperture and
the Littlewood--Paley $g_\lambda^\ast$-function, we further obtain
the Littlewood--Paley $g_\lambda^\ast$-function characterization of the Hardy space $H_X(\rn)$.

Furthermore, in Section \ref{s3}, the boundedness of some Calder\'on--Zygmund
operators on the Hardy space $H_X(\rn)$ is established
(see Theorems \ref{ccz} and \ref{gcz} below).
The main tools to obtain the boundedness
of Calder\'on--Zygmund operators on $H_X(\rn)$ are the atomic and the maximal
function characterizations of $H_X(\rn)$ (see \cite{shyy} or
Lemmas \ref{atom}, \ref{r} and \ref{mch} below).
Moreover, in Section \ref{s4}, the boundedness of $S^0_{1,0}(\rn)$ pseudo-differential
operators on the local Hardy space $h_X(\rn)$ is obtained (see Theorem \ref{pdo} below).
The main idea to obtain the
boundedness of $S^0_{1,0}(\rn)$ pseudo-differential operators on $h_X(\rn)$
is similar to that of Calder\'on--Zygmund operators. More precisely,
we first establish the atomic characterization of the space $h_X(\rn)$.
In Section \ref{s5}, the molecular characterization of $h_X(\rn)$ is obtained
via the molecular characterization of the Hardy space $H_X(\rn)$ and the
atomic characterization of $h_X(\rn)$ (see Theorem \ref{mech} below).
Furthermore, using a Calder\'on reproducing formula (see \cite[Proposition 1.1.6]{gl14}
or \eqref{11.5a} below) and an argument similar to that used in the proof of
the Littlewood--Paley function characterizations of $H_X(\rn)$,
we establish the characterizations of $h_X(\mathbb{R}^n)$
via local Littlewood--Paley functions in Section \ref{s5}
(see Theorem \ref{lpl} below).

Finally, we make some conventions on notation. Throughout the whole article,
let $\nn:=\{1,2,\dots\}$, $\zz_+:=\nn\cup\{0\}$ and ${\vec 0}_n$
be the \emph{origin} of $\rn$. We always use $C$ to denote a \emph{positive constant},
independent of the main parameters involved, but perhaps varying from line to line.
Moreover, we also use $C_{(\alpha,\beta,\dots)}$ to denote a positive constant
depending on the parameters $\alpha, \beta,\dots$.
The \emph{symbol} $f\lesssim g$ means that $f\le Cg$ and, if $f\lesssim g$ and $g\lesssim f$,
we then write $f\sim g$. We also use the following
convention: If $f\le Cg$ and $g=h$ or $g\le h$, we then write $f\ls g\sim h$
or $f\ls g\ls h$, \emph{rather than} $f\ls g=h$
or $f\ls g\le h$. The \emph{symbol} $\lceil s \rceil$ for any $s\in\rr$
denotes the smallest integer not less than $s$ and the \emph{symbol} $\lfloor s \rfloor$ for any $s\in\rr$
denotes the largest integer not greater than  $s$. For any subset $E$ of $\rn$,
we denote by $E^\complement$ the \emph{set} $\rn\setminus E$ and by $\mathbf{1}_{E}$
its \emph{characteristic function}. For any multi-index
$\alpha:=(\alpha_1, \dots,\alpha_n)\in\zn_+:=(\zz_+)^n$, let
$|\alpha|:= \alpha_1+\cdots+\alpha_n.$
Furthermore, for any cube $Q\subset\rn$,
$rQ$ means a cube with the same center as $Q$
and $r$ times the side length of $Q$; the same convention as this is made for any
ball $B\subset\rn$. We always use $\ell(Q)$ to denote the side length
of the cube $Q$ and $Q(x, \ell(Q))$ to denote the cube
with center $x$ and side length $\ell(Q)$. The operator $M$ always
denotes the \emph{Hardy--Littlewood maximal operator}, which is defined by setting,
for any $f\in L^1_{\loc}(\rn)$ (the set of all locally integrable functions  on $\rn$) and $x\in\rn$,
$$M(f)(x):=\sup_{r\in(0,\infty)}\frac{1}{|B(x,r)|}\int_{B(x,r)}|f(y)|\,dy,$$
where $B(x,r)$ denotes the \emph{ball} with the center $x$ and the radius $r$.
We use $\mathcal{S}(\rn)$ to denote the \emph{Schwartz space}
equipped with the well-known classical topology determined by a countable family
of norms, while $\mathcal{S}'(\rn)$ means its topological dual
equipped the weak-$\ast$ topology.
For any $\varphi\in\mathcal{S}(\rn)$ and $t\in(0,\infty)$,
let $\varphi_t(\cdot):= t^{-n}\varphi(\frac{\cdot}{t})$. Finally, for any $q\in[1,\fz]$,
we denote by $q'$ its \emph{conjugate exponent}, namely, $1/q + 1/q'=1$.

\section{Characterizations of $H_X(\rn)$ via $g$-function and
$g_\lambda^\ast$-function}\label{s2}

In this section, we establish the $g$-function and the $g_\lambda^\ast$-function
characterizations of the Hardy space $H_X(\rn)$. We begin with the notion of
the Hardy space $H_X(\rn)$ associated with $X$. In what follows,
let $\mathbb{R}_+^{n+1}:=\rn\times(0,\infty)$.

\begin{definition}\label{H}
Let $X$ be a ball quasi-Banach function space. Let $\Phi \in \mathcal{S}(\rn)$
satisfy $\int_{\rn}\Phi(x)\,dx\neq 0$ and $b\in(0,\infty)$ sufficiently large. Then the
\emph{Hardy space} $H_X(\rn)$ associated with $X$ is defined by setting
$$H_X(\rn):=\lf\{f \in\mathcal{S}'(\rn):\ \|f\|_{H_X(\rn)}:=
\lf\|M_b^{\ast\ast}(f,\Phi)\r\|_X<\infty\r\},$$
where $M_b^{\ast\ast}(f,\Phi)$ is defined by setting, for any $x\in\rn$,
\begin{equation}\label{mbaa}
M_b^{\ast\ast}(f,\Phi)(x) :=\sup_{(y,t)\in\mathbb{R}_+^{n+1}}
\frac{|\Phi_t\ast f(x-y)|}{(1+t^{-1}|y|)^b}.
\end{equation}
\end{definition}

We now recall the notions of the $p$-convexification and the
convexity of $X$ (see, for instance, \cite[Chapter 2]{ors}
and \cite[Definition 1.d.3]{lt} for more details).

\begin{definition}
Let $X$ be a ball quasi-Banach function space and $p \in (0,\infty)$.
\begin{enumerate}
\item[{\rm(i)}] The \emph{$p$-convexification} $X^p$ of $X$ is defined by setting
$$X^p :=\lf\{f\ \text{is measurable on}\  \rn:\ |f|^p\in X\r\},$$
equipped with the \emph{quasi-norm} $\|f\|_{X^p}:=\||f|^p\|_X^{1/p}$.
\item[{\rm(ii)}] The space $X$ is said to be \emph{$p$-convex} if there exists a positive
constant $C$ such that, for any $\{f_j\}_{j\in\nn} \subset X^{1/p}$,
$$\left\|\sum_{j=1}^\infty|f_j|\right\|_{X^{1/p}}\leq C\sum_{j=1}^\infty\|f_j\|_{X^{1/p}}.$$
In particular, when $C=1$, $X$ is said to be strictly \emph{$p$-convex}.
\end{enumerate}
\end{definition}

Denote by $L^{1}_{\rm loc}(\rn)$ the set of all locally integral functions
on $\rn$. For any $\theta\in(0,\infty)$, the \emph{powered Hardy--Littlewood maximal operator}
$M^{(\theta)}$ is defined by setting,
for any $f\in L^{1}_{\rm loc}(\rn)$ and $x\in\rn$,
\begin{equation*}
M^{(\theta)}(f)(x):=\lf\{M\lf(|f|^\theta\r)(x)\r\}^{\frac{1}{\theta}}.
\end{equation*}

Moreover, we also need some basic assumptions on $X$ as follows
(see also \cite[(2.8) and (2.9)]{shyy}).

\begin{assumption}\label{a}
Let $X$ be a ball quasi-Banach function space. For some $\theta,\ s\in (0,1]$ and $\theta<s$,
there exists a positive constant $C$ such that, for any $\{f_j\}_{j=1}^\infty \subset L^{1}_{\rm loc}(\rn)$,
\begin{equation}\label{ma}
\left\|\left\{\sum_{j=1}^\infty\left[M^{(\theta)}(f_j)\right]^s\right\}^{\frac{1}{s}}\right\|_X\leq C
\left\|\left\{\sum_{j=1}^\infty|f_j|^s\right\}^{\frac{1}{s}}\right\|_X
\end{equation}
and
\begin{equation}\label{ma2}
\left\|\left\{\sum_{j=1}^\infty\left[M^{(\theta)}(f_j)\right]^s\right\}^{\frac{1}{s}}\right\|_{X^\frac{s}{2}}\le
C \left\|\left\{\sum_{j=1}^\infty|f_j|^s\right\}^{\frac{1}{s}}\right\|_{X^\frac{s}{2}}.
\end{equation}
\end{assumption}

\begin{remark}\label{r2.1}
The inequalities \eqref{ma} and  \eqref{ma2} are called the \emph{Fefferman--Stein vector-valued maximal inequality}.
If $X:=L^p(\rn)$ with $p\in(1,\infty)$, $\theta=1$ and $s\in(1,\infty]$, the inequality \eqref{ma}
was originally established by Fefferman and Stein \cite[Theorem 1]{fs2}.
Here we give some examples of ball quasi-Banach
function spaces satisfying \eqref{ma} and \eqref{ma2} as follows.

\begin{enumerate}
\item[{\rm(a)}] By \cite[Theorem 1]{fs2}, we know that \eqref{ma} also holds true when $\theta,\ s\in(0,1]$,
$\theta<s$ and $X:=L^p(\rn)$ with $p\in(\theta,\infty)$.
Since $(L^p(\rn))^{s/2}=L^{ps/2}(\rn)$ for any $p,\ s\in(0,\infty)$,
it follows that \eqref{ma2} holds true when $\theta,\ s\in(0,1]$, $\theta<s$ and $X:=L^p(\rn)$ with
$p\in(2\theta/s,\infty)$.

\item[{\rm(b)}] For any $q\in[1,\fz]$, denote by $A_q(\rn)$ the class of all \emph{Muckenhoupt weights} (see, for instance,
\cite[Definitions 7.1.1 and 7.1.3]{gl} for its definition). For any $p\in(0,\fz)$ and $w\in A_\fz(\rn)$,
the \emph{weighted Lebesgue space} $L^p_w(\rn)$ is defined by setting
$$L^p_w(\rn):=\left\{f \ \text{measurable} :\ \|f\|_{L^p_w(\rn)}:=\left[\int_{\rn}|f(x)|^pw(x)\,dx\right]^{1/p}<\infty\right\}.$$
Let $X:=L^p_w(\rn)$ with $p\in(0,\fz)$ and $w\in A_\fz(\rn)$. Then $X$ is a ball quasi-Banach function space
(see \cite[Section 7.1]{shyy}), but it might not be a quasi-Banach function space (see \cite[Remark 5.22(ii)]{wyyz}).
Moreover, \eqref{ma} holds true for  any $\theta,\ s\in(0,1]$,
$\theta<s$ and $X:=L^p_w(\rn)$ with $p\in(\theta,\infty)$ and $w\in A_{p/\theta}(\rn)$ (see, for instance, \cite[Theorem 3.1(b)]{aj}).
By \eqref{ma} and the fact that $(L^p_w(\rn))^{s/2}=L^{ps/2}_w(\rn)$ for any $p,\ s\in(0,\infty)$ and $w\in A_\fz(\rn)$,
we conclude that \eqref{ma2} holds true for any $\theta,\ s\in(0,1]$, $\theta<s$ and $X:=L^p_w(\rn)$
with $p\in(2\theta/s,\infty)$ and $w\in A_{p/\theta}(\rn)$.

\item[{\rm(c)}] For the cube $Q(\vec{0}_n,1)$ and any $j\in\nn$, let
$$S_j(Q(\vec{0}_n,1)):=Q(\vec{0}_n,2^{j+1})\setminus Q(\vec{0}_n,2^j).$$
Here and thereafter, $\vec{0}_{n}$ denotes the \emph{origin} of $\rr^{n}$.
For any $\alpha\in\mathbb{R}$ and $p,\ q\in(0,\infty]$, the \emph{Herz space} $\ck_{p,q}^\alpha(\rn)$
is defined to be the set of all measurable functions $f$ on $\rn$ satisfying
$$\displaystyle\|f\|_{\ck_{p,q}^\alpha(\rn)}:=
\lf\|\mathbf{1}_{Q(\vec{0}_n,2)}f\r\|_{L^p(\rn)}+
\left\{\sum_{j=1}^\infty\lf[2^{j\alpha}\lf\|\mathbf{1}_{S_j(Q(\vec{0}_n,1))}
f\r\|_{L^p(\rn)}\r]^q\right\}^\frac{1}{q}<\fz.
$$
Let $X:=\ck_{p,q}^\alpha(\rn)$ with $p,\ q\in(0,\fz)$ and $\az\in(-n/p,\fz)$.
Then \eqref{ma} holds true for the space $X$, any $s\in(0,1]$ and any
$\theta\in(0,\min\{s,p, [\alpha/n+1/p]^{-1}\})$
(see, for instance, \cite{i} and  \cite{ly} for more details). Moreover,
similarly to \eqref{ma}, \eqref{ma2} holds true for the space $X$, any
$s\in(0,1]$ and any $\theta\in(0,\min\{s,(sp)/2, [\alpha/(sn)+2/(sp)]^{-1}\})$.
By \cite[Remark 5.22(ii)]{wyyz}, it is easy to know that a Herz space might not be a quasi-Banach function  space.

\item[{\rm(d)}] The \emph{Lorentz space} $L^{p,q}(\rn)$ is defined to be
the set of all measurable functions $f$ on $\rn$ satisfying that, when $p,\ q\in(0,\fz)$,
$$\|f\|_{L^{p,q}(\rn)}:=
\left\{\int_0^\infty[t^{1/p}f^\ast(t)]^q\,\frac{dt}{t}\right\}^{\frac{1}{q}}<\fz$$
and, when $p\in(0,\fz)$ and $q=\fz$,
$$\|f\|_{L^{p,q}(\rn)}:=\sup_{t\in(0,\fz)}\lf\{t^{1/p}f^\ast(t)\r\}<\fz,$$
where $f^\ast$, the \emph{decreasing rearrangement function} of $f$, is defined by setting,
for any $t\in[0,\fz)$,
$$f^\ast(t):=\inf\{s\in(0,\fz):\ \mu_f(s)\le t\}$$
with $\mu_f(s):=|\{x\in\rn:\ |f(x)|>s\}|$.
Let $X:=L^{p,q}(\rn)$ with $p\in(0,\infty)$ and $q\in(0,\infty]$.
Then \eqref{ma} holds true for the space $X$, any $s\in(0,1]$ and any
$\theta\in(0, \min\{s, p\})$
(see, for instance, \cite[Theorem 2.3(iii)]{ccmp}); similarly to \eqref{ma},
\eqref{ma2} also holds true for the space $X$, any $s\in(0,1]$
and any $\theta\in(0,\min\{s,(sp)/2\})$.

\item[{\rm(e)}] Let $0<q\le p\le\fz$.
Recall that the \emph{Morrey space} ${\mathcal M}^p_q(\rn)$ is defined to be the set of all
$f\in L^q_{{\rm loc}}(\rn)$ such that
\begin{equation*}
\|f\|_{{\mathcal M}^p_q(\rn)}:=\sup_{B\subset\rn}
|B|^{\frac1p-\frac1q}\left[\int_{B}|f(y)|^q\,dy\right]^{\frac1q}<\fz,
\end{equation*}
where the supremum is taken over all balls $B\subset\rn$.
Let $X:=\mathcal{M}^p_q(\rn)$ with $p\in(0,\infty]$ and $q\in(0,p]$.
Then \eqref{ma} holds true
for the space $X$, any $s\in(0,1]$ and any $\theta\in(0,\min\{s,q\})$
(see, for instance, \cite{cf}, \cite{h13}, \cite{h15},
\cite{st05} and \cite{tx05}); similarly, \eqref{ma2}
also holds true for the space $X$, any $s\in(\theta,1]$ and any $\theta\in(0,\min\{s,(sq)/2\})$.
Recall that a Morrey space might not be a quasi-Banach function space (see, for instance, \cite[(1.5)]{st})

\item[{\rm(f)}] Let $p\in[0,\fz)$ and $\varphi:\ {\mathbb R}^n \times[0,\infty)\to[0,\infty)$
be a function such that, for almost every  $x\in\rn$, $\varphi(x,\cdot)$ is an Orlicz function.
The function $\varphi$ is said to be of \emph{uniformly upper} (resp.,
\emph{lower}) \emph{type $p\in[0,\fz)$} if there exists a positive constant $C$ such that,
for any $x\in\rn$, $t\in[0,\fz)$ and $s\in[1,\fz)$ (resp., $s\in[0,1]$),
$\varphi(x,st)\le Cs^p\varphi(x,t)$.

The function $\varphi$ is said to satisfy the \emph{uniform Muckenhoupt condition
for some $q\in[1,\fz)$}, denoted by $\varphi\in{\mathbb A}_q(\rn)$, if, when $q\in (1,\fz)$,
\begin{equation*}
[\varphi]_{\mathbb{A}_q(\rn)}:=\sup_{t\in
(0,\fz)}\sup_{B\subset\rn}\frac{1}{|B|^q}\int_B
\vz(x,t)\,dx \lf\{\int_B
[\vz(y,t)]^{-q'/q}\,dy\r\}^{q/q'}<\fz,
\end{equation*}
where $1/q+1/q'=1$, or
\begin{equation*}
[\vz]_{\mathbb{A}_1(\rn)}:=\sup_{t\in (0,\fz)}
\sup_{B\subset\rn}\frac{1}{|B|}\int_B \vz(x,t)\,dx
\lf(\esup_{y\in B}[\vz(y,t)]^{-1}\r)<\fz.
\end{equation*}
Here the first suprema are taken over all $t\in(0,\fz)$ and the
second ones over all balls $B\subset\rn$.

The class ${\mathbb A}_\infty(\rn)$ is defined by setting
$${\mathbb A}_\infty(\rn):=\bigcup_{q\in[1,\infty)}{\mathbb A}_q(\rn).
$$

For any given $\varphi\in{\mathbb A}_\infty(\rn)$,
the \emph{critical weight index} $q(\varphi)$ is defined by setting
$$q(\varphi):=\inf\{q \in[1,\infty):\ \varphi\in{\mathbb A}_q(\rn)\}.$$

Then the function $\varphi:\ \rn\times[0,\fz)\to[0,\fz)$ is called
a \emph{growth function} if the following hold true:
\begin{enumerate}
\item[(f)$_1$] $\varphi$ is a \emph{Musielak--Orlicz function}, namely,
\begin{itemize}
\item[] $\varphi(x,\cdot)$ is an Orlicz function for almost every given $x\in\rn$;
\item[] $\varphi(\cdot,t)$ is a measurable function for any given $t\in[0,\fz)$.
\end{itemize}
 \item[(f)$_2$] $\varphi\in {\mathbb A}_{\fz}(\rn)$.
\item[(f)$_3$] The function $\varphi$ is of
uniformly lower type $p$ for some $p\in(0,1]$ and of uniformly
upper type 1.
\end{enumerate}

For a growth function $\varphi$, a measurable function $f$ on $\rn$ is said to
be in the \emph{Musielak--Orlicz space} $L^{\varphi}(\rn)$ if $\int_{\rn}\varphi(x,|f(x)|)\,dx<\fz$,
equipped with the (quasi-)norm
$$\|f\|_{L^{\varphi}(\rn)}:=\inf\left\{\lambda\in(0,\infty): \int_{\rn}
\varphi\left(x,\frac{|f(x)|}{\lambda}\right)dx\leq 1\right\}.$$
Let $X:=L^{\varphi}(\rn)$, where $\varphi$ is a growth function with uniformly
lower type $p_\varphi^-$ and uniformly upper type $p_\varphi^+$.
Then \eqref{ma} holds true for the space $X$,
any $s\in(0,1]$ and any $\theta\in(0, \min\{s,p_\varphi^-/q(\varphi)\})$
(see, for instance, \cite[Theorem 7.14(i)]{shyy}); similarly, \eqref{ma2} also
holds true for the space $X$, any $s\in(0,1]$ and any
$\theta \in (0,\min\{s,2p_\varphi^-/(sq(\varphi))\})$.
Observe that a Musielak--Orlicz space might not be a quasi-Banach function space (see,
for instance, \cite[Remark 5.22(ii)]{wyyz}).

\item[{\rm(g)}] Let $p(\cdot):\ {\mathbb R}^n\to[0,\infty)$ be a measurable function.
Then the \emph{variable Lebesgue space} $L^{p(\cdot)}(\rn)$ is defined to be the set of
all measurable functions $f$ on $\rn$ such that
\begin{equation*}
\|f\|_{L^{p(\cdot)}(\rn)}:=\inf\left\{\lambda\in(0,\fz):\ \int_{{\mathbb R}^n}
\left[\frac{|f(x)|}{\lambda}\right]^{p(x)}dx\le 1\right\}<\infty.
\end{equation*}
For any measurable function $p(\cdot):\ {\mathbb R}^n\to(0,\infty)$, let
\begin{equation}\label{eq-p}
p_+:=\esup_{x\in\rn}p(x)\ \ \text{and}\ \ p_{-}:=\einf_{x\in\rn}p(x).
\end{equation}
A function $p(\cdot):\ \rn\to(0,\fz)$ is said to be \emph{globally log-H\"older continuous}
if there exist positive constants $p_\fz$ and $C$ such that, for any $x,\ y\in\rn$,
$$|p(x)-p(y)|\leq\frac{C}{\log(1/|x-y|)}\quad \text{when}\quad |x-y|\le\frac{1}{2},$$
and
$$|p(x)-p_\fz|\leq\frac{C}{\log(e+|x|)}.$$

Let $X:=L^{p(\cdot)}(\rn)$ with $p(\cdot)$ being globally log-H\"older continuous.
Then \eqref{ma} holds true for the space $X$, any $s\in(0,1]$
and any $\theta\in(0, \min\{s,p_-\})$ (see, for instance, \cite{uf} and \cite{uw});
similarly, \eqref{ma2} also holds true for the space $X$,
any $s\in(0,1]$ and any $\theta\in(0, \min\{s,sp_-/2\})$.

\item[{\rm(h)}] Let $t,\ r\in(0,\fz)$ and $\Phi$ be an Orlicz function
with lower type $p_\Phi^-\in(0,\fz)$ and upper type $p_\Phi^+\in(0,\fz)$.
The \emph{Orlicz-slice space} $(E_\Phi^r)_t(\rn)$ is defined to
be the set of all measurable functions $f$ such that
$$\|f\|_{(E_\Phi^r)_t(\rn)}:=\lf\{\int_{\rn}\lf[\frac{\|f\mathbf{1}_{B(x,t)}\|_{L^\Phi(\rn)}}
{\|\mathbf{1}_{B(x,t)}\|_{L^\Phi(\rn)}}\r]^r\,dx\r\}^{1/r}<\fz.
$$
Let $X:=(E_\Phi^r)_t(\rn)$. It was proved in \cite[Lemma 2.28]{zyyw}
that $X$ is a ball quasi-Banach function space; however, it might not be a quasi-Banach function space
(see, for instance, \cite[Remark 7.4(i)]{zwyy}). Moreover,
the assumption \eqref{ma} holds true for the space $X$, any
$s\in(0,1]$ and any $\tz\in(0,\min\{s,p_\Phi^-,r\})$ (see \cite[Lemma 4.3]{zyyw});
similarly, \eqref{ma2} also holds true for the space $X$,
any $s\in(0,1]$ and any $\theta\in(0,\min\{s,2p_\Phi^-/s,r\})$.
\end{enumerate}
\end{remark}

To state the following assumption on $X$, we need the notion of the associate space.
For any ball Banach function space $Y$, the \emph{associate space}
(also called \emph{K\"{o}the dual}) $Y'$ of $Y$ is defined by setting
$$Y':=\lf\{f \ \text{measurable}:\ \|f\|_{Y'}:=
\sup\{\|fg\|_{L^1(\rn)}:\ g \in Y,\ \|g\|_Y=1\}<\infty\r\}
$$
(see \cite[Chapter 1, Section 2]{br} for the details).
For any ball Banach function space $Y$, $Y'$ is also a
ball Banach function space (see \cite[Proposition 2.3]{shyy}).

\begin{assumption}\label{a2}
Assume that $X$ is a ball quasi-Banach function space, there exists $s\in(0,1]$
such that $X^{1/s}$ is also a ball Banach function space, and there exist $q\in(1,\fz]$
and $C\in(0,\fz)$ such that, for any $f\in (X^{1/s})'$,
\begin{equation}\label{ma21}
\left\|M^{((q/s)')}(f)\right\|_{(X^{1/s})'}\le  C \|f\|_{(X^{1/s})'}.
\end{equation}
\end{assumption}

\begin{remark}
We point out that \cite[Theorems 2.10, 3.7 and 3.21]{shyy} need the additional assumption
that there exists $s\in(0,1]$ such that $X^{1/s}$ is a ball Banach function space.
Indeed, this assumption ensures that $(X^{1/s})'$ is also a ball Banach function space, which implies that,
for any $f\in (X^{1/s})'$,  $f\in L^1_{\loc}(\rn)$ and hence the Hardy-Littlewood maximal
operator can be defined on $(X^{1/s})'$.
\end{remark}

\begin{remark}\label{r2.2}
By the definitions of $X^{1/s}$ and $(X^{1/s})'$, we know that \eqref{ma21} is equivalent to
that there exists a positive constant $C$ such that, for any $f\in [(X^{1/s})']^{1/(q/s)'}$,
\begin{equation}\label{ma22}
\left\|M(f)\right\|_{[(X^{1/s})']^{1/(q/s)'}}\leq C \left\|f\right\|_{[(X^{1/s})']^{1/(q/s)'}}.
\end{equation}
Here we give several examples of
function spaces satisfying \eqref{ma22} as follows.

\begin{enumerate}
\item[{\rm(a)}] Let $X:=L^p(\rn)$ with $p\in(0,\infty)$. Then,
for any $s\in(0,\min\{1,p\})$ and $q\in(\max\{1,p\},\infty]$,
it holds true that
$$[(X^{1/s})']^{1/(q/s)'} =L^{(p/s)'/(q/s)'}(\rn),$$
which, combined with the fact that $M$ is bounded on $L^r(\rn)$ for any $r\in(1,\infty)$, further
implies that \eqref{ma22} holds true in this case.

\item[{\rm(b)}] Let $X:=L^p_w(\rn)$ with $p\in(0,\infty)$ and $w\in A_\fz(\rn)$.
Then, for any $s \in (0,\min\{1,p\})$, $w\in A_{p/s}(\rn)$ and $q\in(\max\{1,p\},\infty]$
large enough such that
$w^{1-(p/s)'}\in A_{(p/s)'/(q/s)'}(\rn)$,  it holds true that $(p/s)'/(q/s)'>1$ and
$$[(X^{1/s})']^{1/(q/s)'}=L^{(p/s)'/(q/s)'}_{w^{1-(p/s)'}}(\rn),$$
which, together with the boundedness of $M$ on the weighted Lebesgue space
(see, for instance, \cite[Theorem 7.3]{d}), further implies that \eqref{ma22} holds true in this case.

\item[{\rm(c)}] Let $X:=\ck^\alpha_{p,r}(\rn)$ with $p,\ r\in(0,\infty)$ and $\alpha\in(-n/p,\infty)$.
Then, for any $s\in(0,\min\{p,r\})$ and $q\in(\max\{1,p\},\infty]$, it holds true that $p/s>1$, $r/s>1$ and $(p/s)'/(q/s)'>1$.
Moreover, by \cite[Corollary 1.2.1]{lyh}, we know that
$$[(X^{1/s})']^{1/(q/s)'} = \ck_{(p/s)'/(q/s)',
(r/s)'/(q/s)'}^{-\alpha s/(q/s)'}(\rn).$$
From this and the fact that $M$ is bounded on $\ck^\alpha_{p,r}(\rn)$ for any $p\in(1,\infty)$ (see, for instance, \cite[Corollary 4.5]{i}),
it follows that \eqref{ma22} holds true in this case.

\item[{\rm(d)}] Let $X:=L^{p,r}(\rn)$ with $p\in(0,\infty)$ and $r\in(0,\infty)$.
Then, for any $s\in(0,\min\{1,p,r\})$ and $q\in(\max\{1,p,r\},\infty]$, it holds true that
$p/s>1$, $r/s>1$, $(p/s)'/(q/s)'>1$ and $(r/s)'/(q/s)'>1$.
Moreover, by \cite[Theorem 1.4.16]{gl}, we find that
$$[(X^{1/s})']^{1/(q/s)'}=L^{(p/s)'/(q/s)',(r/s)'/(q/s)'}(\rn).$$
From this and the fact that $M$ is bounded on $L^{p,r}(\rn)$ for any $p\in(1,\infty)$ and $r\in(1,\infty)$ (see, for instance, \cite[Theorem 2.3(iii)]{ccmp}),
we deduce that \eqref{ma22} holds true in this case.

\item[{\rm(e)}] Let $X:={\mathcal M}^p_r(\rn)$ with $p\in(0,\infty)$ and $r\in(0,p]$.
Assume that $1<p_1\leq p_2<\infty$.
A function $b$ on $\rn$ is called a $(p_2', p_1')$-block if  $\supp(b)\subset Q$ with $Q\in \cq$, and
$$\left[\int_Q|b(x)|^{p_1'}\,dx\right]^{\frac{1}{p_1'}}\leq|Q|^{\frac{1}{p_2}-\frac{1}{p_1}},$$
where $\cq$ denotes the family of all cubes in $\rn$ with
sides parallel to the coordinate axes. The space $\cb_{p_1'}^{p_2'}(\rn)$ is
defined as the set of all functions $f \in L^{p_1'}_{\loc}(\rn)$
(the set of all locally $p_1'$-order integrable functions on $\rn$) equipped with
$$\|f\|_{\cb_{p_1'}^{p_2'}(\rn)}:= \inf\left\{\|\{\lambda_k\}_k\|_{l^1} : f =\sum_{k}\lambda_kb_k \right\}<\infty,$$
where $\|\{\lambda_k\}_k\|_{l^1} :=\sum_{k}|\lambda_k|$
and $\{b_k\}_k$ is a sequence of $(p_2', p_1')$-blocks, and the infimum
is taken over all possible decompositions of $f$ as above (see, for instance, \cite[p.\,666]{st}).
Then, for any
$s\in(0,\min\{1,r\})$ and $q\in(\max\{1,p\},\infty]$, by \cite[Theorem 4.1]{st}, we can show that
$$[(X^{1/s})']^{1/(q/s)'}={\mathcal B}^{(p/s)'/(q/s)'}_{(r/s)'/(q/s)'}(\rn),$$
which, combined with the fact that $M$ is bounded on
${\mathcal B}^p_r(\rn)$ for any $1<p\le r<\fz$ (see, for instance, \cite[Theorem 3.1]{ch14}),
further implies that \eqref{ma22} holds true in this case.

\item[{\rm(f)}] Let $X:=L^{p(\cdot)}(\rn)$ with $0<p_{-}\le p_+<\fz$ and
$p(\cdot)$ being globally log-H\"older continuous.
Then, for any $s\in(0,\min\{1,p_{-}\})$ and $q\in(\max\{1,p_+\},\infty]$,
let $(p(\cdot)/s)'$ be a variable exponent such that $\frac{1}{(p(\cdot)/s)'}+\frac{1}{p(\cdot)/s}=1$,
we can then show that
$$[(X^{1/s})']^{1/(q/s)'}=L^{(p(\cdot)/s)'/(q/s)'}(\rn),$$
which, together with the fact that $M$ is bounded on $L^{p(\cdot)}(\rn)$
with $p_{-}\in(1,\fz)$ (see, for instance, \cite[Theorem 3.16]{uf}),
implies that \eqref{ma22} holds true in this case.

\item[{\rm(g)}] Assume that $\Phi$ is an \emph{Orlicz function} with \emph{lower type}
$p_\Phi^- \in (0,1)$ and \emph{upper type} $p_\Phi^+ =1$. Let $X:=L^\Phi(\rn)$.
Then $X^{1/s}=L^{\widetilde{\Phi}}(\rn)$, where, for any $t \in [0,\infty)$,
$\widetilde{\Phi}(t):=\Phi(t^{1/s})$. Let $\Psi$ be the \emph{conjugate} of $\widetilde{\Phi}$, which is defined by setting, for any $t\in[0,\infty)$, $\Psi(t):=\sup_{s\in[0,\infty)}[st-\widetilde{\Phi}(s)]$ (see, for instance, \cite[p.\,83, Theorem 13.6(a)]{m83}).
For any $t\in [0,\infty)$, let $\widetilde{\Psi}(t):=\Psi(t^{1/(q/s)'})$. Then we can show that,
for any $s \in (0,p_\Phi^-)$  and $q\in (1, \infty]$,
$$[(X^{1/s})']^{1/(q/s)'}=L^{\widetilde{\Psi}}(\rn).$$
Moreover, by \cite[Proposition 7.8]{shyy},
we conclude that $\widetilde{\Psi}$ is an Orlicz function with the lower type
$p_{\widetilde{\Psi}}^-=(p_\Phi^+/s)'/(q/s)'.$
It is easy to see that $p_{\widetilde{\Psi}}^-\in(1,\fz)$.
From this and the fact that, if $\Theta$ is an Orlicz function
with lower type $p_\Theta^-\in(1,\infty)$, then $M$ is
bounded on the space $L^\Theta(\rn)$
(see, for instance, \cite[Theorem 7.12]{shyy}),
we deduce that \eqref{ma22} holds true for any $s\in(0,p_\Phi^-)$
and $q\in(1,\fz]$.

\item[\rm(h)] Let $t,\ r\in(0,\fz)$, $\Phi$ be an Orlicz function
with $p_\Phi^-,\ p_\Phi^+\in(0,\fz)$ and $X:=(E_\Phi^r)_t(\rn)$.
It was proved in \cite[Lemma 4.4]{zyyw}
that \eqref{ma22} holds true for any $q\in(\max\{1,r, p_\Phi^+\},\fz]$ and
$s\in(0,\min\{1,p_\Phi^-,r\})$.
\end{enumerate}
\end{remark}

To give the definitions of Lusin-area functions, $g$-functions and
$g^\ast_\lambda$-functions considered in this article, we need the following
notions.

Let ${\mathcal F}$ and ${\mathcal F}^{-1}$ be, respectively, the \emph{Fourier transform}
and its \emph{inverse}. Namely,
for any $f\in\cs(\rn)$ and $\xi\in\rn$,
\begin{align*}
{\mathcal F}f(\xi):=(2\pi)^{-\frac{n}{2}}
\int_{{\mathbb R}^n} f(x)e^{-i x \cdot \xi}\,d x \ \ \text{and}\ \
{\mathcal F}^{-1}f(\xi):={\mathcal F}f(-\xi),
\end{align*}
here and thereafter, $x\cdot\xi :=\sum_{i=1}^nx_i\xi_i$ for any $x:=(x_1,\dots,x_n)$, $\xi:=(\xi_1,\dots,\xi_n)\in\rn$. Then $\mathcal{F}$ and ${\mathcal F}^{-1}$ are naturally generalized to $\mathcal{S}'(\rn)$, namely, for any $f\in\mathcal{S}'(\rn)$, $\mathcal{F}f$ is defined by setting, for any $\phi\in\cs(\rn)$,
$\langle\mathcal{F}f,\phi\rangle:=\langle f,\mathcal{F}\phi\rangle$
and ${\mathcal F}^{-1}f$ is defined by setting, for any $\xi\in\rn$,
${\mathcal F}^{-1}f(\xi):={\mathcal F}f(-\xi)$.

\begin{definition}
For any $t\in(0,\infty)$, $f\in\mathcal{S}'(\rn)$ and $\varphi \in \mathcal{S}(\rn)$, let
$$\varphi(tD)(f):=\mathcal{F}^{-1}[\varphi(t\cdot)\mathcal{F}f].$$

A distribution $f\in \mathcal{S}'(\rn)$ is said to
\emph{vanish weakly at infinity}
if, for any $\varphi\in\mathcal{S}(\rn)$,
$$\lim_{t\to\infty}\varphi(tD)(f)=0\quad \text{in} \quad \mathcal{S}'(\rn).$$
\end{definition}

Now we recall the notions of the Lusin-area function, the
Littlewood--Paley $g$-function and
the Littlewood--Paley $g_\lambda^\ast$-function, respectively, as follows.

\begin{definition}
Let $\varphi \in \mathcal{S}(\rn)$. For any $f \in \mathcal{S}'(\rn)$,
the \emph{Lusin-area function $S(f)$}, the \emph{Littlewood--Paley $g$-function $g(f)$}
and the \emph{Littlewood--Paley $g_\lambda^\ast$-function $g_\lambda^\ast(f)$}
with $\lambda \in (0,\infty)$ of $f$ are, respectively,
defined by setting, for any $x\in\rn$,
\begin{equation}
S(f)(x):=\left\{\int_{\Gamma(x)}\lf|\varphi(tD)(f)(y)\r|^2\,
\frac{dydt}{t^{n+1}}\right\}^{\frac{1}{2}},
\end{equation}
where, for any $x\in\rn$, $\Gamma(x):=\{(y,t)\in\mathbb{R}_+^{n+1}:\ |x-y|<t\}$,
\begin{equation}
g(f)(x):=\left\{\int_0^\infty\lf|\varphi(tD)(f)(x)\r|^2\,\frac{dt}{t}\right\}^\frac{1}{2}
\end{equation}
and
\begin{equation}
g_\lambda^\ast(f)(x):=\left\{\int_0^\infty\int_{\rn}\left(\frac{t}{t+|x-y|}\right)^{\lambda n}
\lf|\varphi(tD)(f)(x)\r|^2\,\frac{dydt}{t^{n+1}}\right\}^\frac{1}{2}.
\end{equation}
\end{definition}

Then the main results of this section can be  stated as follows.

\begin{theorem}\label{g}
Let $\varphi\in\mathcal{S}(\rn)$ satisfy
$$\mathbf{1}_{B(\vec{0}_n,4)\setminus B(\vec{0}_n,2)}\leq\varphi\leq
\mathbf{1}_{B(\vec{0}_n,8)\setminus B(\vec{0}_n,1)}.$$
Assume that $X$ is a ball quasi-Banach function space satisfying
\eqref{ma} with some $0<\theta<s\le 1$ and Assumption \ref{a2} for some $q\in(1,\fz]$
and the same $s\in(0,1]$ as in \eqref{ma}.
\begin{itemize}
\item[\rm(i)] If $X$ satisfies \eqref{ma2} with the same $\theta$ and $s$ as in \eqref{ma}, then
$f\in H_X(\rn)$ if and only if $f\in\cs'(\rn)$, $f$ vanishes weakly at infinity and
$\|g(f)\|_X<\infty$. Moreover, for any $f\in H_X(\rn)$,
$\|f\|_{H_X(\rn)}\sim\|g(f)\|_X$ with the positive equivalence
constants independent of $f$.

\item[\rm(ii)] Let $\lambda\in(\max\{2/\theta,2/\theta+(1-2/q)\},\fz)$.
Then $f\in H_X(\rn)$ if and only if $f\in\cs'(\rn)$, $f$ vanishes weakly at infinity and
$\|g_\lambda^\ast(f)\|_X<\infty$. Moreover, for any $f\in H_X(\rn)$,
$\|f\|_{H_X(\rn)}\sim\|g_\lambda^\ast(f)\|_X$ with the positive equivalence
constants independent of $f$.
\end{itemize}
\end{theorem}

To prove Theorem \ref{g}, we need the following Lusin-area function
characterization of $H_X(\rn)$ obtained in \cite[Theorem 3.21]{shyy}.

\begin{lemma}\label{la}
Assume that $X$ is a ball quasi-Banach function space satisfying
\eqref{ma} and Assumption \ref{a2} with the same $s\in (0,1]$.
Let $\varphi\in\mathcal{S}(\rn)$ satisfy
$$\mathbf{1}_{B(\vec{0}_n,4)\setminus B(\vec{0}_n,2)}\leq\varphi\leq
\mathbf{1}_{B(\vec{0}_n,8)\setminus B(\vec{0}_n,1)}.$$
Then $f \in H_X(\rn)$ if and only if $f \in \mathcal{S}'(\rn)$, $f$ vanishes weakly at infinity and
$\left\|S(f)\right\|_X<\infty$. Moreover, for any $f\in H_X(\rn)$,
$$\|f\|_{H_X(\rn)}\sim\left\|S(f)\right\|_X,$$
where the positive equivalence constants are independent of $f$.
\end{lemma}

By an argument similar to that used in the proof of the necessity part of
Lemma \ref{la}, we obtain the following proposition with the details omitted here.

\begin{proposition}\label{b}
Let $X$ be a ball quasi-Banach function space satisfying \eqref{ma} and Assumption \ref{a2}
with the same $s\in (0,1]$. If $f \in H_X(\rn)$, then $f \in \mathcal{S}'(\rn)$,
$f$ vanishes weakly at infinity and $g(f) \in X$. Moreover,
there exists a positive constant $C$ such that, for any $f\in H_X(\rn)$,
$$\|g(f)\|_X\le  C\|f\|_{H_X(\rn)}.$$
\end{proposition}

Moreover, to show Theorem \ref{g}(i), we need the following maximal function
of Peetre type, which is motivated by Ullrich \cite{u} and Yan et al. \cite{yyyz}.

\begin{definition}
Let $\phi\in\mathcal{S}(\rn)$.
For any $t,\ a\in(0,\infty)$, $f \in \mathcal{S}'(\rn)$ and $x\in\rn$, define
\begin{equation}\label{p}
(\phi_t^\ast f)_a(x):=\sup_{y\in\rn}\frac{|\phi(tD)(f)(x+y)|}{(1+|y|/t)^a}
\end{equation}
and
$$g_{a,\ast}(f)(x):=\left\{\int_0^\infty\lf[(\phi_t^\ast f)_a(x)\r]^2\,\frac{dt}{t}\right\}^{\frac{1}{2}}.$$
\end{definition}

The following estimate, obtained in \cite[Lemma 3.5]{lsuyy}, is needed in the proof of Theorem \ref{g}.

\begin{lemma}\label{ep}
Let $\varepsilon \in (0,\infty)$, $R,\ N_0\in\mathbb{N}$ and $\phi\in\mathcal{S}(\rn)$
satisfy $|\mathcal{F}(\phi)(\xi)| >0$ on $\{\xi \in \rn: \varepsilon/2<|\xi|<2\varepsilon\}$ and
$D^\alpha(\mathcal{F}(\phi))({\vec 0}_n)=0$ for any $\alpha\in\zz_+^n$ and $|\alpha|\le  R$.
Then, for any $t\in[1,2]$, $a \in (0,N_0]$, $l\in\zz$, $f \in \mathcal{S}'(\rn)$ and $x\in\rn$,
there exists a positive constant $C_{(N_0,r)}$, depending only on $N_0$ and $r$, such that
\begin{equation*}
\sup_{y\in\rn}\frac{|\phi_{2^{-l}t}\ast f(x+y)|^r}{(1+|y|/(2^{-l}t))^{ar}}
\le  C_{(N_0,r)}\sum_{k=0}^\infty2^{-kN_0r}2^{(k+l)n}\int_{\rn}
\frac{|\phi_{2^{-(k+l)}t}\ast f(y)|^r}{(1+2^l|x-y|/t)^{ar}}\,dy.
\end{equation*}
\end{lemma}

Furthermore, to deal with the $g_\lambda^\ast$-function characterization of $H_X(\rn)$,
we use a method different from that used in the proof of the $g$-function characterization of $H_X(\rn)$,
which is motivated by \cite{a11}.
We first recall some notions of \emph{X-tent spaces}.

\begin{definition}
Let $\alpha \in (0,\infty)$. For any $x\in\rn$,
the \emph{cone} $\Gamma_\alpha(x)$ of aperture $\alpha$ with vertex $x$
is defined by setting
$$\Gamma_\alpha(x):=\lf\{(y,t)\in\mathbb{R}^{n+1}_+:\ |x-y|<\alpha t\r\}.$$
Moreover, for any ball $B(x,r)\subset\rn$ with $x\in\rn$ and $r\in(0,\fz)$, let
$$T_\alpha(B):=\lf\{(y,t)\in\mathbb{R}^{n+1}_+:\ 0<t<r/\alpha,\ |y-x|<r-\alpha t\r\}.$$
When $\alpha:=1$, we denote $\Gamma_\alpha(x)$ and $T_\alpha(B)$ simply,
respectively, by $\Gamma(x)$ and $T(B)$.
\end{definition}

Let $\alpha,\ p\in(0,\fz)$ and
$g:\ \mathbb{R}^{n+1}_+\to \mathbb{C}$ be a measurable function.
Then the \emph{Lusin-area function} $\ca_\alpha(g)$, with aperture $\alpha$,
is defined by setting, for any $x\in\rn$, $$\ca_\alpha(g)(x):=\left\{\int_{\Gamma_\alpha(x)}|g(y,t)|^2\,
\frac{dydt}{t^{n+1}}\right\}^{\frac{1}{2}}.$$
A measurable function $g$ is said to belong to the \emph{tent space}
$T_2^{p,\alpha}(\mathbb{R}^{n+1}_+)$  if
$$\|g\|_{T_2^{p,\alpha}(\mathbb{R}^{n+1}_+)}:=\lf\|\ca_\alpha(g)\r\|_{\lp}<\infty.$$
Recall that Coifman et al. \cite{cms} introduced the tent space $T_2^{p,\alpha}(\mathbb{R}^{n+1}_+)$
for any $p\in (0,\fz)$ and $\alpha:=1$.
For any given ball quasi-Banach function space $X$, the \emph{$X$-tent space} $T_X^\alpha(\mathbb{R}^{n+1}_+)$,
with aperture $\alpha$,  is defined to be the set of all measurable functions $g:\ \mathbb{R}^{n+1}_+\to\mathbb{C}$ such that
$\|g\|_{T_X^\alpha(\mathbb{R}^{n+1}_+)}:=\|\ca_\alpha(g)\|_X <\infty$.

\begin{definition}
Let $p\in(1,\infty)$ and $\alpha\in(0,\infty)$. A measurable function $a:\ \mathbb{R}^{n+1}_+\to\mathbb{C}$
is called a \emph{$(T_X^\alpha,p)$-atom}
if there exists a ball $B\subset\rn$ such that
\begin{enumerate}
\item[{\rm{(i)}}] $\supp(a):=\{(x,t)\in\rr^{n+1}_+:\ a(x,t)\neq0\}\subset T_\alpha(B)$;
\item[{\rm{(ii)}}] $\|a\|_{T_2^{p,\alpha}(\mathbb{R}^{n+1}_+)}
\leq|B|^{1/p}\|\mathbf{1}_B\|_X^{-1}$.
\end{enumerate}
Moreover, if $a$ is a $(T_X^\alpha,p)$-atom for any $p\in(1,\infty)$,
then $a$ is called a \emph{$(T_X^\alpha,\infty)$-atom}.
\end{definition}

The following lemma is a direct conclusion of the definition of $(T_X^\alpha,p)$-atoms;
the details are omitted here.

\begin{lemma}\label{8.13}
Let $p\in (1,\infty)$ and $\alpha \in (0,\infty)$. Then,
for any $(T_X^\alpha,p)$-atom $a$ supported in $T_\alpha(B)$, $\ca_\alpha(a)$ is
supported in $B$ and $\|\ca_\alpha(a)\|_{\lp}\leq|B|^{1/p}\|\mathbf{1}_B\|_X^{-1}$.
\end{lemma}

\begin{definition}
Assume that $X$ is a ball quasi-Banach function space. Let $s\in(0,1]$,
$p\in(1,\fz)$, $\alpha\in[1,\infty)$, $\{\lambda_j\}_{j\in\nn}\subset [0,\infty)$
be a sequence and $\{a_j\}_{j\in\nn}$
a sequence of $(T_X^\alpha,p)$-atoms supported, respectively,  in
$\{T_\alpha(B_j)\}_{j\in\nn}$. Then define
$$\Lambda\lf(\{\lambda_j a_j\}_{j\in\nn}\r):=\left\|\left[\sum_{j=1}^\infty
\left(\frac{\lambda_j}{\|\mathbf{1}_{B_j}\|_X}\right)^s
\mathbf{1}_{B_j}\right]^{1/s}\right\|_X.$$
\end{definition}

For the $X$-tent space $T^1_X(\mathbb{R}^{n+1}_+)$,
we have the following atomic characterization, which was obtained in
\cite[Theorem 3.19]{shyy}.

\begin{lemma}\label{tad}
Let $X$ be a ball quasi-Banach function space and
$f:\ \mathbb{R}^{n+1}_+\to\mathbb{C}$ a measurable function.
Assume further that $X$ satisfies \eqref{ma} and Assumption \ref{a2} with the same $s\in (0,1]$. Then
$f\in T^1_X(\mathbb{R}^{n+1}_+)$ if and only if there exist
a sequence $\{\lambda_j\}_{j\in\nn}\subset[0,\infty)$ and a sequence $\{a_j\}_{j\in\nn}$
of $(T_X^1,\infty)$-atoms supported, respectively,  in $\{T(B_j)\}_{j\in\nn}$
such that, for almost every $(x,t)\in\mathbb{R}^{n+1}_+$,
$$f(x,t)=\sum_{j=1}^\infty\lambda_j a_j(x,t)\quad\text{and}\quad
|f(x,t)|=\sum_{j=1}^\infty\lambda_j |a_j(x,t)|.$$
Moreover,
$$\|f\|_{T^1_X(\mathbb{R}^{n+1}_+)}\sim\Lambda\lf(\{\lambda_j a_j\}_{j\in\nn}\r),$$
where the positive equivalence constants are independent of $f$,
but may depend on $s$.
\end{lemma}

Using Lemma \ref{tad}, we obtain the following lemma,
which is an expansion of the aperture estimate for the classical
tent space $T^{p,\alpha}_2(\mathbb{R}^{n+1}_+)$,
obtained in \cite[(4)]{a11}, to the case of the $X$-tent space.

\begin{lemma}\label{af}
Let $\alpha \in [1, \infty)$ and $X$ be a ball quasi-Banach
function space satisfying
\eqref{ma} with some $0<\theta<s\le 1$ and Assumption \ref{a2} for some $q\in(1,\fz]$
and the same $s\in(0,1]$ as in \eqref{ma}.
Assume that $f:\ \mathbb{R}^{n+1}_+\to\mathbb{C}$ is a measurable function.
Then there exists a positive constant $C$, independent of $\alpha$ and $f$, such that
$$\lf\|\ca_\alpha(f)\r\|_{X}\le C\max\lf\{\alpha^{(\frac{1}{2}-\frac1q)n},1\r\}
\alpha^{\frac n\theta}\lf\|\ca_1(f)\r\|_{X}.$$
\end{lemma}

Moreover, to prove Lemma \ref{af}, we need the following conclusion,
which was obtained in \cite[Theorem 2.10]{shyy}.

\begin{lemma}\label{r}
Assume that $X$ is a ball quasi-Banach function space satisfying
\eqref{ma} with some $s\in (0, 1]$ and Assumption \ref{a2} for some $q\in(1,\fz]$
and the same $s\in(0,1]$ as in \eqref{ma}. Let $\{a_j\}_{j=1}^\infty\subset L^q(\rn)$
be supported, respectively,  in cubes $\{Q_j\}^\infty_{j=1}$,
and a sequence $\{\lambda_j\}^\infty_{j=1}\subset[0,\fz)$ such that,
for any $j \in \nn$,
$$\|a_j\|_{L^q(\rn)}\leq\frac{|Q_j|^{1/q}}{\|\mathbf{1}_{Q_j}\|_X}$$
and
$$\left\|\left[\sum_{j=1}^\infty\left(\frac{\lambda_j}{\|\mathbf{1}_{Q_j}\|_X}\right)^s
\mathbf{1}_{Q_j}\right]^{1/s}\right\|_X<\infty.$$
Then $f=\sum_{j=1}^\infty \lambda_ja_j$ converges in $\mathcal{S}'(\rn)$ and
there exists a positive constant $C$, independent of $f$, such that
$$\|f\|_X\le C\left\|\left[\sum_{j=1}^\infty\left(\frac{\lambda_j}
{\|\mathbf{1}_{Q_j}\|_X}\right)^s\mathbf{1}_{Q_j}\right]^{1/s}\right\|_X.$$
\end{lemma}

Now we prove Lemma \ref{af} by using Lemmas \ref{tad} and \ref{r}.

\begin{proof}[Proof of Lemma \ref{af}]
Without loss of generality, we may assume that $f \in T^1_X(\mathbb{R}^{n+1}_+)$.
Then $\ca_1(f)\in X$. By Lemma \ref{tad}, we know that there exist
a sequence $\{\lambda_j\}_{j\in\nn}\subset[0,\infty)$ and a sequence $\{a_j\}_{j\in\nn}$
of $(T_X^1,\infty)$-atoms supported, respectively,  in $\{T(B_j)\}_{j\in\nn}$
such that, for almost every $(x,t)\in\mathbb{R}^{n+1}_+$,
$$f(x,t)=\sum_{j=1}^\infty\lambda_j a_j(x,t)\quad\text{and}\quad\|f\|_{T^1_X(\mathbb{R}^{n+1}_+)}\sim
\Lambda\lf(\{\lambda_j a_j\}_{j\in\nn}\r).$$

Let $\alpha\in[1,\infty)$ and $a$ be a $(T_X^1,\infty)$-atom supported in $T(B)$.
Then, from \cite[Theorem 1.1]{a11}, we deduce that, for any $p\in(1,\fz)$,
there exists a positive constant $\wz{C}$, independent of $a$ and $\alpha$,
such that
$$\lf\|\ca_\alpha(a)\r\|_{L^p(\rn)}\le\wz{C}\max\lf\{\alpha^{\frac n2},
\alpha^{\frac np}\r\}\lf\|\ca_1(a)\r\|_{L^p(\rn)},
$$
which further implies that, for any $p\in(1,\fz)$,
\begin{equation}\label{ae20}
\lf\|\ca_\alpha(a)\r\|_{L^p(\rn)}\le\wz{C}\max\lf\{\alpha^{\frac n2-\frac np},
1\r\}\frac{\|\mathbf{1}_{\alpha B}\|_X}{\|\mathbf{1}_{B}\|_X}|\alpha B|^{\frac1p}
\|\mathbf{1}_{\alpha B}\|_X^{-1}.
\end{equation}
Moreover, it is easy to see that $T(B)\subset T_\alpha(\alpha B)$ and
hence $\supp(a)\subset T_\alpha(\alpha B)$,
which, combined with \eqref{ae20}, implies that, for any $p\in(1,\fz)$,
$\|\mathbf{1}_{B}\|_X\|\mathbf{1}_{\alpha B}\|_X^{-1}[\wz{C}\max\{\alpha^{\frac n2-\frac np},
1\}]^{-1} a$ is a $(T_X^\alpha,p)$-atom supported in $T_\alpha(\alpha B)$.
By this and Lemma \ref{8.13}, we find that, for any $j\in\nn$,
$\ca_\alpha(a_j)$ is supported in $\alpha B_j$ and
\begin{equation}\label{ae1}
\lf\|\ca_\alpha\lf(\frac{1}{\wz{C}\max\{\alpha^{\frac n2-\frac nq},
1\}}\frac{\|\mathbf{1}_{B_j}\|_X}{\|\mathbf{1}_{\alpha B_j}\|_X}a_j\r)\r\|_{L^q(\rn)}
\le |\alpha B_j|^{\frac1q}
\|\mathbf{1}_{\alpha B_j}\|_X^{-1}.
\end{equation}
Furthermore, from the definition of the Hardy--Littlewood maximal function,
it follows that, for any $j\in\nn$,
$$\mathbf{1}_{\alpha B_j}\ls\alpha^{\frac n\theta} M^{(\theta)}(\mathbf{1}_{B_j}),
$$
which, together with \eqref{ae1}, \eqref{ma} and Lemma \ref{r}, further implies that
\begin{align*}
\|\ca_\alpha(f)\|_X&\le\left\|\sum_{j=1}^\infty\lambda_j\ca_\alpha(a_j)\right\|_X\\
&=\wz{C}\max\lf\{\alpha^{\frac n2-\frac nq},1\r\}\left\|\sum_{j=1}^\infty\lambda_j
\frac{\|\mathbf{1}_{\alpha B_j}\|_X}{\|\mathbf{1}_{B_j}\|_X}\ca_\alpha\lf(\frac{1}{\wz{C}\max\{\alpha^{\frac n2-\frac nq},
1\}}\frac{\|\mathbf{1}_{B_j}\|_X}{\|\mathbf{1}_{\alpha B_j}\|_X}a_j\r)\right\|_X\\
&\lesssim\max\lf\{\alpha^{\frac n2-\frac nq},1\r\}\left\|\left[\sum_{j=1}^\infty
\left(\frac{\lambda_j\|\mathbf{1}_{\alpha B_j}\|_X/\|\mathbf{1}_{B_j}\|_X}{\|\mathbf{1}_{\alpha B_j}\|_X}\right)^s
\mathbf{1}_{\alpha B_j}\right]^{1/s}\right\|_X\\
&\sim\max\lf\{\alpha^{\frac n2-\frac nq},1\r\}\left\|\left[\sum_{j=1}^\infty\left(\frac{\lambda_j}
{\|\mathbf{1}_{B_j}\|_X}\right)^s\mathbf{1}_{\alpha B_j}\right]^{1/s}\right\|_X\\
&\lesssim\max\lf\{\alpha^{\frac n2-\frac nq},1\r\}\left\|\left\{\sum_{j=1}^\infty\left(\frac{\lambda_j}{\|\mathbf{1}_{B_j}\|_X}\right)^s
\left[\alpha^{\frac n\theta}M^{(\theta)}(\mathbf{1}_{B_j})\right]^s\right\}^{1/s}\right\|_X\\
&\lesssim\max\lf\{\alpha^{\frac n2-\frac nq},1\r\}\alpha^{\frac n\theta}\left\|\left[\sum_{j=1}^\infty\left(\frac{\lambda_j}
{\|\mathbf{1}_{B_j}\|_X}\right)^s\mathbf{1}_{B_j}\right]^{1/s}\right\|_X
\sim\max\lf\{\alpha^{(\frac12-\frac1q)n},1\r\}\alpha^{\frac n\theta} \|\ca_1(f)\|_X.
\end{align*}
This finishes the proof of Lemma \ref{af}.
\end{proof}

We now prove Theorem \ref{g} by using Proposition \ref{b}
and  Lemmas \ref{ep} and \ref{af}.
\begin{proof}[Proof of Theorem \ref{g}]
We first show (i), namely, $f\in H_X(\rn)$
if and only if $f\in\cs'(\rn)$, $f$ vanishes weakly at infinity and
$\lf\|g(f)\r\|_X<\infty$. Assume that $f\in H_X(\rn)$. Then, by Definition \ref{H}, we know that $f\in\cs'(\rn)$ and, from Proposition \ref{b},
it follows that $f$ vanishes weakly at infinity and
$\|g(f)\|_X\lesssim\|f\|_{H_X(\rn)} <\infty$.

Conversely, by Lemma \ref{la}, we just need to show that,
for any $f \in \mathcal{S}'(\rn)$ vanishing weakly at infinity and $\|g(f)\|_X<\infty$,
it holds true that
\begin{equation}\label{nn}
\left\|\left\{\int_{\Gamma(\cdot)}\lf|\varphi(tD)(f)(y)\r|^2\,
\frac{dydt}{t^{n+1}}\right\}^{\frac{1}{2}}\right\|_X\lesssim\|g(f)\|_X.
\end{equation}
To this end, it is easy to see that, for any $a\in (0,\infty)$ and almost every $x\in\rn$,
\begin{equation}\label{np}
\left\{\int_{\Gamma(x)}\lf|\varphi(tD)(f)(y)\r|^2\,
\frac{dydt}{t^{n+1}}\right\}^{\frac{1}{2}}\lesssim g_{a,\ast}(f)(x).
\end{equation}
Therefore, to show \eqref{nn}, it suffices to prove that, for some $a\in (0,\infty)$,
\begin{equation}\label{gg}
\lf\|g_{a,\ast}(f)\r\|_X\lesssim\|g(f)\|_X.
\end{equation}

Observe that $\varphi(tD)(f)=(\mathcal{F}^{-1}\varphi)_t\ast f$.
Thus, by Lemma \ref{ep} and the Minkowski inequality, we conclude that,
for any given $r\in(0,2)$, $N_0\in\nn$ and $a\in(0,N_0]$, which are determined later,
and for any $x\in\rn$,
\begin{align}\label{gaa}
&[g_{a,\ast}(f)(x)]^r\notag\\
&\hs=\left\{\sum_{j\in\zz}\int_1^2[(\varphi_{2^{-j}t}^\ast f)_a(x)]^2\,
\frac{dt}{t}\right\}^{\frac{r}{2}}\notag\\
&\hs\lesssim\left\{\sum_{j\in\zz}\int_1^2\left[\sum_{k=0}^\infty2^{-kN_0r}
2^{(k+j)n}\int_{\rn}\frac{|(\mathcal{F}^{-1}\varphi)_{2^{-(k+j)}t}\ast
f(y)|^r}{(1+2^j|x-y|/t)^{ar}}\,dy\right]^{\frac{2}{r}}\,\frac{dt}{t}
\right\}^{\frac{r}{2}}\notag\\
&\hs\lesssim\left[\sum_{j\in\zz}\left\{\sum_{k=0}^\infty2^{-kN_0r}2^{(k+j)n}\int_{\rn}
\frac{[\int_1^2|(\mathcal{F}^{-1}\varphi)_{2^{-(k+j)}t}\ast f(y)|^2
\frac{dt}{t}]^{\frac{r}{2}}}{(1+2^j|x-y|)^{ar}}\,dy\right\}^{\frac{2}{r}}
\right]^{\frac{r}{2}}\notag\\
&\hs\lesssim \sum_{k=0}^\infty2^{-k(N_0r-n)}\left[\sum_{j\in\zz}
2^{j\frac{2n}{r}}\left\{\int_{\rn}\frac{[\int_1^2|
(\mathcal{F}^{-1}\varphi)_{2^{-(k+j)}t}\ast f(y)|^2
\frac{dt}{t}]^{\frac{r}{2}}}{(1+2^j|x-y|)^{ar}}\,dy\right\}^{\frac{2}{r}}\right]^{\frac{r}{2}}\notag\\
&\hs\lesssim \sum_{k=0}^\infty2^{-k(N_0r-n)}\left[\sum_{j\in\zz}
2^{j\frac{2n}{r}}\left\{\sum_{i=0}^\infty2^{-iar}\int_{|x-y|\sim2^{i-j}}\left[\int_1^2|
(\mathcal{F}^{-1}\varphi)_{2^{-(k+j)}t}\ast f(y)|^2
\frac{dt}{t}\right]^{\frac{r}{2}}\,dy\right\}^{\frac{2}{r}}\right]^{\frac{r}{2}},
\end{align}
where $|x-y|\sim2^{i-j}$ means that $|x-y|<2^{-j}$ when $i=0$ or $2^{i-j-1}<|x-y|<2^{i-j}$
when $i\in\nn$. From \eqref{gaa} and the Minkowski inequality
again, we deduce that, for any given  $r\in(0,2)$ and any $x\in\rn$,
\begin{align}\label{gaa1}
&[g_{a,\ast}(f)(x)]^r\notag\\
&\hs\lesssim\sum_{k=0}^\infty2^{-k(N_0r-n)}\sum_{i=0}^\infty\left(\sum_{j\in\zz}
\left\{2^{-iar}2^{jn}\int_{|x-y|\sim2^{i-j}}\left[\int_1^2|
(\mathcal{F}^{-1}\varphi)_{2^{-(k+j)}t}\ast f(y)|^2
\frac{dt}{t}\right]^{\frac{r}{2}}\,dy\right\}^{\frac{2}{r}}\right)^{\frac{r}{2}}\notag\\
&\hs\lesssim\sum_{k=0}^\infty2^{-k(N_0r-n)}\sum_{i=0}^\infty2^{-iar+in}
\left\{\sum_{j\in\zz}\left[M\left(\left[\int_1^2|
(\mathcal{F}^{-1}\varphi)_{2^{-(k+j)}t}\ast f(y)|^2\frac{dt}{t}
\right]^{\frac{r}{2}}\right)\right]^{\frac{2}{r}}\right\}^{\frac{r}{2}}.
\end{align}
Let $r:=\frac{2\theta}{s}$, where $\theta$ and $s$ are as
in Assumptions \ref{a} and \ref{a2}.
Choose $N_0\in\nn$ and $a\in(0,N_0]$ large enough such that $N_0r-n>0$ and
$ar-n>0$. Then, by \eqref{gaa1}, we find that, for any $x\in\rn$,
\begin{align*}
g_{a,\ast}(f)(x)&\lesssim\left\{\sum_{j\in\zz}\left[M\left
(\left[\int_1^2|(\mathcal{F}^{-1}\varphi)_{2^{-j}t}\ast f(y)|^2\frac{dt}{t}\right]^{\frac{\theta}{s}}\right)\right]^{\frac{s}{\theta}}\right\}^{\frac{1}{2}}\\
&\lesssim\left\{\sum_{j\in\zz}\left[M^{(\theta)}\left
(\left[\int_1^2|(\mathcal{F}^{-1}\varphi)_{2^{-j}t}\ast f(y)|^2
\frac{dt}{t}\right]^{\frac{1}{s}}\right)\right]^s\right\}^{\frac{1}{2}}.
\end{align*}
From this and Assumption \ref{a}, it follows that
\begin{align*}
\lf\|[g_{a,\ast}(f)]^{\frac{2}{s}}\r\|_{X^{\frac{s}{2}}}&\lesssim\left\|
\left\{\sum_{j\in\zz}\left[M^{(\theta)}\left(\left[\int_1^2|
(\mathcal{F}^{-1}\varphi)_{2^{-j}t}\ast f(y)|^2\frac{dt}{t}\right]^{\frac{1}{s}}\right)\right]^s\right\}^{\frac{1}{s}}\right\|_{X^{\frac{s}{2}}}\\
&\lesssim \left\|\left\{\sum_{j\in\zz}\int_1^2|(\mathcal{F}^{-1}
\varphi)_{2^{-j}t}\ast f(y)|^2\frac{dt}{t}\right\}^{\frac{1}{s}}\right\|_{X^{\frac{s}{2}}}
\lesssim\lf\|[g(f)]^{\frac{2}{s}}\r\|_{X^{\frac{s}{2}}},
\end{align*}
which implies that \eqref{gg} holds true. This finishes the proof of (i).

We now prove (ii), namely, $f\in H_X(\rn)$ if and only
if $f\in\cs'(\rn)$, $f$ vanishes weakly at infinity and
$\|g_\lambda^\ast(f)\|_X<\infty$ for some $\lambda\in(\max\{2/\theta,
2/\theta+(1-2/q)\},\fz)$.
Assume that $f\in\mathcal{S}'(\rn)$ vanishes weakly at infinity and
$\|g_\lambda^\ast(f)\|_X<\infty$.
By Lemma \ref{la} and the fact that, for any $f\in \mathcal{S}'(\rn)$ and $x\in \rn$,
$S(f)(x)\ls g_\lambda^\ast(f)(x)$ (see, for instance, \cite[p.\,89, (19)]{s70}),
we  conclude that $f\in H_X(\rn)$ and $\|f\|_{H_X(\rn)}\ls\|g_\lambda^\ast(f)\|_X$.

Conversely, assume that $f \in H_X(\rn)$. Then, by Definition \ref{H},
we know that $f\in\cs'(\rn)$ and, for any $x \in \rn$, we have
\begin{align*}
g_\lambda^\ast(f)(x)&=\left\{\int_0^\infty\int_{\rn}\left(\frac{t}
{t+|x-y|}\right)^{\lambda n}\lf|\varphi(tD)(f)(x)\r|^2\,\frac{dydt}
{t^{n+1}}\right\}^\frac{1}{2}\\
&=\left\{\int_0^\infty\int_{|x-y|<t}\left(\frac{t}{t+|x-y|}
\right)^{\lambda n}\lf|\varphi(tD)(f)(x)\r|^2\,\frac{dydt}{t^{n+1}}\right.\\
&\hs+\sum_{m=0}^\infty\left.\int_0^\infty\int_{2^mt<|x-y|<2^{m+1}t}\left(\frac{t}{t+|x-y|}\right)^{\lambda n}\lf|\varphi(tD)(f)(x)\r|^2\,\frac{dydt}{t^{n+1}}\right\}^\frac{1}{2}\\
&\leq\left\{\lf[\ca_1\lf(\varphi(tD)(f)\r)(x)\r]^2+
\sum_{m=0}^\infty2^{-\lambda nm}\left[\ca_{2^{m+1}}\lf(\varphi(tD)(f)\r)(x)\right]^2
\right\}^{\frac{1}{2}}\\
&\lesssim \ca_1\lf(\varphi(tD)(f)\r)(x)+\sum_{m=0}^\infty
2^{\frac{-\lambda nm}{2}}\ca_{2^{m+1}}\lf(\varphi(tD)(f)\r)(x),
\end{align*}
which, combined with $\lambda\in(\max\{2/\theta,2/\theta+(1-2/q)\},\fz)$,
Lemmas \ref{la} and \ref{af}, and Assumption \ref{a2},
further implies that
\begin{align*}
\lf\|g_\lambda^\ast(f)\r\|_X^s &= \lf\|\lf[g_\lambda^\ast(f)\r]^s\r\|_{X^{1/s}}
\lesssim\left\|\left\{\ca_1(\varphi(tD)(f))+\sum_{m=0}^\infty
2^{\frac{-\lambda nm}{2}}\ca_{2^{m+1}}(\varphi(tD)(f))\right\}^s\right\|_{X^{1/s}}\\
&\lesssim \left\|\lf\{\ca_1(\varphi(tD)(f))\r\}^s\right\|_{X^{1/s}}
+\sum_{m=0}^\infty2^{\frac{-\lambda nms}{2}}\left\|\lf\{\ca_{2^{m+1}}
(\varphi(tD)(f))\r\}^s\right\|_{X^{1/s}}\\
&\lesssim\lf\|\ca_1(\varphi(tD)(f))\r\|_X^s+\sum_{m=0}^\infty
2^{\frac{-\lambda nms}{2}}\left\|\ca_{2^{m+1}}(\varphi(tD)(f))\right\|_X^s\\
&\lesssim\lf\|\ca_1(\varphi(tD)(f))\r\|_X^s+\sum_{m=0}^\infty
2^{\frac{-\lambda nms}{2}}\max\lf\{2^{(m+1)(\frac12-\frac1q)ns},1\r\}
2^{\frac{ns}{\theta}(m+1)}\lf\|\ca_1(\varphi(tD)(f))\r\|_X^s\\
&\lesssim\|S(f)\|_X^s\lesssim\|f\|_{H_X(\rn)}^s.
\end{align*}
Thus, $\|g_\lambda^\ast(f)\|_X\ls\|f\|_{H_X(\rn)}$,
which completes the proof of (ii) and hence of Theorem \ref{g}.
\end{proof}

In what follows, we denote the classical \emph{Hardy space},
the \emph{weighted Hardy space}, the \emph{Herz--Hardy space},
the \emph{Lorentz--Hardy space}, the \emph{Morrey--Hardy space},
the \emph{variable Hardy space},
the \emph{Orlicz--Hardy space} and the \emph{Orlicz-slice Hardy space}, respectively, by
$$H^p(\rn),\quad H^p_w(\rn),\quad H\ck^\alpha_{p,q}(\rn),\quad H^{p,q}(\rn),\quad
H\cm^p_q(\rn),\quad H^{p(\cdot)}(\rn),\quad H^\Phi(\rn)$$
and $(HE_\Phi^r)_t(\rn)$.
Then, by Remarks \ref{r2.1} and \ref{r2.2} and
Theorem \ref{g}, we obtain the following conclusions on these function spaces.

\begin{corollary}\label{c2.1}
\begin{itemize}
\item[\rm(i)] Let $p\in(0,1]$. Then $f\in H^p(\rn)$ if and only if $f\in\cs'(\rn)$, $f$ vanishes
weakly at infinity and $\|g(f)\|_{L^p(\rn)}<\infty$ or
$\|g_\lambda^\ast(f)\|_{L^p(\rn)}<\infty$
for some $\lambda\in(2/p,\fz)$. Moreover, for any $f\in H^p(\rn)$,
$$\|f\|_{H^p(\rn)}\sim
\|g(f)\|_{L^p(\rn)}\sim\lf\|g_\lambda^\ast(f)\r\|_{L^p(\rn)},$$
where the positive equivalence constants are independent of $f$.
  \item[\rm(ii)] Let $p\in(0,1]$ and $w\in A_\fz(\rn)$.
Then $f\in H^p_w(\rn)$ if and only if $f\in\cs'(\rn)$, $f$ vanishes
weakly at infinity and $\|g(f)\|_{L^p_w(\rn)}<\infty$ or
$\|g_\lambda^\ast(f)\|_{L^p_w(\rn)}<\infty$
for some $\lambda\in(2/\theta+1,\fz)$, where $\theta$ is as in
Remarks \ref{r2.1}(b) and \ref{r2.2}(b).
Moreover, for any $f\in H^p_w(\rn)$,
$$\|f\|_{H^p_w(\rn)}\sim
\|g(f)\|_{L^p_w(\rn)}\sim\lf\|g_\lambda^\ast(f)\r\|_{L^p_w(\rn)},$$
where the positive equivalence constants are independent of $f$.

\item[\rm(iii)] Let $p\in(0,1]$, $q\in(0,\fz)$ and $\alpha\in(-n/p,\infty)$.
Then $f\in H\ck^\alpha_{p,q}(\rn)$ if and only if $f\in\cs'(\rn)$, $f$ vanishes weakly at infinity
and $\|g(f)\|_{K^\alpha_{p,q}(\rn)}<\infty$ or
$\|g_\lambda^\ast(f)\|_{K^\alpha_{p,q}(\rn)}<\infty$
for some $\lambda \in (2/\min\{p,[\alpha/n+1/p]^{-1}\},\infty)$.
Moreover, for any $f\in H\ck^\alpha_{p,q}(\rn)$,
$$\|f\|_{H\ck^\alpha_{p,q}(\rn)}\sim\|g(f)\|_{K^\alpha_{p,q}(\rn)}
\sim\lf\|g_\lambda^\ast(f)\r\|_{K^\alpha_{p,q}(\rn)},$$
where the positive equivalence constants are independent of $f$.

\item[\rm(iv)] Let $p\in (0,1]$ and $q\in(0,\fz)$.
Then $f\in H^{p,q}(\rn)$ if and only if $f\in\cs'(\rn)$,
$f$ vanishes weakly at infinity and $\|g(f)\|_{L^{p,q}(\rn)}<\infty$ or
$\|g_\lambda^\ast(f)\|_{L^{p,q}(\rn)}<\infty$
for some $\lambda \in (2/\min\{p,q\}+(1-2/\max\{1,q\}),\infty)$.
Moreover, for any $f\in H^{p,q}(\rn)$,
$$\|f\|_{H^{p,q}(\rn)}\sim\|g(f)\|_{L^{p,q}(\rn)}
\sim\lf\|g_\lambda^\ast(f)\r\|_{L^{p,q}(\rn)},$$
where the positive equivalence constants are independent of $f$.

\item[{\rm(v)}] Let $p\in(0,1]$ and $q\in(0,p]$.
Then $f\in H\cm^p_q(\rn)$ if and only if $f\in\cs'(\rn)$, $f$ vanishes weakly at infinity and
$\|g(f)\|_{\cm^p_q(\rn)}<\infty$ or
$\|g_\lambda^\ast(f)\|_{\cm^p_q(\rn)}<\infty$ for some $\lambda\in
(2/q,\infty)$.
Moreover, for any $f\in H\cm^p_q(\rn)$,
$$\|f\|_{H\cm^p_q(\rn)}\sim\|g(f)\|_{\cm^p_q(\rn)}\sim
\|g_\lambda^\ast(f)\|_{\cm^p_q(\rn)},$$
where the positive equivalence constants are independent of $f$.

\item[{\rm(vi)}] Let $p(\cdot)$ be globally log-H\"older continuous with $0<p_-\le p_+<\fz$,
where $p_-$ and $p_+$ are as in \eqref{eq-p}.
Then $f\in H^{p(\cdot)}(\rn)$ if and only if $f\in\cs'(\rn)$, $f$ vanishes weakly at infinity and
$\|g(f)\|_{L^{p(\cdot)}(\rn)}<\infty$ or
$\|g_\lambda^\ast(f)\|_{L^{p(\cdot)}(\rn)}<\infty$
for some $\lambda\in(2/\min\{1,p_-\}+(1-2/\max\{1,p_+\}),\infty)$.
Moreover, for any $f\in H^{p(\cdot)}(\rn)$,
$$\|f\|_{H^{p(\cdot)}(\rn)}\sim\|g(f)\|_{L^{p(\cdot)}(\rn)}
\sim\lf\|g_\lambda^\ast(f)\r\|_{L^{p(\cdot)}(\rn)},$$
where the positive equivalence constants are independent of $f$.

\item[{\rm(vii)}] Let $\Phi$ be an Orlicz function with lower type $p_\Phi^-\in(0,1)$
and upper type $p_\Phi^+=1$. Then $f\in H^\Phi(\rn)$ if and only if $f\in\cs'(\rn)$, $f$ vanishes weakly
at infinity and $\|g(f)\|_{L^\Phi(\rn)}<\infty$ or
$\|g_\lambda^\ast(f)\|_{L^\Phi(\rn)}<\infty$ for some $\lambda \in (2/p_\Phi^-,\infty)$.
Moreover, for any $f\in H^\Phi(\rn)$,
$$\|f\|_{H^\Phi(\rn)}\sim\|g(f)\|_{L^\Phi(\rn)}\sim\lf\|g_\lambda^\ast(f)\r\|_{L^\Phi(\rn)},$$
where the positive equivalence constants are independent of $f$.

\item[{\rm(viii)}] Let $t,\ r\in(0,\fz)$ and $\Phi$ be an Orlicz function with both
lower type $p_\Phi^-$ and upper type $p_\Phi^+$ belonging to $(0,\fz)$.
Then $f\in(HE_\Phi^r)_t(\rn)$ if and only if $f\in\cs'(\rn)$, $f$ vanishes weakly
at infinity and
$$\|g(f)\|_{(E_\Phi^r)_t(\rn)}<\infty\ \ \ \text{or}\ \ \
\lf\|g_\lambda^\ast(f)\r\|_{(E_\Phi^r)_t(\rn)}<\infty$$
for some $\lambda \in (2/\min\{1,p_\Phi^-,r\}+(1-2/\max\{1,p_\Phi^+,r\}),\infty)$.
Furthermore, for any $f\in (HE_\Phi^r)_t(\rn)$,
$$\|f\|_{(HE_\Phi^r)_t(\rn)}\sim\|g(f)\|_{(E_\Phi^r)_t(\rn)}
\sim\lf\|g_\lambda^\ast(f)\r\|_{(E_\Phi^r)_t(\rn)},$$
where the positive equivalence constants are independent of $f$ and $t$.
\end{itemize}
\end{corollary}

\begin{remark}\label{r2.3}
\begin{enumerate}
  \item[\rm(i)] Let $p\in(0,1]$. It is well known that, when
$\lambda\in(2/p,\infty)$, the classical Hardy space $H^p(\rn)$
can be characterized by the Littlewood--Paley $g_\lambda^\ast$-function
(see, for instance, \cite[Chapter 7]{fs82}). From Corollary \ref{c2.1}(i),
it follows that, when $X:=L^p(\rn)$ with $p\in(0,1]$,
the range of the index $\lambda$ in Corollary \ref{c2.1}(i) coincides
with the best-known classical case $\lambda\in(2/p,\infty)$.

\item[\rm(ii)] The $g_\lambda^\ast$-function characterization of
$H^p_w(\rn)$ in Corollary \ref{c2.1}(ii) was obtain in \cite[Theorem 2]{as77}.
Moreover, the $g$-function characterization of $H^p_w(\rn)$ is also known
(see, for instance, \cite[Theorem 4.8]{lhy12}).

\item[\rm(iii)] For the Herz--Hardy space $H\ck^\alpha_{p,q}(\rn)$,
Corollary \ref{c2.1}(iii)  is known (see, for instance, \cite{lyh}
and \cite[Theorem 3.1]{lw00}).

\item[\rm(iv)] For the Lorentz--Hardy space $H^{p,q}(\rn)$,
Corollary \ref{c2.1}(iv) was obtained in \cite[Theorems 2.8 and 2.9]{lyy18}.

\item[\rm(v)] To our best knowledge, Corollary \ref{c2.1}(v) is new.
We also point out that another version of the Littlewood--Paley characterization
of the Morrey--Hardy space $H\cm^p_q(\rn)$ was established in \cite[Theorem 4.2]{s09}.

\item[\rm(vi)] For the variable Hardy space $H^{p(\cdot)}(\rn)$,
Corollary \ref{c2.1}(vi) was obtained in \cite[Corollary 1.5]{zyl}
(see also \cite{jzzw,lwyy17,lwyy18,ns12}). We also point out that,
for the space $H^{p(\cdot)}(\rn)$, the range of the parameter
$\lambda$, obtained in \cite{jzzw,lwyy17,lwyy18},
is $(1+\frac{2}{\min\{2,p_-\}},\fz)$. Thus, our results improve those
results when $p_-\in(0,1]$
via widening the range $(1+\frac{2}{p_-},\fz)$
of the parameter $\lambda$ into $(1+\frac{2}{p_-}-\frac{2}{\max\{1,p_+\}},\fz)$.

\item[\rm(vii)]  For the Orlicz--Hardy space $H^\Phi(\rn)$, Corollary
\ref{c2.1}(vii) is known (see, for instance, \cite{lhy12} and \cite{ylk}).

\item[\rm(viii)] For the Orlicz-slice Hardy space $(HE_\Phi^r)_t(\rn)$,
the conclusion of Corollary \ref{c2.1}(viii)
was obtained in \cite[Theorems 3.18 and 3.19]{zyyw}.
\end{enumerate}
\end{remark}

\section{Boundedness of Calder\'on--Zygmund operators}\label{s3}

In this section, we obtain the boundedness of  Calder\'on--Zygmund
operators on $H_X(\rn)$. We begin with the notion of convolutional
Calder\'on--Zygmund operators (see, for instance, \cite[Section 5.3.2]{gl}).

\begin{definition}
For any given $\delta \in (0,1)$, a \emph{convolutional $\delta$-type
Calder\'on--Zygmund operator $T$}
is a linear bounded operator on $L^2(\rn)$ with the kernel $K\in\mathcal{S}'(\rn)$
coinciding with a locally integrable function on $\rn\setminus\{\vec{0}_n\}$ and satisfying:
\begin{enumerate}
\item[\rm{(i)}] There exists a positive constant $C$ such that,
for any $x,\ y\in \rn$ with $|x|>2|y|$,
$$|K(x-y)-K(x)|\le C \frac{|y|^\delta}{|x|^{n+\delta}};$$
\item[\rm{(ii)}] For any $f\in L^2(\rn)$ and $x\in\rn$, $T(f)(x)={\rm p.\,v.}\,K\ast f(x)$.
\end{enumerate}
\end{definition}

Moreover, we recall the notion of absolutely continuous quasi-norms as follows
(see, for instance, \cite[Definition 3.1]{br}).

\begin{definition}
Let $X$ be a ball quasi-Banach function space. A function $f\in X$ is said to have an
\emph{absolutely continuous quasi-norm} in $X$ if $\|f\mathbf{1}_{E_j}\|_X\downarrow 0$
whenever $\{E_j\}_{j=1}^\infty$ is a sequence of measurable sets that
satisfy $E_j \supset E_{j+1}$ for any $j \in \nn$ and $\cap_{j=1}^\infty E_j = \emptyset$.
Moreover, $X$ is said to have an \emph{absolutely continuous quasi-norm} if, for any $f\in X$, $f$  have an
absolutely continuous quasi-norm in $X$.
\end{definition}

\begin{remark}
We point out that the definition of the absolutely continuous (quasi-)norm
presented in \cite[Definition 2.5]{shyy} is not correct. More precisely,
in \cite[Definition 2.5]{shyy}, the condition that $\|\mathbf{1}_{E_j}\|_X\downarrow 0$
whenever $\{E_j\}_{j=1}^\infty$ is a sequence of measurable sets that
satisfy $E_j \supset E_{j+1}$ for any $j \in \nn$ and $\cap_{j=1}^\infty E_j = \emptyset$
should be replaced by that, for any $f\in X$, $\|f\mathbf{1}_{E_j}\|_X\downarrow 0$
whenever $\{E_j\}_{j=1}^\infty$ is a sequence of measurable sets that
satisfy $E_j \supset E_{j+1}$ for any $j \in \nn$ and $\cap_{j=1}^\infty E_j = \emptyset$.
\end{remark}
\begin{remark}
We point out that, except the Morrey space $\mathcal{M}^p_q(\rn)$,
the other examples of function spaces
in Remark \ref{r2.1} all have absolutely continuous quasi-norms (see
also \cite[p.\,10]{shyy} for some details).
\begin{enumerate}
\item[{\rm(i)}] By the Lebesgue dominated convergence theorem, it is easy to see that $L^p(\rn)$,
$L^p_w(\rn)$, $\ck^\alpha_{p,q}(\rn)$ and $(E_\Phi^r)_t(\rn)$ have an absolutely continuous quasi-norm.

\item[{\rm(ii)}]  We claim that  $L^{p(\cdot)}(\rn)$ has an absolutely continuous quasi-norm.
Indeed, let $f\in L^{p(\cdot)}(\rn)$ and $\{E_j\}_{j=1}^\infty$ be a sequence of measurable sets that
satisfy $E_j \supset E_{j+1}$ for any $j \in \nn$ and $\cap_{j=1}^\infty E_j = \emptyset$.
By the definition of  $L^{p(\cdot)}(\rn)$ and the Lebesgue dominated convergence theorem,
we know that
$$\int_{{\mathbb R}^n}|f(x)\mathbf{1}_{E_j}(x)|^{p(x)}dx\to 0 \quad\mathrm{as} \quad j\to\infty.$$
From this, we deduce that, for any $\varepsilon\in(0,1)$, there exists $j_0\in\nn$ such that,
for any $j\in\nn$ and $j>j_0$, $\int_{{\mathbb R}^n}|f(x)\mathbf{1}_{E_j}(x)|^{p(x)}dx<\varepsilon^{1/p_-}$, which implies that
$$\int_{{\mathbb R}^n}\left[\frac{|f(x)\mathbf{1}_{E_j}(x)|}{\varepsilon}\right]^{p(x)}dx<1.$$
By this, we find that $\|f\mathbf{1}_{E_j}\|_{L^{p(\cdot)}(\rn)}\leq \varepsilon$ and
hence $\lim_{j\to\infty}\|f\mathbf{1}_{E_j}\|_{L^{p(\cdot)}(\rn)}=0$.
This proves that $L^{p(\cdot)}(\rn)$ has an absolutely continuous quasi-norm.
Furthermore, using a similar argument as above, we can show that
the Musielak--Orlicz space $L^{\varphi}(\rn)$ has also an absolutely continuous quasi-norm
and we omit the details.

\item[{\rm(iii)}] We now claim that $L^{p,q}(\rn)$ with $p,\ q \in(0,\infty)$ has an absolutely continuous quasi-norm.
Indeed, by \cite[Proposition 1.4.9]{gl}, we know that, for any $f\in L^{p,q}(\rn)$ with $p,\ q\in (0,\infty)$,
\begin{equation}\label{pql}\|f\|_{L^{p,q}(\rn)}=p^{\frac{1}{q}}
\left(\int_0^\infty\lf\{[\mu_f(s)]^{\frac{1}{p}}s\r\}^q\,\frac{ds}{s}\right)^{\frac{1}{q}},
\end{equation}
where, for any $s\in(0,\fz)$,
$$\mu_f(s):=|\{x\in\rn:\ |f(x)|>s\}|.$$
Let $f\in L^{p,q}(\rn)$ with $p,\ q\in (0,\infty)$ and $\{E_j\}_{j=1}^\infty$ be a sequence of measurable sets that
satisfy $E_j \supset E_{j+1}$ for any $j \in \nn$ and $\cap_{j=1}^\infty E_j = \emptyset$. By \eqref{pql} and the fact that $\mu_f$ is non-increasing, we find that, for any $s\in(0,\infty)$, $\mu_f(s)<\fz$.
From this, it follows that, for any given $s\in(0,\fz)$,
$$\mu_{f\mathbf{1}_{E_j}}(s)=|\{x\in\rn:\ |f(x)\mathbf{1}_{E_j}(x)|>s\}|=|\{x\in\rn:\ |f(x)|>s\}
\cap E_j|\to 0$$
as $j\to \infty$. By this, the fact that, for any $j\in\nn$, $\mu_{f\mathbf{1}_{E_j}}\leq\mu_f$, and the Lebesgue dominated convergence theorem, we conclude that
\begin{align*}
\lim_{j\to\infty}\|f\mathbf{1}_{E_j}\|_{L^{p,q}(\rn)}&=p^{\frac{1}{q}}\lim_{j\to\infty}
\left(\int_0^\infty\lf\{[\mu_{f\mathbf{1}_{E_j}}(s)]^{\frac{1}{p}}s\r\}^q\,\frac{ds}{s}\right)^{\frac{1}{q}}\\
&=p^{\frac{1}{q}}
\left(\int_0^\infty\lim_{j\to\infty}\lf\{[\mu_{f\mathbf{1}_{E_j}}(s)]^{\frac{1}{p}}s\r\}^q\,
\frac{ds}{s}\right)^{\frac{1}{q}}
=0,
\end{align*}
which prove that $L^{p,q}(\rn)$ with $p,\ q \in(0,\infty)$ has an absolutely continuous quasi-norm.
\end{enumerate}
\end{remark}

Now we have the following conclusion on convolutional $\delta$-type
Calder\'on--Zygmund operators.

\begin{theorem}\label{ccz}
Assume that $X$ is a ball quasi-Banach function space satisfying both
\eqref{ma} and Assumption \ref{a2} with the same $s\in (0,1]$,
and having an absolutely continuous quasi-norm. Let $\delta\in(0,1)$ and $T$ be a
convolutional $\delta$-type Calder\'on--Zygmund operator.
If $\theta\in [\frac{n}{n+\delta},1)$,
then $T$ has a unique extension on $H_X(\rn)$.
Moreover, there exists a positive constant $C$ such that, for any $f\in H_X(\rn)$,
$$\lf\|T(f)\r\|_{H_X(\rn)}\le C\|f\|_{H_X(\rn)}.$$
\end{theorem}

To prove Theorem \ref{ccz}, we need the following properties of $H_X(\rn)$
and $X$, where Lemma \ref{atom} is just \cite[Theorem 3.7]{shyy} and
Lemma \ref{mch} is a part of \cite[Theorem 3.1]{shyy}.

\begin{lemma}\label{atom}
Let $X$ be a ball quasi-Banach function space satisfying
\eqref{ma} with $0<\theta<s\le 1$.
Assume further that $X$ has an absolutely continuous quasi-norm,
\begin{equation}\label{dx}
d_X:=\lceil n(1/\theta-1)\rceil
\end{equation}
and $d \in [d_X,\infty)\cap\zz_+$ is a fixed integer.  Then, for any $f\in H_X(\rn)$,
there exist a sequence $\{a_j\}_{j=1}^\infty$ of $(X,\infty,d)$-atoms supported, respectively,  in a
sequence $\{Q_j\}^\infty_{j=1}$ of cubes, a sequence
$\{\lambda_j\}^\infty_{j=1}$ of non-negative numbers and a positive constant $C_{(s)}$,
independent of $f$ but depending on $s$, such that
$$f=\sum_{j=1}^\infty \lambda_ja_j\quad\quad\text{in}\quad\mathcal{S}'(\rn)$$
and
$$\left\|\left\{\sum_{j=1}^\infty\left(\frac{\lambda_j}{\|\mathbf{1}_{Q_j}\|_X}
\right)^s\mathbf{1}_{Q_j}\right\}^{1/s}\right\|_X\le C_{(s)}\|f\|_{H_X(\rn)}.$$
\end{lemma}

\begin{lemma}\label{mch}
Let $X$ be a ball quasi-Banach function space, $r\in(0,\infty)$, $b\in(n/r,\infty)$
and $M$ be bounded on $X^{1/r}$.
Assume that $\Phi\in\mathcal{S}(\rn)$ satisfies $\int_{\rn}\Phi(x)\,dx\neq 0$.
Then, for any $f\in\mathcal{S}'(\rn)$,
$$\lf\|M_b^{\ast\ast}(f,\Phi)\r\|_X\sim \|M(f,\Phi)\|_X,$$
where $M_b^{\ast\ast}(f,\Phi)$ is as in \eqref{mbaa},
$M(f,\Phi):= \sup_{t\in(0,\infty)}|\Phi_t\ast f|$ and the positive equivalence constants are independent of $f$.
\end{lemma}

Using Lemmas \ref{atom} and \ref{mch}, we now show Theorem \ref{ccz}.

\begin{proof}[Proof of Theorem \ref{ccz}]
Since $L^2(\rn)\cap H_X(\rn)$ is dense in $H_X(\rn)$,
to prove Theorem \ref{ccz}, by a standard density argument,
it suffices to show that $T$ is bounded from
$L^2(\rn)\cap H_X(\rn)$ to $H_X(\rn)$.

Assume that $f\in L^2(\rn) \cap H_X(\rn)$. By Lemma \ref{atom},
we know that there exist a sequence $\{a_j\}_{j=1}^\infty$ of $(X,2,d)$-atoms
supported, respectively,  in cubes $\{Q_j\}^\infty_{j=1}$,
and a sequence $\{\lambda_j\}^\infty_{j=1}$ of non-negative numbers such that
\begin{equation}\label{de}
f=\sum_{j=1}^\infty \lambda_ja_j\quad\quad\text{in}\quad L^2(\rn)
\end{equation}
and
\begin{equation}\label{nc}
\left\|\left\{\sum_{j=1}^\infty\left(\frac{\lambda_j}{\|\mathbf{1}_{Q_j}\|_X}\right)^s
\mathbf{1}_{Q_j}\right\}^{1/s}\right\|_X\lesssim\|f\|_{H_X(\rn)}.
\end{equation}
Since $T$ is bounded on $L^2(\rn)$, it follows that
$$T(f)=\sum_{j=1}^\infty \lambda_jT(a_j) \quad \quad \text{in}\quad L^2(\rn).$$

Let $\phi \in \mathcal{S}(\rn)$ satisfy $\int_{\rn}\phi(x)\,dx\neq0$.
For any $x\in\rn$, define
$$\phi_+^\ast(T(f))(x):= \sup_{t\in(0,\infty)}\lf|T(f)\ast\phi_t(x)\r|.$$
Then, by Lemma \ref{mch}, we know that
$$\|T(f)\|_{H_X(\rn)}\sim\lf\|\phi_+^\ast(T(f))\r\|_X.$$
Thus, to prove Theorem \ref{ccz}, we only need to show that
\begin{equation}\label{ph}
\lf\|\phi_+^\ast(T(f))\r\|_X\lesssim\|f\|_{H_X(\rn)}.
\end{equation}
To this end, from \eqref{de}, we first deduce that
\begin{align}\label{pii}
\lf\|\phi_+^\ast(T(f))\r\|_X&\leq\left\|\sum_{j=1}^\infty\lambda_j\phi_+^\ast
(T(a_j))\right\|_X\notag\\
&\lesssim\left\|\sum_{j=1}^\infty\lambda_j\phi_+^\ast(T(a_j))\mathbf{1}_{4Q_j}\right\|_X+
\left\|\sum_{j=1}^\infty\lambda_j\phi_+^\ast(T(a_j))\mathbf{1}_{(4Q_j)^\complement}\right\|_X\notag\\
&=:\mathrm{I}_1+\mathrm{I}_2
\end{align}
Notice that, for any $j \in \nn$, $\phi_+^\ast(T(a_j))\lesssim M(T(a_j))$
and $a_j\in L^2(\rn)$. Since $M$ and $T$ are both bounded on $L^2(\rn)$,
it follows that
$$\left\|\phi_+^\ast(T(a_j))\mathbf{1}_{4Q_j}\right\|_{L^2(\rn)}
\lesssim\lf\|M(T(a_j))\r\|_{L^2(\rn)}
\lesssim\|T(a_j)\|_{L^2(\rn)}\lesssim\|a_j\|_{L^2(\rn)}
\lesssim\frac{|Q_j|^{1/2}}{\|\mathbf{1}_{Q_j}\|_X},$$
which, combined with Lemma \ref{r}, implies that
\begin{equation}\label{e1}
\mathrm{I}_1\lesssim\left\|\left[\sum_{j=1}^\infty
\left(\frac{\lambda_j}{\|\mathbf{1}_{Q_j}\|_X}\right)^s
\mathbf{1}_{Q_j}\right]^{1/s}\right\|_X.
\end{equation}
Moreover, it was proved in \cite[p.\,2881]{yyyz} that, for any $x\in(4Q_j)^\complement$,
$$|\phi_+^\ast(T(a_j))(x)|\lesssim\frac{\ell_j^\delta}{|x-x_j|^{n+\delta}}
\|a_j\|_{L^2(\rn)}|Q_j|^{1/2},$$
where $x_j$ denotes the center of $Q_j$ and $\ell_j$ its side length.
Form this and the fact that $\theta\in [\frac{n}{n+\delta},1)$, we deduce that
$$|\phi_+^\ast(T(a_j))|\mathbf{1}_{(4Q_j)^\complement}\lesssim\frac{1}{\|\mathbf{1}_{Q_j}\|_X}
[M(\mathbf{1}_{Q_j})]^{\frac{n+\delta}{n}}\lesssim\frac{1}
{\|\mathbf{1}_{Q_j}\|_X}M^{(\theta)}(\mathbf{1}_{Q_j}),
$$
which, together with \eqref{ma}, further implies that
\begin{equation}\label{e2}
\mathrm{I}_2\lesssim\left\|\sum_{j=1}^\infty\frac{\lambda_j}
{\|\mathbf{1}_{Q_j}\|_X}M^{(\theta)}(\mathbf{1}_{Q_j})\right\|_X\lesssim
\left\|\left[\sum_{j=1}^\infty\left(\frac{\lambda_j}
{\|\mathbf{1}_{Q_j}\|_X}\right)^s\mathbf{1}_{Q_j}\right]^{1/s}\right\|_X.
\end{equation}
Combining \eqref{nc}, \eqref{pii}, \eqref{e1} and \eqref{e2}, we  conclude that
$$\lf\|\phi_+^\ast(T(f))\r\|_X\lesssim\|f\|_{H_X(\rn)},$$
which completes the proof of Theorem \ref{ccz}.
\end{proof}

Based on Theorem \ref{ccz}, we can weaken the assumption
that $X$ has an absolutely continuous quasi-norm into a weaker assumption
which is applicable to Morrey spaces.

\begin{theorem}\label{ccz-new}
Assume that $X$ and  $Y$ are two ball quasi-Banach function spaces satisfying both
\eqref{ma} with $0<\theta<s\le1$ and Assumption \ref{a2} with the same $s$ as in \eqref{ma}.
Moreover, assume that $Y$ has an absolutely continuous quasi-norm
and $X$ continuously embeds into $Y$. Let $\delta\in(0,1)$ and $T$ be a
convolutional $\delta$-type Calder\'on--Zygmund operator.
If $\theta\in [\frac{n}{n+\delta},1)$,
then $T$ has a unique extension on $H_X(\rn)$.
Moreover, there exists a positive constant $C$ such that, for any $f\in H_X(\rn)$,
$$\lf\|T(f)\r\|_{H_X(\rn)}\le C\|f\|_{H_X(\rn)}.$$
\end{theorem}

\begin{proof}
Assume $f\in H_X(\rn)$ and $d \in [\max\{d_X,d_Y\},\infty)\cap\zz_+$ is a fixed integer,
where $d_X$ and $d_Y$ are as in \eqref{dx}. Then, by \cite[Theorem 3.7]{shyy},
we find that there exist a sequence $\{a_j\}_{j=1}^\infty$ of $(X,\infty,d)$-atoms supported,
respectively,  in a sequence $\{Q_j\}^\infty_{j=1}$ of cubes, and a sequence
$\{\lambda_j\}^\infty_{j=1}$ of non-negative numbers,
independent of $f$ but depending on $s$, such that
$$f=\sum_{j=1}^\infty \lambda_ja_j\quad\quad\text{in}\quad\mathcal{S}'(\rn)$$
and
$$\left\|\left\{\sum_{j=1}^\infty\left(\frac{\lambda_j}{\|\mathbf{1}_{Q_j}\|_X}
\right)^s\mathbf{1}_{Q_j}\right\}^{1/s}\right\|_X\lesssim\|f\|_{H_X(\rn)}.$$
From the assumption that $X$ continuously embeds into $Y$, we deduce  that
\begin{align}\label{new-1}
\left\|\left\{\sum_{j=1}^\infty\left(\lambda_j\frac{\|\mathbf{1}_{Q_j}\|_Y}{\|\mathbf{1}_{Q_j}\|_X}
\frac{1}{\|\mathbf{1}_{Q_j}\|_Y}
\right)^s\mathbf{1}_{Q_j}\right\}^{1/s}\right\|_Y
&=\left\|\left\{\sum_{j=1}^\infty\left(\frac{\lambda_j}{\|\mathbf{1}_{Q_j}\|_X}
\right)^s\mathbf{1}_{Q_j}\right\}^{1/s}\right\|_Y\notag\\
&\lesssim \left\|\left\{\sum_{j=1}^\infty\left(\frac{\lambda_j}{\|\mathbf{1}_{Q_j}\|_X}
\right)^s\mathbf{1}_{Q_j}\right\}^{1/s}\right\|_X
\lesssim\|f\|_{H_X(\rn)}.
\end{align}
Moreover, since, for any $j\in\nn$, $a_j$ is an $(X,\infty,d)$-atom,
it follows that, for any $j\in\nn$,
$\frac{\|\mathbf{1}_{Q_j}\|_X}{\|\mathbf{1}_{Q_j}\|_Y}a_j$ is a $(Y,\infty,d)$-atom.
By this, \cite[Theorem 3.6 and Corollary 3.11]{shyy} and \eqref{new-1}, we know that
\begin{equation}\label{conv-f}
\sum_{j=1}^\infty\left(\lambda_j\frac{\|\mathbf{1}_{Q_j}\|_Y}{\|\mathbf{1}_{Q_j}\|_X}\right)
\left(\frac{\|\mathbf{1}_{Q_j}\|_X}{\|\mathbf{1}_{Q_j}\|_Y}a_j\right)=
\sum_{j=1}^\infty \lambda_ja_j=f \quad\quad\text{in}
\quad\mathcal{S}'(\rn) \quad \text{and} \quad H_Y(\rn).
\end{equation}
Furthermore, from Theorem \ref{ccz} and \eqref{conv-f}, we deduce that $T$ is bounded on $H_Y(\rn)$ and hence
$$Tf = \sum_{j=1}^\infty \lambda_jT(a_j) \quad\text{in}\quad H_Y(\rn)\quad
\text{and} \quad \mathcal{S}'(\rn).$$
Using this and repeating the arguments used in the proof of
Theorem \ref{ccz}, we conclude that
$$\|Tf\|_{H_X(\rn)}\lesssim \|f\|_{H_X(\rn)},$$
which completes the proof of Theorem \ref{ccz-new}.
\end{proof}

\begin{remark}\label{morr-lw}
If $X$ is one of the following spaces
$$L^p(\rn),\quad L^p_w(\rn),\quad \ck^\alpha_{p,q}(\rn),
\quad L^{p,q}(\rn),\quad L^{p(\cdot)}(\rn),\quad L^\Phi(\rn)\quad\mathrm{or}\quad (E_\Phi^r)_t(\rn),$$
then we can just choose $Y:=X$ itself. If $X:=\cm^p_q(\rn)$ with $0<q\leq p <\infty$,
then we choose $Y:=L^q_w(\rn)$ with
$w:=[M(\mathbf{1}_{B({\vec 0_n},1)})]^{\widetilde{\theta}}$
and $\widetilde{\theta}\in (1-\frac{q}{p},1)$. By \cite[Theorem 7.2.7]{gl},
we know that $w\in A_1(\rn)$ and therefore $Y$ satisfies both
\eqref{ma} with $0<\theta<s\le 1$ and Assumption \ref{a2} with the same $s$ as in \eqref{ma},
and has an absolutely continuous quasi-norm.
We claim that $X$ continuously embeds into $Y$.
To show this, it suffices to show that there exists a positive constant $C$ such that, for any $f\in \cm^p_q(\rn)$,
\begin{equation}\label{lw_m}
\|f\|_{L^q_w(\rn)}\le C\|f\|_{\cm^p_q(\rn)}.
\end{equation}
Indeed, it is easy to see that, for any $x\in B({\vec 0_n},2)$,
\begin{equation}\label{est_c}
M(\mathbf{1}_{B({\vec 0_n},1)})(x) \leq 1.
\end{equation}
Also, observe that, for any $x\in\rn$,
$$M(\mathbf{1}_{B({\vec 0_n},1)})(x)=\sup_{r\in(0,\infty)}\frac{1}{|B(x,r)|}
\int_{B(x,r)}\mathbf{1}_{B({\vec 0_n},1)}(y)\,dy=\sup_{r\in(0,\infty)}\frac{|B({\vec 0_n},1)\cap B(x,r)|}{|B(x,r)|}$$
and, for any $k\in\nn$, $x\in B({\vec 0_n},2^{k+1})\setminus B({\vec 0_n},2^k)$
and $r\in (0,2^k-1)$, we have $B({\vec 0_n},1)\cap B(x,r)=\emptyset$.
From these, we further deduce that, for any $k\in\nn$ and $x\in B({\vec 0_n},2^{k+1})\setminus B({\vec 0_n},2^k)$,
\begin{equation}\label{est_r}
M(\mathbf{1}_{B({\vec 0_n},1)})(x)=\sup_{r\in(2^k-1,\infty)}\frac{|B({\vec 0_n},1)\cap B(x,r)|}{|B(x,r)|}\lesssim2^{-kn}.
\end{equation}
Combining \eqref{est_c}, \eqref{est_r} and $\widetilde{\theta}\in (1-\frac{q}{p},1)$,
we conclude that, for any $f\in \cm^p_q(\rn)$,
\begin{align*}
\int_{\rn}|f(x)|^q\left[M(\mathbf{1}_{B({\vec 0_n},1)})\right]^{\widetilde{\theta}}\,dx
&=\int_{B({\vec 0_n},2)}|f(x)|^q\left[M(\mathbf{1}_{B({\vec 0_n},1)})\right]^{\widetilde{\theta}}\,dx\\
&\quad+\sum_{k=1}^\infty\int_{B({\vec 0_n},2^{k+1})\setminus B({\vec 0_n},2^k)}|f(x)|^q\left[M(\mathbf{1}_{B({\vec 0_n},1)})\right]^{\widetilde{\theta}}\,dx\\
&\lesssim \|f\|_{\cm^p_q(\rn)}^q+\sum_{k=1}^\infty2^{-kn\widetilde{\theta}}\int_{B({\vec 0_n},2^{k+1})}|f(x)|^q\,dx\\
&\lesssim \|f\|_{\cm^p_q(\rn)}^q+\sum_{k=1}^\infty2^{-kn\widetilde{\theta}}2^{kn(1-\frac{q}{p})}\|f\|_{\cm^p_q(\rn)}^q
\lesssim \|f\|_{\cm^p_q(\rn)}^q,
\end{align*}
which implies that \eqref{lw_m} holds true and hence completes the proof of the above claim.
\end{remark}

Now we establish the boundedness of
 so-called \emph{$\gamma$-order Calder\'on--Zygmund operators}
on the Hardy space $H_X(\rn)$ and we begin with their definitions.

\begin{definition}
For any given $\gamma \in (0,\infty)$, a $\gamma$-order
Calder\'on--Zygmund operator $T$
is a linear bounded operator on $L^2(\rn)$ and there exists a kernel $K$ on
$(\rn\times\rn)\setminus\{(x,x):\ x\in\rn\}$ satisfying:
\begin{enumerate}
\item[{\rm(i)}] For any $\alpha:=(\az_1,\ldots,\az_n) \in \zn_+$ with $|\alpha|\leq\lceil\gamma\rceil-1$,
there exists a positive constant $C$ such that, for any $x,\ y,\ z\in\rn$ with $|x-y|>2|y-z|$,
$$\lf|\partial_x^\alpha K(x,y)-\partial_x^\alpha K(x,z)\r|\le C\frac{|y-z|^{\gamma-\lceil\gamma\rceil+1}}{|x-y|^{n+\gamma}},$$
here and thereafter, $\partial_x^\alpha:=(\partial/\partial x_1)^{\az_1}\cdots(\partial/\partial x_1)^{\az_n}$;
\item[{\rm(ii)}] For any $f\in L^2(\rn)$ with compact support and $x\notin\supp(f)$,
$$Tf(x) =\int_{\supp(f)}K(x,y)f(y)\,dy.$$
\end{enumerate}
\end{definition}
For any $m\in\nn$, an operator $T$ is said to have the \emph{vanishing moments up to order $m$} if, for any $f\in L^2(\rn)$
with compact support satisfying that, for any $\beta:=(\bz_1,\ldots,\bz_n)\in\zn_+$ with
$|\beta|\le m$, $\int_{\rn}x^\beta f(x)\,dx=0$,
it holds true that $\int_{\rn}x^\beta Tf(x)\,dx=0$, here and thereafter,
$x^\beta:=x_1^{\bz_1}\cdots x_n^{\bz_n}$ for any$x:=(x_1,\ldots,x_n)\in\rn$.

Then we have the following conclusion about $\gamma$-order Calder\'on--Zygmund operators.

\begin{theorem}\label{gcz}
Assume that $X$ is a ball quasi-Banach function space satisfying both \eqref{ma}
with $0<\theta<s\le 1$ and Assumption \ref{a2} with the same $s$ as in \eqref{ma}, and having an absolutely
continuous quasi-norm.
Let $\gamma\in(0,\infty)$ and $T$ be a $\gamma$-order
Calder\'on--Zygmund operator and have the vanishing moments up to order $\fg-1$.
If $\theta\in [\frac{n}{n+\gamma},\frac{n}{n+\fg-1})$, then $T$ has a
unique extension on $H_X(\rn)$.
Moreover, there exists a positive constant $C$ such that, for any $f\in H_X(\rn)$,
$$\|T(f)\|_{H_X(\rn)}\le C\|f\|_{H_X(\rn)}.$$
\end{theorem}

\begin{proof}
Let $f$, $\{\lambda_j\}_{j\in\nn}$ and $\{a_j\}_{j\in\nn}$ be the same as in the proof of Theorem \ref{ccz}. By an argument similar to that used therein, to prove this theorem,
it suffices to show that
$$\left\|\sum_{j=1}^\infty\lambda_j\phi_+^\ast(T(a_j))\mathbf{1}_{(4Q_j)^\complement}
\right\|_X\lesssim \|f\|_{H_X(\rn)}.$$
To this end, we employ some estimates from \cite{zyyw}. Indeed,
the following estimates were obtained in \cite{zyyw}
(see \cite[pp.\,54-56, the estimates of ${\rm II_1}$, ${\rm II_2}$ and ${\rm II_3}$]{zyyw} for the details): for any $x \in (4Q_j)^\complement$,
 \begin{equation*}
 \phi_+^\ast(T(a_j))(x)\lesssim\max\left\{\frac{\ell_j^{\fg}}{|x-x_j|^{n+\fg}},
 \frac{\ell_j^{\gamma}}{|x-x_j|^{n+\gamma}}\right\}\|a_j\|_{L^2(\rn)}|Q_j|^{\frac{1}{2}},
 \end{equation*}
where $x_j$ denotes the center of $Q_j$ and $\ell_j$ its side length.
By this and the fact that $\theta\in [\frac{n}{n+\gamma},\frac{n}{n+\fg-1})$,
we conclude that, for any $x \in (4Q_j)^\complement$,
 $$\phi_+^\ast(T(a_j))(x)\lesssim\frac{\ell_j^{n+\gamma}}{|x-x_j|^{n+\gamma}}
 \frac{1}{\|\mathbf{1}_{Q_j}\|_X}\lesssim
 \frac{1}{\|\mathbf{1}_{Q_j}\|_X}M^{(\theta)}(\mathbf{1}_{Q_j})(x),$$
which, combined with \eqref{ma}, implies that
\begin{equation*}
\left\|\sum_{j=1}^\infty\lambda_j\phi_+^\ast(T(a_j))\mathbf{1}_{(4Q_j)^\complement}\right\|_X
\lesssim\left\|\sum_{j=1}^\infty\frac{\lambda_j}{\|\mathbf{1}_{Q_j}\|_X}M^{(\theta)}
(\mathbf{1}_{Q_j})\right\|_X\lesssim\left\|\left[\sum_{j=1}^\infty\left(\frac{\lambda_j}
{\|\mathbf{1}_{Q_j}\|_X}\right)^s\mathbf{1}_{Q_j}\right]^{1/s}\right\|_X.
\end{equation*}
This finishes the proof of Theorem \ref{gcz}.
\end{proof}

From Theorems \ref{ccz} and \ref{gcz}, and Remarks \ref{r2.1} and \ref{r2.2},
it follows the following conclusions on several concrete function spaces.

\begin{corollary}\label{c3.1}
Let $\delta\in(0,1)$ and $\gamma\in(0,\infty)$.
Assume that $T$ is a convolutional $\delta$-type Calder\'on--Zygmund operator
or a $\gamma$-order Calder\'on--Zygmund operator with the vanishing moments
up to order $\fg-1$.
\begin{enumerate}
\item[{\rm(i)}] If $p\in(\frac{n}{n+\delta},1]$
or $p\in(\frac{n}{n+\gamma},\frac{n}{n+\fg-1}]$, then $T$ has a unique extension on $H^p(\rn)$.
Moreover, there exists a positive constant $C$ such that, for any $f\in H^p(\rn)$,
$\|T(f)\|_{H^p(\rn)}\leq C\|f\|_{H^p(\rn)}$.

\item[{\rm(ii)}]If $\theta\in[\frac{n}{n+\delta},1)$
or $\theta\in[\frac{n}{n+\gamma},\frac{n}{n+\fg-1})$, and
$w\in A_\infty(\rn)$, where $\theta$ is as in Remarks
\ref{r2.1}(b) and \ref{r2.2}(b), then $T$ has a unique extension on $H^p_w(\rn)$.
Moreover, there exists a positive constant $C$ such that, for any $f\in H^p_w(\rn)$,
$\|T(f)\|_{H^p_w(\rn)}\leq C\|f\|_{H^p_w(\rn)}$.

\item[{\rm(iii)}] Let $p\in (0,1]$, $q\in(0,\fz)$ and $\alpha\in(-n/p,\infty)$.
If $\min\{p,[\alpha/n+1/p]^{-1}\}\in(\frac{n}{n+\delta},1]$
or $\min\{p,[\alpha/n+1/p]^{-1}\}\in(\frac{n}{n+\gamma},\frac{n}{n+\fg-1}]$,
then $T$ has a unique extension on $H\ck^\alpha_{p,q}(\rn)$.
Moreover, there exists a positive constant $C$ such that,
for any $f\in H\ck^\alpha_{p,q}(\rn)$,
$\|T(f)\|_{H\ck^\alpha_{p,q}(\rn)}\leq C\|f\|_{H\ck^\alpha_{p,q}(\rn)}.$

\item[{\rm(iv)}] Let $p\in (0,1]$ and $q\in(0,\fz)$.
If $\min\{p,q\}\in(\frac{n}{n+\delta},1]$ or  $\min\{p,q\}\in(\frac{n}{n+\gamma},
\frac{n}{n+\fg-1}]$, then $T$ has a unique extension on $H^{p,q}(\rn)$.
Moreover, there exists a positive constant $C$ such that, for any $f\in H^{p,q}(\rn)$,
$\|T(f)\|_{H^{p,q}(\rn)}\leq C\|f\|_{H^{p,q}(\rn)}.$

\item[{\rm(v)}] Let $p:\ \rn\to (0,1]$ be globally log-H\"older continuous
with $p_-$ and $p_+$ being as in \eqref{eq-p}.
If $p_-\in(\frac{n}{n+\delta},1]$ or  $p_-\in(\frac{n}{n+\gamma},\frac{n}{n+\fg-1}]$,
then $T$ has a unique extension on $H^{p(\cdot)}(\rn)$. Moreover,
there exists a positive constant $C$ such that, for any $f\in H^{p(\cdot)}(\rn)$,
$\|T(f)\|_{H^{p(\cdot)}(\rn)}\leq C\|f\|_{H^{p(\cdot)}(\rn)}.$

\item[{\rm(vi)}] Let $\Phi$ be an Orlicz function with lower type $p_\Phi^-\in(0,1)$
and upper type $p_\Phi^+=1$. If $p_\Phi^-\in(\frac{n}{n+\delta},1]$ or
$p_\Phi^-\in(\frac{n}{n+\gamma},\frac{n}{n+\fg-1}]$, then
$T$ has a unique extension on $H^\Phi(\rn)$. Moreover, there exists a positive
constant $C$ such that, for any $f\in H^\Phi(\rn)$,
$\|T(f)\|_{H^\Phi(\rn)}\leq C\|f\|_{H^\Phi(\rn)}.$

\item[{\rm(vii)}] Let $t,\ r\in(0, \fz)$
and $\Phi$ be an Orlicz function with
$p_\Phi^-,\ p_\Phi^+\in(0,\fz)$. If
$\min\{p_\Phi^-, r\}\in(\frac{n}{n+\dz},1]$ or $\min\{p_\Phi^-, r\}\in
(\frac{n}{n+\gamma}, \frac{n}{n+\fg-1}]$, then
$T$ has a unique extension on $(HE^r_\Phi)_t(\rn)$.
Moreover, there exists a positive
constant $C$ such that, for any $t\in(0, \fz)$ and $f\in (HE^r_\Phi)_t(\rn)$,
$$\|T(f)\|_{(HE^r_\Phi)_t(\rn)}\leq C\|f\|_{(HE^r_\Phi)_t(\rn)}.$$
\end{enumerate}
\end{corollary}

\begin{remark}\label{r3.1}
Corollary \ref{c3.1}(i) was obtained in \cite{am86,dj84}.
Moreover, Corollary \ref{c3.1}(ii) was established in \cite[Theorem 3]{qy00}.
For the Herz--Hardy space $H\ck^\alpha_{p,q}(\rn)$,
the conclusion of Corollary \ref{c3.1}(iii) was obtained in \cite[Theorems 1 and 2]{ll97}.
For the variable Hardy space $H^{p(\cdot)}(\rn)$,
Corollary \ref{c3.1}(v) was established in \cite[Theorems 5.3]{s13}.
For Corollary \ref{c3.1}(vi), see also \cite[Proposition 5.3 and Theorem 5.5]{ns14}.
Furthermore, Corollary \ref{c3.1}(vii) was obtained in \cite[Theorems 6.11 and 6.13]{zyyw}.
\end{remark}

Similarly to Theorem \ref{ccz-new}, we have the following theorem, which is also applicable
to Morrey spaces; we omit the details.

\begin{theorem}\label{gcz-new}
Assume that $X$ and  $Y$ are two ball quasi-Banach function spaces satisfying both
\eqref{ma} with $0<\theta<s\le 1$ and Assumption \ref{a2} with the same $s$ as in \eqref{ma}.
Moreover, assume that $Y$ has an absolutely continuous quasi-norm
and $X$ continuously embeds into $Y$.
Let $\gamma\in(0,\infty)$ and $T$ be a $\gamma$-order
Calder\'on--Zygmund operator and have the vanishing moments up to order $\fg-1$.
If $\theta\in [\frac{n}{n+\gamma},\frac{n}{n+\fg-1})$, then $T$ has a
unique extension on $H_X(\rn)$.
Moreover, there exists a positive constant $C$ such that, for any $f\in H_X(\rn)$,
$$\|T(f)\|_{H_X(\rn)}\le C\|f\|_{H_X(\rn)}.$$
\end{theorem}

\section{Boundedness of $S^0_{1,0}(\rn)$ pseudo-differential operators}\label{s4}

In this section, we obtain the boundedness of $S^0_{1,0}(\rn)$
pseudo-differential operators on
the local Hardy space $h_X(\rn)$ (see Definition \ref{lh} below for its definition).
We begin with the notion of pseudo-differential operators. We refer the reader to \cite{g}
or \cite[Definition 8.6.4]{ylk} for more details.

\begin{definition}
A \emph{symbol} in $S^0_{1,0}(\rn)$ is a smooth function $\sigma$
defined on $\rn\times\rn$ satisfying that, for any multi-indices $\alpha,\  \beta \in \zn_+$,
there exists a positive constant $C_{(\alpha,\beta)}$,
independent of $x$ and $\xi$, but depending on $\alpha$ or $\beta$,
such that, for any $x,\,\xi\in\rn$,
$$\lf|\partial_x^\alpha\partial_\xi^\beta\sigma(x,\xi)\r|\le
C_{(\alpha, \beta)}(1+|\xi|)^{-|\beta|}.$$

Let $f\in\mathcal{S}(\rn)$. Then the \emph{$S_{1,0}^0(\rn)$ pseudo-differential
operator $T$} is defined by setting, for any $x \in \rn$,
$$T(f)(x):=\int_{\rn}\sigma(x,\xi)e^{ix\cdot\xi}\mathcal{F}(f)(\xi)\,d\xi.$$
\end{definition}

We also recall some local-type maximal functions and the local Hardy-type space $h_X(\rn)$
associated with the ball quasi-Banach space $X$ as follows
(see also \cite[Definition 5.1]{shyy} for more details).

\begin{definition}\label{lh}
Let $N\in\nn$, $a,\ b\in(0,\infty)$, $\Phi\in\mathcal{S}(\rn)$ and $f\in \mathcal{S}'(\rn)$.
\begin{enumerate}
\item[{\rm(i)}] The \emph{local radial maximal function} $m(f,\Phi)$
is defined by setting, for any $x\in\rn$,
$$m(f,\Phi)(x):=\sup_{t\in(0,1)}|(\Phi_t\ast f)(x)|.$$
\item[{\rm(ii)}] The \emph{local non-tangential maximal function}
$m^\ast_a(f,\Phi)$ is defined by
setting, for any $x\in\rn$,
$$m^\ast_a(f,\Phi)(x):=\sup_{t\in(0,1)}\left\{\sup_{y\in\rn,|y-x|<at}|\Phi_t\ast f(y)|\right\}.$$
\item[{\rm(iii)}] The \emph{local maximal function $m_b^{\ast\ast}(f,\Phi)$
of Peetre type}
is defined by setting, for any $x\in\rn$,
$$m_b^{\ast\ast}(f,\Phi)(x):= \sup_{(y,t)\in\rn\times(0,1)}
\frac{|(\Phi_t\ast f)(x-y)|}{(1+t^{-1}|y|)^b}.$$
\item[{\rm(iv)}] The \emph{local grand maximal function $m_{b,N}^{\ast\ast}(f)$
of Peetre type} is defined by setting, for any $x\in\rn$,
$$m_{b,N}^{\ast\ast}(f)(x):= \sup_{\psi\in\mathcal{F}_N}\left
\{\sup_{(y,t)\in\rn\times(0,1)}\frac{|(\psi_t\ast f)(x-y)|}{(1+t^{-1}|y|)^b}\right\}.$$
\item[{\rm(v)}] Assume further that $\int_{\rn}\Phi(x)\,dx\neq 0$.
Then the \emph{local Hardy-type space $h_X(\rn)$} is defined to be
the set of all $g\in\mathcal{S}'(\rn)$ such that
    $$\|g\|_{h_X(\rn)}:=\lf\|m_b^{\ast\ast}(g,\Phi)\r\|_X <\infty,$$
    where $m_b^{\ast\ast}(g,\Phi)$ is as in {\rm{(iii)}} with $b$ sufficiently large.
\end{enumerate}
\end{definition}

The following maximal function characterizations of the space $h_X(\rn)$
were obtained in \cite[Theorem 5.3]{shyy}.

\begin{lemma}\label{lgc}
Let $X$ be a ball quasi-Banach function space and $\Phi\in\mathcal{S}(\rn)$
satisfy $\int_{\rn}\Phi(x)\,dx\neq 0$. Assume that $r,\ a,\ A,\ b\in(0,\infty)$
satisfy $(b-A)r>n$. Let $N\in\nn$ satisfy $N \geq\lceil b+2\rceil$.
If $X$ is strictly $r$-convex and
there exists a positive constant $C$ such that, for any $z\in\rn$ and $f\in X$,
\begin{equation}\label{ix}
\left\|\left\{\int_{z+[0,1]^n}|f(\cdot-y)|^r\,dy\right\}^{1/r}\right\|_X
\le C (1+|z|)^A\|f\|_X,
\end{equation}
then
$$\lf\|m(f,\Phi)\r\|_X \sim\lf\|m^\ast_a(f,\Phi)\r\|_X \sim\lf\|m_b^{\ast\ast}(f,\Phi)\r\|_X
\sim\lf\|m_{b,N}^{\ast\ast}(f)\r\|_X,$$
where the positive equivalence constants are independent of $f$.
\end{lemma}

\begin{remark}\label{r4.1}
We point out that \eqref{ix} holds true for any ball quasi-Banach function space
$X$ satisfying  \eqref{ma}. Indeed, for any given $r\in(0,\infty)$ and any $f\in X$,  $x,\ z\in\rn$,
\begin{align*}
\left\{\int_{z+[0,1]^n}|f(x-y)|^r\,dy\right\}^{1/r}&=
\left\{\int_{x-z-[0,1]^n}|f(y)|^r\,dy\right\}^{1/r}\\
&\le \left\{\int_{B(x,\sqrt{n}(1+|z|))}|f(y)|^r\,dy\right\}^{1/r}\\
&\sim(1+|z|)^{\frac{n}{r}}\left\{\frac{1}{|B(x,\sqrt{n}(1+|z|))|}
\int_{B(x,\sqrt{n}(1+|z|))}|f(y)|^r\,dy\right\}^{1/r}\\
&\lesssim (1+|z|)^{\frac{n}{r}}M^{(r)}(f)(x),
\end{align*}
where the implicit positive constants are independent of $f$, $x$ and $z$,
but may depend on $n$ and $r$.
From this and the assumption that $X$ satisfies \eqref{ma},
we deduce that, for any given $r:=\theta$ and $A>\frac{n}{r}$ and any $f\in X$ and $z\in\rn$,
$$\left\|\left\{\int_{z+[0,1]^n}|f(\cdot-y)|^r\,dy\right\}^{1/r}\right\|_X\lesssim
(1+|z|)^{\frac{n}{r}}\left\|M^{(r)}(f)\right\|_X\lesssim (1+|z|)^A\|f\|_X,$$
where the implicit positive constants are independent of $f$ and $z$, but may depend on $n$ and $r$.
\end{remark}

Now we state our main result of this section as follows.

\begin{theorem}\label{pdo}
Let $X$ be a ball quasi-Banach function space satisfying
both \eqref{ma} and Assumption \ref{a2} with the same $s\in (0,1]$,
and having an absolutely continuous quasi-norm.
Assume that $T$ is an $S_{1,0}^0(\rn)$ pseudo-differential operator.
Then there exists a positive constant $C$ such that, for any $f\in h_X(\rn)$,
\begin{equation}\label{pdob}
\|T(f)\|_{h_X(\rn)}\le C\|f\|_{h_X(\rn)}.
\end{equation}
\end{theorem}

To show Theorem \ref{pdo}, we borrow some ideas from \cite{g} and \cite{ylk}.
We first need to establish an atomic characterization of $h_X(\rn)$.

\begin{definition}
Let $X$ be a ball quasi-Banach function space and $q \in [1, \infty]$.
Assume that $d \in \zz_+$ satisfies $d \geq d_X$, where $d_X$ is as in \eqref{dx}.
Then a measurable function $a$ is called a \emph{local-$(X, q, d)$-atom} if
\begin{enumerate}
\item[{\rm(i)}] there exists a cube $Q\subset\rn$ such that
$\supp (a):=\{x\in\rn:\ a(x)\neq0\} \subset Q$;
\item[{\rm(ii)}] $\|a\|_{L^q(\rn)} \le \frac{|Q|^{1/q}}{\|\mathbf{1}_Q\|_X}$;
\item[{\rm(iii)}] if $|Q| < 1$,
then $\int_{\rn}a(x)x^\alpha\,dx=0$ for any multi-index
$\alpha \in \zn_+$ with $|\alpha|\le d$.
\end{enumerate}
\end{definition}

The following lemma clarifies the relation between $H_X(\rn)$ and
$h_X(\rn)$ (see \cite[Lemma 5.5]{shyy} for details).

\begin{lemma}\label{rhh}
Let $X$ be a ball quasi-Banach function space satisfying
\eqref{ma}. Let $\psi \in \mathcal{S}(\rn)$ satisfy $\mathbf{1}_{Q({\vec 0_n},2)}
\leq\psi\leq\mathbf{1}_{Q({\vec 0_n},4)}$.
Then, for any $b \in (1,\infty)$ sufficiently large and $f \in \mathcal{S}'(\rn)$,
$$\|f\|_{h_X(\rn)} \sim\lf\|(\psi_1^\ast f)_b\r\|_X+\lf\|(1-\psi(D))f\r\|_{H_X(\rn)},$$
where $(\psi_1^\ast f)_b$ is as in \rm{\eqref{p}} and the positive equivalence
constants are independent of $f$ and $\psi$.
\end{lemma}

Based on Lemma \ref{rhh}, we have the following conclusion.

\begin{theorem}\label{ah}
Let $X$ be a ball quasi-Banach function space, satisfying
both \eqref{ma} and Assumption \ref{a2} with the same $s\in (0,1]$,
and $d\in \nn$ such that $d\geq d_X$, where $d_X$ is as in \eqref{dx}.
Then $f \in h_X(\rn)$
if and only if $f\in \mathcal{S}'(\rn)$ and there exist a sequence $\{a_j\}_{j=1}^\infty$ of local-$(X, \infty, d)$-atoms supported, respectively,  in cubes $\{Q_j\}_{j=1}^\infty$ and a sequence
$\{\lambda_j\}_{j=1}^\infty$ of non-negative numbers such that
\begin{equation}\label{lad}
f=\sum_{j=1}^\infty \lambda_j a_j \quad\quad \text{in}\quad \mathcal{S}'(\rn)
\end{equation}
and
\begin{equation}\label{lne}
\left\|\left[\sum_{j=1}^\infty\left(\frac{\lambda_j}{\|\mathbf{1}_{Q_j}\|_X}\right)^s
\mathbf{1}_{Q_j}\right]^{1/s}\right\|_X <\infty.
\end{equation}
Moreover,
$$\|f\|_{h_X(\rn)}\sim\inf\lf\{\left\|\left[\sum_{j=1}^\infty
\left(\frac{\lambda_j}{\|\mathbf{1}_{Q_j}\|_X}\right)^s
\mathbf{1}_{Q_j}\right]^{1/s}\right\|_X\r\},$$
where the infimum is taken over all decompositions of $f$ as in \eqref{lad}
and the positive equivalence constants are independent of $f$ but may depend on $s$.
\end{theorem}

\begin{proof}
Let $f \in h_X(\rn)$. Choose $\psi \in \mathcal{S}(\rn)$ such that
$$\mathbf{1}_{Q({\vec 0_n},2)}\leq\psi \le \mathbf{1}_{Q({\vec 0_n},4)}$$
and $\int_{\rn}\mathcal{F}^{-1}(\psi)(x)\,dx\neq 0$. Then
\begin{equation}\label{dof}
f=(1-\psi(D))f +\psi(D)f.
\end{equation}
By Lemma \ref{rhh}, we know that
$$\lf\|(\psi_1^\ast f)_b\r\|_X+\lf\|(1-\psi(D))f\r\|_{H_X(\rn)}
\sim \|f\|_{h_X(\rn)} < \infty,$$
which implies that $(1-\psi(D))f \in H_X(\rn)$.
From Lemma \ref{atom}, it follows that there exist a sequence
$\{a_{1,j}\}_{j=1}^\infty$ of $(X,\infty,d)$-atoms supported,
respectively, in cubes $\{Q_{1,j}\}_{j=1}^\infty$
and a sequence $\{\lambda_{1,j}\}_{j=1}^\infty$ of non-negative numbers such that
\begin{equation}\label{term11}
(1-\psi(D))f =\sum_{j=1}^\infty \lambda_{1,j} a_{1,j}
\quad\quad \text{in}\quad \mathcal{S}'(\rn)
\end{equation}
and
\begin{equation}\label{term12}
\left\|\left[\sum_{j=1}^\infty\left(\frac{\lambda_{1,j}}
{\|\mathbf{1}_{Q_{1,j}}\|_X}\right)^s
\mathbf{1}_{Q_{1,j}}\right]^{1/s}\right\|_X\lesssim
\|(1-\psi(D))f\|_{H_X(\rn)}\lesssim \|f\|_{h_X(\rn)}.
\end{equation}
Let $2\zz:=\{2z:\ z\in\zz\}$ and $\mathcal{Q}_1$ be the collection of all cubes in $\rn$
which are translations of $(0,2]^n$ and whose vertices lie on the lattice $(2\zz)^n$.
Take an arrangement of all cubes of $\mathcal{Q}_1$,
which is denoted by $\{Q_{2,j}\}_{j=1}^\infty$.
Then, for almost every $x\in\rn$, we have
$$\psi(D)f(x) = \sum_{j=1}^\infty\psi(D)f(x)\mathbf{1}_{Q_{2,j}}(x).$$
For any $j\in\nn$, if $\|\psi(D)f\|_{L^\infty(Q_{2,j})}=0$, define
$$\lambda_{2,j}:=0\quad \text{and}\quad a_{2,j}:=0;$$
if $\|\psi(D)f\|_{L^\infty(Q_{2,j})}\neq0$, define
$$\lambda_{2,j}:=\|\mathbf{1}_{Q_{2,j}}\|_X\|
\psi(D)f\|_{L^\infty(Q_{2,j})} \quad \text{and} \quad
a_{2,j}:=\frac{\psi(D)f\mathbf{1}_{Q_{2,j}}}
{\|\mathbf{1}_{Q_{2,j}}\|_X\|\psi(D)f\|_{L^\infty(Q_{2,j})}}.$$
Then, for almost every $x\in\rn$,
\begin{equation}\label{term21}
\psi(D)f(x)=\sum_{j=1}^\infty\lambda_{2,j}a_{2,j}(x).
\end{equation}
Since $|Q_{2,j}|>1$ for any $j\in\nn$, it follows that,
for any $j\in\nn$, $a_{2,j}$ is a local-$(X,\infty,d)$-atom.
Therefore, from \eqref{term11} and \eqref{term21}, we deduce that \eqref{lad} holds true.

We now prove \eqref{lne}. To show \eqref{lne}, by \eqref{term12}, it suffices to prove that
\begin{equation}\label{term22}
\left\|\left[\sum_{j=1}^\infty\left(\frac{\lambda_{2,j}}
{\|\mathbf{1}_{Q_{2,j}}\|_X}\right)^s
\mathbf{1}_{Q_{2,j}}\right]^{1/s}\right\|_X\lesssim \|f\|_{h_X(\rn)}.
\end{equation}
Indeed, for any fixed $x_0\in\rn$, there exists only one $j_0\in\zz_+$
such that $x_0 \in Q_{2,j_0}$. Moreover, there exists a positive
constant $c_0$ such that, for any $j\in\nn$ and $x\in Q_{2,j}$,
$Q_{2,j}\subset B(x,c_0)$.
Then, from the definition of $\{\lambda_{2,j}\}_{j=1}^\infty$, we deduce that
\begin{align}\label{10.31.1}
\left[\sum_{j=1}^\infty\left(\frac{\lambda_{2,j}}{\|\mathbf{1}_{Q_{2,j}}\|_X}\right)^s
\mathbf{1}_{Q_{2,j}}(x_0)\right]^{1/s}&=\frac{\lambda_{2,j_0}}
{\|\mathbf{1}_{Q_{2,j_0}}\|_X}
=\lf\|\psi(D)f\r\|_{L^\infty(Q_{2,j_0})}\notag\\
&\leq\sup_{y\in B(x_0,c_0)}\lf|(\mathcal{F}^{-1}\psi)\ast f(y)\r|\notag\\
&\leq\sup_{t\in(0,2)}\left\{\sup_{y\in B(x_0,c_0t)}
\lf|(\mathcal{F}^{-1}\psi)_t\ast f(y)\r|\right\}\notag\\
&\lesssim\sup_{t\in(0,1)}\left\{\sup_{y\in B(x_0,2c_0t)}
\lf|\widetilde{\psi}_t\ast f(y)\r|\right\},
\end{align}
where  $\widetilde{\psi}(\cdot):=\mathcal{F}^{-1}\psi(\frac{\cdot}{2})$.
By this and Lemma \ref{lgc}, we conclude that
$$\left\|\left[\sum_{j=1}^\infty\left(\frac{\lambda_{2,j}}
{\|\mathbf{1}_{Q_{2,j}}\|_X}\right)^s
\mathbf{1}_{Q_{2,j}}\right]^{1/s}\right\|_X \lesssim \lf\|m_{2c_0}^\ast
(f,\widetilde{\psi})\r\|_X\lesssim \lf\|m_{b,N}^{\ast\ast}(f)\r\|_X
\sim \|f\|_{h_X(\rn)},$$
which further implies that \eqref{term22} holds true.
Furthermore, from \eqref{term12} and \eqref{term22}, it follows that
$$\left\|\left[\sum_{i=1}^2\sum_{j=1}^\infty\left(\frac{\lambda_{i,j}}
{\|\mathbf{1}_{Q_{i,j}}\|_X}\right)^s
\mathbf{1}_{Q_{i,j}}\right]^{1/s}\right\|_X \lesssim\|f\|_{h_X(\rn)}.
$$

Conversely,  let $\{a_{1,j}\}_{j=1}^\infty\cup\{a_{2,j}\}_{j=1}^\infty$ be a
sequence of local-$(X, \infty, d)$-atoms  supported, respectively,
in cubes $\{Q_{1,j}\}_{j=1}^\infty\cup\{Q_{2,j}\}_{j=1}^\infty$
and $\{\lambda_{1,j}\}_{j=1}^\infty\cup\{\lambda_{2,j}\}_{j=1}^\infty$
a sequence of non-negative numbers such that
\begin{equation*}
\sum_{i=1}^{2}\sum_{j=1}^\infty \lambda_{i,j} a_{i,j} \quad\quad \text{converge in}\quad \mathcal{S}'(\rn)
\end{equation*}
and
\begin{equation*}
\left\|\left[\sum_{i=1}^2\sum_{j=1}^\infty\left(\frac{\lambda_{i,j}}
{\|\mathbf{1}_{Q_{i,j}}\|_X}\right)^s
\mathbf{1}_{Q_{i,j}}\right]^{1/s}\right\|_X<\infty,
\end{equation*}
where, for any $j\in\nn$,  $\ell(Q_{1,j})<1$ and
$\ell(Q_{2,j})\geq1$. Here, $\ell(Q)$ denotes the side length of the cube $Q$.

Let $\varphi \in \mathcal{S}(\rn)$, $\supp(\varphi) \subset B(\vec0_n,1)$ and
$\int_{\rn}\varphi(x)\,dx\neq0$. To finish the proof of this theorem,
by Lemma \ref{lgc}, it suffices to show
\begin{equation}\label{9.6.3}
\left\|m\left(\sum_{i=1}^2\sum_{j=1}^\infty \lambda_{i,j}
a_{i,j}, \varphi\right)\right\|_X\lesssim\left\|\left[\sum_{i=1}^2
\sum_{j=1}^\infty\left(\frac{\lambda_{i,j}}{\|\mathbf{1}_{Q_{i,j}}\|_X}\right)^s
\mathbf{1}_{Q_{i,j}}\right]^{1/s}\right\|_X.
\end{equation}
Observe that
\begin{align*}
\lf\|m\left(\sum_{i=1}^2\sum_{j=1}^\infty \lambda_{i,j}a_{i,j}, \varphi\right)\r\|_X&=
\lf\|\sup_{t\in(0,1)}\left|\varphi_t\ast\left(\sum_{i=1}^2\sum_{j=1}^\infty
\lambda_{i,j}a_{i,j}\right)\right|\r\|_X\\
&\ls\lf\|\sum_{j=1}^\infty \lambda_{1,j}\sup_{t\in(0,1)}|\varphi_t\ast a_{1,j}|\r\|_X
+\lf\|\sum_{j=1}^\infty \lambda_{2,j}\sup_{t\in(0,1)}|\varphi_t\ast a_{2,j}|\r\|_X\\
&=:\mathrm{I}_1+\mathrm{I}_2.
\end{align*}
By the facts that, for any $j\in\nn$, $a_{1,j}$ is also an $(X,\infty,d)$-atom and
$m(a_{1,j},\varphi)\leq M(a_{1,j},\varphi)$, together with \cite[Theorem 3.6]{shyy},
we know that
\begin{equation}\label{9.6.1}
\mathrm{I}_1\lesssim \left\|\left[\sum_{j=1}^\infty\left(
\frac{\lambda_{1,j}}{\|\mathbf{1}_{Q_{1,j}}\|_X}\right)^s
\mathbf{1}_{Q_{1,j}}\right]^{1/s}\right\|_X.
\end{equation}
Moreover, for any $j\in \nn$ and $t\in(0,1)$, from
$\supp(\varphi) \subset B(\vec0_n,1)$ and $\ell(Q_{2,j})\ge1$,
we deduce that $\supp(\varphi_t\ast a_{2,j})\subset 2Q_{2,j}$.
Let $q\in(1,\infty)$. Then, for any $t\in(0,1)$, by \cite[Theorem 2.1.10]{gl}
and the fact that $M$ is bounded on $L^q(\rn)$, we find that, for any $j\in\nn$,
\begin{equation*}
\left\|\mathbf{1}_{2Q_{2,j}}\sup_{t\in(0,1)}
\lf|\varphi_t\ast a_{2,j}\r|\right\|_{L^q(\rn)}
\lesssim\lf\|M(a_{2,j})\r\|_{L^q(\rn)} \lesssim \|a_{2,j}\|_{L^q(\rn)}
\lesssim \frac{|Q_{2,j}|^{1/q}}{\|\mathbf{1}_{Q_{2,j}}\|_X},
\end{equation*}
which, combined with Lemma \ref{r}, further implies that
\begin{equation}\label{9.6.2}
\mathrm{I}_2\lesssim \left\|\left[\sum_{j=1}^\infty
\left(\frac{\lambda_{2,j}}{\|\mathbf{1}_{Q_{2,j}}\|_X}\right)^s
\mathbf{1}_{Q_{2,j}}\right]^{1/s}\right\|_X.
\end{equation}
Combining \eqref{9.6.1} and \eqref{9.6.2}, we know that
\eqref{9.6.3} holds true. This finishes the proof of Theorem \ref{ah}.
\end{proof}

\begin{remark}\label{l2}
Let $f \in L^2(\rn) \cap h_X(\rn)$ and $\psi$ be as in Theorem \ref{ah}.
It is easy to see that $\psi(D)f \in L^2(\rn)$. Then, by Lemma \ref{rhh},
we further know that $(1-\psi(D))f \in L^2(\rn)\cap H_X(\rn)$, which,
together with \eqref{de} and \eqref{term21}, further
implies that \eqref{lad} also holds true in $L^2(\rn)$.
\end{remark}

Combining Theorem \ref{ah} and the fact that $L^2(\rn) \cap H_X(\rn)$
is dense in $H_X(\rn)$, we obtain the following conclusion.

\begin{corollary}\label{ld}
Let $X$ be as  in Theorem {\rm\ref{ah}} and have an absolutely continuous quasi-norm.
Then $L^2(\rn)\cap h_X(\rn)$
is dense in $h_X(\rn)$.
\end{corollary}

\begin{proof}
Let $f \in h_X(\rn)$. Choose $\Phi \in \mathcal{S}(\rn)$ satisfying
$\int_{\rn} \Phi(x)\,dx\neq 0$.
By Theorem \ref{ah}, we conclude that
$$f=(1-\psi(D))f+\psi(D)f=\sum_{j=1}^\infty \lambda_{1,j} a_{1,j}
+\sum_{j=1}^\infty\lambda_{2,j}a_{2,j},$$
where all the symbols are the same as in the proof of Theorem \ref{ah}.
For any $n \in \nn$, let
$$f_n:=\sum_{j=1}^n\lambda_{1,j}a_{1,j}+\sum_{j=1}^n\lambda_{2,j}a_{2,j}.$$
It is easy to see that, for any $n\in\nn$, $f_n\in L^2(\rn)\cap h_X(\rn)$.
To finish the proof of  this corollary,  it suffices to show that
\begin{equation}\label{lim}
\lim_{n\to\infty}\lf\|f-f_n\r\|_{h_X(\rn)}=0.
\end{equation}
Indeed, from \eqref{dof}, \eqref{term11} and \eqref{term21}, it follows that
\begin{equation}\label{don}
\|f-f_n\|_{h_X(\rn)}\le\left\|(1-\psi(D))f-\sum_{j=1}^n \lambda_{1,j}
a_{1,j}\right\|_{h_X(\rn)}+\left\|\sum_{j=n+1}^\infty
\lambda_{2,j} a_{2,j}\right\|_{h_X(\rn)}=:\mathrm{J}_1+\mathrm{J}_2.
\end{equation}
By Lemma \ref{lgc}, we find that
\begin{align*}
\mathrm{J}_1&\sim \left\|\sup_{t\in(0,1)}\left|\Phi_t\ast\left[(1-\psi(D))f-
\sum_{j=1}^n \lambda_{1,j} a_{1,j}\right]\right|\right\|_X\\
&\lesssim \left\|\sup_{t\in(0,\infty)}\left|\Phi_t\ast\left[(1-\psi(D))f-
\sum_{j=1}^n \lambda_{1,j} a_{1,j}\right]\right|\right\|_X\\
&\sim \left\|(1-\psi(D))f-\sum_{j=1}^n \lambda_{1,j} a_{1,j}\right\|_{H_X(\rn)},
\end{align*}
which, combined with the definitions of $\lambda_{1,j}$ and $a_{1,j}$
and \cite[Remark 3.12]{shyy}, further implies that
$\mathrm{J}_1\to 0$ as $n\to\infty$.

From Lemma \ref{lgc}, we deduce that
$$\mathrm{J}_2\sim\left\|\sup_{t\in(0,1)}\left|\Phi_t\ast
\left(\sum_{j=n+1}^\infty \lambda_{2,j} a_{2,j}\right)\right|
\right\|_X\lesssim\left\|\sum_{j=n+1}^\infty \lambda_{2,j}\sup_{t\in(0,1)}
|\Phi_t\ast a_{2,j}|\right\|_X.$$
Since, for any $j\in\nn$, $\ell(Q_{2,j})=2$, it follows that,
for any $t\in(0,1)$ and $j\in\nn$,
$$\supp(\Phi_t\ast a_{2,j}) \subset 2Q_{2,j}.$$
Let $q\in(1,\infty)$. Then, for any $t\in(0,1)$, by \cite[Theorem 2.1.10]{gl}
and the fact that $M$ is bounded on $L^q(\rn)$, we find that
\begin{equation}\label{rmf}
\left\|\mathbf{1}_{Q_{2,j}}\sup_{t\in(0,1)}\lf|\Phi_t\ast a_{2,j}\r|\right\|_{L^q(\rn)}
\lesssim\lf\|M(a_{2,j})\r\|_{L^q(\rn)} \lesssim \|a_{2,j}\|_{L^q(\rn)}
\lesssim \frac{|Q_{2,j}|^{1/q}}{\|\mathbf{1}_{Q_{2,j}}\|_X},
\end{equation}
which, together with Lemma \ref{r}, implies that
$$\mathrm{J}_2\lesssim \left\|\left[\sum_{j=n+1}^\infty
\left(\frac{\lambda_{2,j}}{\|\mathbf{1}_{Q_{2,j}}\|_X}\right)^s
\mathbf{1}_{Q_{2,j}}\right]^{1/s}\right\|_X.$$

From this, \eqref{term22} and the fact that $X$ has an absolutely continuous quasi-norm,
we deduce that $\mathrm{J}_2\to 0$
as $n\to\infty$,
which, combined with \eqref{don} and the fact that ${\rm J}_1\to 0$
as $n \to \infty$, further implies that
\eqref{lim} holds true. This finishes the proof of Corollary \ref{ld}.
\end{proof}

To prove Theorem \ref{pdo}, we need the following property of pseudo-differential operators,
which was obtained in \cite[Lemma 6]{g}.

\begin{lemma}\label{kop}
Let $T$ be an $S_{1,0}^0(\rn)$ pseudo-differential operator.
If $\psi\in\mathcal{S}(\rn)$, then $T_tf:=\psi_t\ast Tf$
has a symbol $\sigma_t$ satisfying that, for any $t\in(0,1)$ and $x,\ \xi\in\rn$,
 $$\lf|\partial_x^\beta\partial_\xi^\alpha\sigma_t(x,\xi)\r|\le
 C_{(\alpha, \beta)}(1+|\xi|)^{-|\alpha|},$$
 and a kernel $K_t$ satisfying that, for any $t\in(0,1)$ and $x,\ \xi\in\rn$,
 $$\lf|\partial_x^\beta\partial_\xi^\alpha K_t(x,\xi)\r|
 \le C_{(\alpha, \beta)}|\xi|^{-n-|\alpha|},$$
where $C_{(\alpha,\beta)}$ is a positive constant independent
of $t$, $x$ and $\xi$, but depending on $\alpha$ and $\beta$.
\end{lemma}

We now show Theorem \ref{pdo} by using Lemma \ref{kop}.

\begin{proof}[Proof of Theorem \ref{pdo}]
Let $f\in L^2(\rn)\cap h_X(\rn)$ and $d$, $s$ be as in Theorem \ref{ah}. Then,
by Theorem \ref{ah} and Remark \ref{l2}, we know that there
exist a sequence $\{a_j\}_{j=1}^\infty$ of local-$(X,\infty,d)$-atoms supported, respectively,
in cubes $\{Q_j\}_{j=1}^\infty$ and
a sequence $\{\lambda_j\}_{j=1}^\infty$ of non-negative numbers such that
\begin{equation}\label{lad2}
f=\sum_{j=1}^\infty \lambda_j a_j \quad\quad \text{in}\quad L^2(\rn)
\end{equation}
and
\begin{equation*}
\left\|\left[\sum_{j=1}^\infty\left(\frac{\lambda_j}{\|\mathbf{1}_{Q_j}\|_X}\right)^s
\mathbf{1}_{Q_j}\right]^{1/s}\right\|_X\lesssim\|f\|_{h_X(\rn)}.
\end{equation*}
Since the operator $T$ is bounded on $L^2(\rn)$, it follows that
$$Tf = \sum_{j=1}^\infty\lambda_jTa_j$$
holds true in $L^2(\rn)$. Choose $\Phi\in\mathcal{S}(\rn)$
satisfying that $\supp(\Phi)\subset B({\vec 0_n},1)$ and
$\int_{\rn} \Phi(x)\,dx\neq 0$.  Then, by Lemma \ref{lgc},
Corollary \ref{ld} and a dense argument, to prove this theorem, we only need to show that
\begin{equation}\label{mf}
\left\|\sup_{t\in(0,1)}|\Phi_t\ast Tf|\right\|_X\lesssim \|f\|_{h_X(\rn)}.
\end{equation}
Indeed, from \eqref{lad2}, we deduce that
\begin{align*}
\left\|\sup_{t\in(0,1)}|\Phi_t\ast Tf|\right\|_X &=\left\|\sup_{t\in(0,1)}\left|\Phi_t\ast
\left(\sum_{j=1}^\infty\lambda_jTa_j\right)\right|\right\|_X
\lesssim \left\|\sum_{j=1}^\infty \lambda_j\sup_{t\in(0,1)}|\Phi_t\ast Ta_j|\right\|_X\\
&\lesssim  \left\|\sum_{j=1}^\infty \lambda_j\mathbf{1}_{4Q_j}
\sup_{t\in(0,1)}|\Phi_t\ast Ta_j|\right\|_X +  \left\|\sum_{j=1}^\infty
\lambda_j\mathbf{1}_{(4Q_j)^\complement}\sup_{t\in(0,1)}|\Phi_t\ast Ta_j|\right\|_X\\
&=:\mathrm{J}_1+\mathrm{J}_2.
\end{align*}
Similarly to \eqref{rmf}, we obtain, for any $j\in\nn$,
$$\left\|\mathbf{1}_{4Q_j}\sup_{t\in(0,1)}|\Phi_t\ast a_j|\right\|_{L^q(\rn)}
\lesssim \frac{|Q_j|^{1/q}}{\|\mathbf{1}_{Q_j}\|_X},$$
which, together with Lemma \ref{r} and \eqref{lne}, implies that
\begin{equation}\label{j1}
\mathrm{J}_1\lesssim \left\|\left[\sum_{j=1}^\infty
\left(\frac{\lambda_j}{\|\mathbf{1}_{Q_j}\|_X}\right)^s
\mathbf{1}_{Q_j}\right]^{1/s}\right\|_X\lesssim \|f\|_{h_X(\rn)}.
\end{equation}

Moreover, if $a_j$ has the vanishing moments, namely,
$\int_{\rn}a(x)x^\alpha\,dx=0$ for any multi-index $\alpha \in \zn_+$
with $|\alpha|\le d$, by Lemma \ref{kop}, the Taylor expansion,
the fact that $d\geq d_X \geq n(1/\theta-1)-1$ and an argument
similar to that used in \cite[p.\,76]{yy},
we conclude that, for any $x\in(4Q_j)^\complement$,
\begin{align}\label{vc}
\sup_{t\in(0,1)}|\Phi_t\ast Ta_j(x)|&\lesssim |x-x_j|^{-(n+d+1)}
|Q_j|^{\frac{d+1}{n}}\|a_j\|_{L^1(\rn)}\notag\\
&\lesssim \frac{1}{\|\mathbf{1}_{Q_j}\|_X}\lf[M(\mathbf{1}_{Q_j})
(x)\r]^{\frac{n+d+1}{n}}\lesssim \frac{1}{\|\mathbf{1}_{Q_j}\|_X}
M^{(\theta)}(\mathbf{1}_{Q_j})(x),
\end{align}
where $x_j$ denotes the center of $Q_j$. If $a_j$ has no the vanishing moment,
then, from the proof of Theorem \ref{ah}, we deduce that, for any such $j$, $|Q_j|=2^n >1$;
in this case, similarly to the proof of \cite[(8.44)]{yy}, we know that, for any $N\in\nn$ and $x\in(4Q_j)^\complement$,
$$\sup_{t\in(0,1)}\lf|\Phi_t\ast Ta_j(x)\r|\lesssim|x-x_j|^{-N}\|a_j\|_{L^1(\rn)},$$
where the implicit positive constant may depend on $N$. Choosing $N >n+d+1$, we then
have, for any $x\in\rn$,
\begin{equation}\label{nvc}
\sup_{t\in(0,1)}\lf|\Phi_t\ast Ta_j(x)\r|\lesssim|x-x_j|^{-(n+d+1)}
|Q_j|^{\frac{d+1}{n}}\|a_j\|_{L^1(\rn)}
\lesssim \frac{1}{\|\mathbf{1}_{Q_j}\|_X}M^{(\theta)}(\mathbf{1}_{Q_j})(x),
\end{equation}
which, together with \eqref{vc}, \eqref{nvc}, \eqref{ma} and \eqref{lne}, implies that
$$\mathrm{J}_2\lesssim \left\|\sum_{j=1}^\infty\frac{\lambda_j}
{\|\mathbf{1}_{Q_j}\|_X}M^{(\theta)}
(\mathbf{1}_{Q_j})\right\|_X\lesssim\left\|\left[\sum_{j=1}^\infty
\left(\frac{\lambda_j}{\|\mathbf{1}_{Q_j}\|_X}\right)^s
\mathbf{1}_{Q_j}\right]^{1/s}\right\|_X\lesssim \|f\|_{h_X(\rn)}.$$
From this and \eqref{j1}, it follows that \eqref{pdob} holds true
for any $f\in L^2(\rn)\cap h_X(\rn)$,
which, combined  with Corollary \ref{ld} and a dense argument,
further completes the proof of Theorem \ref{pdo}.
\end{proof}

\begin{remark}
Recently,  Abl\'e and Feuto \cite{af} obtained the boundedness of
pseudo-differential operators
on Hardy-amalgam spaces $\mathcal{H}^{(p,q)}_\loc(\rn)$. Compared with the results in \cite{af}
and the results for the classical local Hardy space $h^p(\rn)$,
the difficulty to show the above results is that, to
estimate the $h_X(\rn)$-norm of $Tf$, we cannot  just only show that,
for any local-$(X, \infty, d)$-atom $a$, $\|Ta\|_{h_X(\rn)}\lesssim 1$
and we have to estimate $\|Tf\|_{h_X(\rn)}$ directly.
This problem is caused by Assumption \ref{a} and we do not
have a concrete form of the $h_X(\rn)$-norm.
\end{remark}

Based on Theorem \ref{pdo}, we can weaken the assumption
that $X$ has an absolutely continuous quasi-norm into a weaker assumption
which is applicable to Morrey spaces.

\begin{theorem}\label{pdo-new}
Assume that $X$ and  $Y$ are two ball quasi-Banach function spaces satisfying both
\eqref{ma} and Assumption \ref{a2} with the same $s\in (0,1]$.
Moreover, assume that $Y$ has an absolutely continuous quasi-norm
and $X$ continuously embeds into $Y$.
Assume that $T$ is an $S_{1,0}^0(\rn)$ pseudo-differential operator.
Then there exists a positive constant $C$ such that, for any $f\in h_X(\rn)$,
\begin{equation*}
\|T(f)\|_{h_X(\rn)}\le C\|f\|_{h_X(\rn)}.
\end{equation*}
\end{theorem}

\begin{proof}
By Theorem \ref{pdo} and an argument similar to that used in the proof of Theorem \ref{ccz-new},
we can obtain the desired conclusion of this theorem and we omit the details.
This finishes the proof of Theorem \ref{pdo-new}.
\end{proof}

In what follows, we denote the classical \emph{local Hardy space},
the \emph{local weighted Hardy space}, the \emph{local Herz--Hardy space},
the \emph{local Lorentz--Hardy space}, the \emph{local Morrey--Hardy space},
the \emph{local variable Hardy space},
the \emph{local Orlicz--Hardy space} and the \emph{local Orlicz-slice Hardy space}, respectively, by
$$h^p(\rn),\quad h^p_w(\rn),\quad h\ck^\alpha_{p,q}(\rn),\quad h^{p,q}(\rn),\quad
h\cm^p_q(\rn),\quad h^{p(\cdot)}(\rn),\quad h^\Phi(\rn)$$
and $(hE_\Phi^r)_t(\rn)$.

As the corollaries of Theorem \ref{pdo-new} and Remarks \ref{r2.1}, \ref{r2.2} and \ref{morr-lw},
we have the following conclusions.

\begin{corollary}\label{c4.1}
Assume that $T$ is an $S_{1,0}^0(\rn)$ pseudo-differential operator.
\begin{enumerate}
\item[{\rm(a)}] For any $p\in(0,\infty)$,
there exists a positive constant $C$ such that, for any $f\in h^p(\rn)$,
$\|T(f)\|_{h^p(\rn)}\leq C\|f\|_{h^p(\rn)}.$

\item[{\rm(b)}] For any  $p\in(0,1]$ and $w\in A_\fz(\rn)$,
there exists a positive constant $C$ such that,
for any $f\in h^p_w(\rn)$,
$\|T(f)\|_{h^p_w(\rn)}\leq C\|f\|_{h^p_w(\rn)}.$

\item[{\rm(c)}] For any  $p,\ r\in(0,\infty)$ and $\alpha\in(-n/p,\infty)$,
there exists a positive constant $C$ such that,
for any $f\in h\ck^\alpha_{p,r}(\rn)$,
$\|T(f)\|_{h\ck^\alpha_{p,r}(\rn)}\leq C\|f\|_{h\ck^\alpha_{p,r}(\rn)}.$

\item[{\rm(d)}] For any $p,\ r\in(0,\infty)$,
there exists a positive constant $C$ such that,
for any $f\in h^{p,r}(\rn)$,
$\|T(f)\|_{h^{p,r}(\rn)}\leq C\|f\|_{h^{p,r}(\rn)}.$

\item[{\rm(e)}] For any  $p\in(0,\infty)$ and $r\in(0,p]$,
there exists a positive constant $C$ such that,
for any $f\in h{\mathcal M}^p_r(\rn)$,
$\|T(f)\|_{h{\mathcal M}^p_r(\rn)}\leq C\|f\|_{h{\mathcal M}^p_r(\rn)}.$

\item[{\rm(f)}] If $p(\cdot)$ is globally log-H\"older continuous,
then there exists a positive constant $C$ such that,
for any $f\in h^{p(\cdot)}(\rn)$,
$\|T(f)\|_{h^{p(\cdot)}(\rn)}\leq C\|f\|_{h^{p(\cdot)}(\rn)}.$

\item[{\rm(g)}] Assume that $\Phi$ is an Orlicz function with lower type
$p_\Phi^- \in (0,1)$ and upper type $p_\Phi^+ =1$.
Then there exists a positive constant $C$ such that,
for any $f\in h^\Phi(\rn)$,
$\|T(f)\|_{h^\Phi(\rn)}\leq C\|f\|_{h^\Phi(\rn)}.$

\item[\rm(h)] Let $r\in(0,\fz)$ and $\Phi$ be an Orlicz function
with $p_\Phi^-,\ p_\Phi^+\in(0,\fz)$.
Then there exists a positive constant $C$ such that,
for any $t\in(0,\fz)$ and $f\in (hE_\Phi^r)_t(\rn)$,
$\|T(f)\|_{(hE_\Phi^r)_t(\rn)}\leq C\|f\|_{(hE_\Phi^r)_t(\rn)}.$
\end{enumerate}
\end{corollary}

\begin{remark}\label{r4.2} Corollary \ref{c4.1}(a) is just \cite[Theorem 4]{g}.
Moreover, Corollary \ref{c4.1}(b) was obtained in \cite[Theorem 1]{m91}.
Furthermore, Corollary \ref{c4.1}(c) was established in \cite[Theorems 2.4 and 2.5]{fy00}.
Corollary \ref{c4.1}(f) was obtained in \cite[Theorem 1.2]{ks13}.
Moreover, Corollary \ref{c4.1}(g)
was established in \cite[Theorem 8.18]{yy}.
In the case of Morrey spaces, Corollary \ref{c4.1}(e)  is a consequence of \cite[Theorem 1.1 and Proposition 1.5]{s09}.
Furthermore, it is worth pointing out that (d) and (h) of Corollary \ref{c4.1} are new.
\end{remark}

\section{Characterizations of $h_X(\rn)$ via molecules and Littlewood--Paley functions
}\label{s5}

In this section, we establish a molecular characterization
and Littlewood--Paley characterizations of the local Hardy space $h_X(\rn)$.
We begin with the notions of local molecules.

\begin{definition}
Let $X$ be a ball quasi-Banach function space,
$q\in[1,\infty]$, $d\in\zz_+$ and $\tau\in(0,\infty)$.
A measurable function $m$ on $\rn$ is called a \emph{local-$(X,q,d,\tau)$-molecule}
centered at a cube $Q\subset\rn$ if
\begin{enumerate}
\item[{\rm(i)}] $\|m\mathbf{1}_Q\|_{L^q(\rn)}\le|Q|^{1/q}\|\mathbf{1}_Q\|_X^{-1}$;
\item[{\rm(ii)}] for any $j\in\nn$,
$$\lf\|m\mathbf{1}_{S_j(Q)}\r\|_{L^q(\rn)}\le2^{-\tau j}
|Q|^{1/q}\|\mathbf{1}_Q\|_X^{-1},$$
where $S_j(Q):=(2^{j}Q)\backslash(2^{j-1}Q)$;
\item[{\rm(iii)}] if $|Q|<1$, then, for any multi-index $\alpha\in\zz^n_+$
with $|\alpha|\le d$, $\int_{\rn}m(x)x^\alpha\,dx=0$.
\end{enumerate}
\end{definition}

Using the atomic characterization of $h_X(\rn)$, we now obtain
the following local-$(X,q,d,\tau)$-molecular characterization of $h_X(\rn)$.

\begin{theorem}\label{mech}
Let $X$ be a ball quasi-Banach function space
satisfying both \eqref{ma} with $0<\theta<s\le 1$ and Assumption \ref{a2}
with some $q\in(1,\fz]$ and the same $s$ as in \eqref{ma}.
Assume that $d\in\zz_+$ with $d\ge d_X$ and $\tau\in(0,\infty)$ satisfying
$\tau > n(1/\theta-1/q)$, where $d_X$ is as in \eqref{dx}.
Then $f\in h_X(\rn)$ if and only if there exist a sequence $\{m_j\}_{j=1}^\infty$
of local-$(X,q,d,\tau)$-molecules centered,  respectively, at cubes $\{Q_j\}_{j=1}^\infty$ and a sequence
$\{\lambda_j\}_{j=1}^\infty$ of non-negative numbers such that
\begin{equation}\label{lmd}
f=\sum_{j=1}^\infty \lambda_j m_j \quad\quad \text{in}\quad \mathcal{S}'(\rn)
\end{equation}
and
\begin{equation*}
\left\|\left[\sum_{j=1}^\infty\left(\frac{\lambda_j}{\|\mathbf{1}_{Q_j}\|_X}\right)^s
\mathbf{1}_{Q_j}\right]^{1/s}\right\|_X <\infty.
\end{equation*}
Moreover,
$$\|f\|_{h_X(\rn)}\sim\inf\lf\{\left\|\left[\sum_{j=1}^\infty
\left(\frac{\lambda_j}{\|\mathbf{1}_{Q_j}\|_X}\right)^s
\mathbf{1}_{Q_j}\right]^{1/s}\right\|_X\r\},$$
where the infimum is taken over all the decompositions of $f$ as in \eqref{lmd}
and the positive equivalence constants are independent of $f$ but may depend on $n$ and $s$.
\end{theorem}

\begin{proof}
The necessity part of Theorem \ref{mech} is obtained by Theorem \ref{ah}
and the fact that a local-$(X,q,d)$-atom is also a local-$(X,q,d,\tau)$-molecule
for any $q\in(1,\fz]$, $d\in\zz_+$ and $\tau\in(0,\infty)$; the details are omitted here.

Now we turn to the sufficiency part of Theorem \ref{mech}.
Let $f=\sum_{i=1}^2\sum_{j=1}^\infty \lambda_{i,j} m_{i,j}$ in $\cs'(\rn)$,
where $\{m_{1,j}\}_{j=1}^\infty\cup\{m_{2,j}\}_{j=1}^\infty$ is a
sequence of local-$(X,q,d,\tau)$-molecules  centered, respectively,
at cubes $\{Q_{1,j}\}_{j=1}^\infty\cup\{Q_{2,j}\}_{j=1}^\infty$,
and $\{\lambda_{1,j}\}_{j=1}^\infty\cup\{\lambda_{2,j}\}_{j=1}^\infty$
is a sequence of non-negative numbers satisfying that
\begin{equation*}
\left\|\left[\sum_{i=1}^2\sum_{j=1}^\infty\left(\frac{\lambda_{i,j}}
{\|\mathbf{1}_{Q_{i,j}}\|_X}\right)^s
\mathbf{1}_{Q_{i,j}}\right]^{1/s}\right\|_X<\infty.
\end{equation*}
Here, for any $j\in\nn$, $\ell(Q_{1,j})<1$ and $\ell(Q_{2,j})\geq1$
with $\ell(Q)$ denoting the side length of the cube $Q$.

By the observation that, for any $j\in\nn$, $m_{1,j}$ is also an
$(X,q,d,\tau)$-molecule, and the molecular characterization of $H_X(\rn)$
(see \cite[Theorem 3.9]{shyy}), we find that
\begin{align}\label{10.29.1}
\left\|\sum_{j=1}^\infty\lambda_{1,j}m_{1,j}\right\|_{h_X(\rn)}
&\le\left\|\sum_{j=1}^\infty\lambda_{1,j}m_{1,j}\right\|_{H_X(\rn)}\notag\\
&\ls \left\|\left[\sum_{j=1}^\infty\left(\frac{\lambda_{1,j}}
{\|\mathbf{1}_{Q_{1,j}}\|_X}\right)^s
\mathbf{1}_{Q_{1,j}}\right]^{1/s}\right\|_X
<\infty.
\end{align}
Moreover, for any $j\in\nn$, let $m_{2,j}^{(0)}:=m_{2,j}\mathbf{1}_{Q_{2,j}}$
and, for any $k\in\nn$,
$$m_{2,j}^{(k)}:=2^{k(\tau+n/q)}\frac{\|\mathbf{1}_{Q_{2,j}}\|_X}
{\|\mathbf{1}_{2^kQ_{2,j}}\|_X}m_{2,j}\mathbf{1}_{S_k(Q_{2,j})}.$$
Then it is easy to show that, for any $j\in\nn$ and $k\in\zz_+$, $m_{2,j}^{(k)}$ is a
local-$(X,q,d)$-atom supported in $2^kQ_{2,j}$, and
$$\sum_{j=1}^\infty\lambda_{2,j}m_{2,j}=\sum_{j=1}^\infty\sum_{k=0}^\infty
\lambda_{2,j}2^{-k(\tau+n/q)}\frac{\|\mathbf{1}_{2^kQ_{2,j}}\|_X}
{\|\mathbf{1}_{Q_{2,j}}\|_X}m_{2,j}^{(k)},$$
which, combined with Theorem \ref{ah} and $\tau > n(1/\theta-1/q)$, further implies that
\begin{align*}
\left\|\sum_{j=1}^\infty\lambda_{2,j}m_{2,j}\right\|_{h_X(\rn)}&
\lesssim\left\|\left[\sum_{j=1}^\infty\sum_{k=0}^\infty\left(\lambda_{2,j}
2^{-k(\tau+n/q)}\frac{1}{\|\mathbf{1}_{Q_{2,j}}\|_X}\right)^s
\mathbf{1}_{2^kQ_{2,j}}\right]^{1/s}\right\|_X\\
&\lesssim \left\|\left[\sum_{j=1}^\infty\sum_{k=0}^\infty2^{-k(\tau+n/q)s+kns/\theta}
\left(\frac{\lambda_{2,j}}{\|\mathbf{1}_{Q_{2,j}}\|_X}\right)^s
M^{(\theta)}\left(\mathbf{1}_{Q_{2,j}}\right)\right]^{1/s}\right\|_X\\
&\lesssim \left\|\left[\sum_{j=1}^\infty\left(\frac{\lambda_{2,j}}
{\|\mathbf{1}_{Q_{2,j}}\|_X}\right)^s\mathbf{1}_{Q_{2,j}}\right]^{1/s}\right\|_X<\infty.
\end{align*}
From this and \eqref{10.29.1}, it follows that $f\in h_X(\rn)$ and
$$\|f\|_{h_X(\rn)}\lesssim\left\|\left[\sum_{i=1}^2
\sum_{j=1}^\infty\left(\frac{\lambda_{i,j}}
{\|\mathbf{1}_{Q_{i,j}}\|_X}\right)^s
\mathbf{1}_{Q_{i,j}}\right]^{1/s}\right\|_X<\infty,$$
which completes the proof of Theorem \ref{mech}.
\end{proof}

Now we give the notions of local Littlewood--Paley functions as follows.

\begin{definition}
Let $\psi,\ \varphi\in\mathcal{S}(\rn)$ satisfy $\mathbf{1}_{Q({\vec 0_n},2)}
\le\psi\le\mathbf{1}_{Q({\vec 0_n},4)}$ and $\mathbf{1}_{Q(\vec{0}_n,4)\setminus
Q(\vec{0}_n,2)}\le\varphi\le\mathbf{1}_{Q(\vec{0}_n,8)\setminus Q(\vec{0}_n,1)}$,
$b\in (1,\infty)$ and $\lz\in(0,\fz)$. Then, for any $f\in\mathcal{S}'(\rn)$,
the \emph{local Lusin-area function $S_l(f)$, $g$-function $g_l(f)$} and
\emph{$g^\ast_\lambda$-function $(g^\ast_\lambda)_l(f)$} are
defined, respectively, by setting, for any $x\in\rn$,
\begin{equation*}
S_l(f)(x):=\left\{\int_0^1\int_{B(x,t)}
\lf|\varphi(tD)(f)(y)\r|^2\,\frac{dydt}{t^{n+1}}\right\}^{\frac{1}{2}}
+(\psi_1^\ast f)_b(x),
\end{equation*}
\begin{equation*}
g_l(f)(x):=\left\{\int_0^1\lf|\varphi(tD)(f)(x)\r|^2\,\frac{dt}{t}\right\}^\frac{1}{2}
+(\psi_1^\ast f)_b(x)
\end{equation*}
and
\begin{equation*}
(g_\lambda^\ast)_l(f)(x):=\left\{\int_0^1\int_{\rn}\left(\frac{t}{t+|x-y|}\right)^{\lambda n}
\lf|\varphi(tD)(f)(x)\r|^2\,\frac{dydt}{t^{n+1}}\right\}^\frac{1}{2}+(\psi_1^\ast f)_b(x),
\end{equation*}
where $(\psi_1^\ast f)_b$ is as in \eqref{p}.
\end{definition}

Then we have the following characterizations of $h_X(\rn)$
via local Littlewood--Paley functions.

\begin{theorem}\label{lpl}
Let $\psi,\ \varphi\in \mathcal{S}(\rn)$
satisfy
$$\mathbf{1}_{Q({\vec 0_n},2)}\le\psi\le\mathbf{1}_{Q({\vec 0_n},4)}\quad \text{and} \quad
\mathbf{1}_{Q(\vec{0}_n,4)\setminus Q(\vec{0}_n,2)}\le\varphi\le
\mathbf{1}_{Q(\vec{0}_n,8)\setminus Q(\vec{0}_n,1)}.$$
Assume that $X$ is a ball quasi-Banach function space satisfying
\eqref{ma} with $0<\theta<s\le 1$, Assumption \ref{a2}
with some $q\in(1,\fz]$ and the same $s$ as in \eqref{ma}, and \eqref{ix}.
Then
\begin{itemize}
\item[\rm(i)] $f\in h_X(\rn)$ if and only if $f \in \mathcal{S}'(\rn)$ and
$\|S_l(f)\|_X<\infty$. Moreover, for any $f\in h_X(\rn)$,
$\|f\|_{h_X(\rn)}\sim\|S_l(f)\|_X$ with the positive equivalence
constants independent of $f$.

\item[\rm(ii)] If $X$ further satisfies \eqref{ma2} with the same $\theta$ and $s$ as in \eqref{ma}, then
$f\in h_X(\rn)$ if and only if $f \in \mathcal{S}'(\rn)$ and
$\|g_l(f)\|_X<\infty$. Moreover, for any $f\in h_X(\rn)$,
$\|f\|_{h_X(\rn)}\sim\|g_l(f)\|_X$ with the positive equivalence
constants  independent of $f$.

\item[\rm(iii)] Let $\lambda\in(\max\{2/\theta,2/\theta+(1-2/q)\},\fz)$.
Then $f\in h_X(\rn)$ if and only if $f \in \mathcal{S}'(\rn)$ and
$\|(g_\lambda^\ast)_l(f)\|_X<\infty$. Moreover, for any $f\in h_X(\rn)$,
$\|f\|_{h_X(\rn)}\sim\|(g_\lambda^\ast)_l(f)\|_X$ with the positive equivalence
constants  independent of $f$.
\end{itemize}
\end{theorem}

To prove Theorem \ref{lpl}, we need the following several lemmas.

\begin{lemma}\label{10.29.2}
Let all the notation be the same as in Theorem \ref{lpl}. If $f\in \cs'(\rn)$ satisfies $\|S_l(f)\|_X<\fz$,
then $f\in h_X(\rn)$ and there exists a positive constant $C$, independent of $f$, such that
$\|f\|_{h_X(\rn)}\le C\|S_l(f)\|_X$.\end{lemma}

\begin{proof}
Let $f\in \cs'(\rn)$ satisfy $\|S_l(f)\|_X<\infty$.
Then, by \cite[Proposition 1.1.6]{gl14}, we know that there exist
$\phi,\ \eta\in\cs(\rn)$, with $\vec 0_n\notin \supp(\cf^{-1}(\phi))$ and
$B(\vec 0_n,1)\subset \supp(\cf^{-1}(\eta))$, such that
\begin{equation}\label{11.5a}
f=\eta(D)\psi(D)(f)+\int_0^1\phi(tD)\varphi(tD)(f)\,\frac{dt}{t}\quad\text{in}\quad \mathcal{S}'(\rn).
\end{equation}
Let $2\zz:=\{2z:\ z\in\zz\}$ and $\mathcal{Q}_1$ be the collection of all cubes in $\rn$
which are translations of $(0,2]^n$ and whose vertices lie on the lattice $(2\zz)^n$.
Take an arrangement of all cubes of $\mathcal{Q}_1$,
which is denoted by $\{Q_{2,j}\}_{j=1}^\infty$.
Then, for almost every $x\in\rn$,
$$\eta(D)\psi(D)(f)(x)=\sum_{j=1}^\infty\eta(D)\psi(D)(f)(x)\mathbf{1}_{Q_{2,j}}(x).$$
For any $j\in\nn$, if $\|\eta(D)\psi(D)(f)\|_{L^\infty(Q_{2,j})}=0$, define
$$\lambda_{2,j}:=0\quad \text{and}\quad m_{2,j}:=0;$$
if $\|\eta(D)\psi(D)(f)\|_{L^\infty(Q_{2,j})}\neq0$, define
$$\lambda_{2,j}:=\|\mathbf{1}_{Q_{2,j}}\|_X\|
\eta(D)\psi(D)(f)\|_{L^\infty(Q_{2,j})} \quad \text{and} \quad
m_{2,j}:=\frac{\eta(D)\psi(D)(f)\mathbf{1}_{Q_{2,j}}}
{\|\mathbf{1}_{Q_{2,j}}\|_X\|\eta(D)\psi(D)(f)\|_{L^\infty(Q_{2,j})}}.$$
Then, for almost every $x\in\rn$,
\begin{equation*}
\eta(D)\psi(D)(f)(x)=\sum_{j=1}^\infty\lambda_{2,j}m_{2,j}(x).
\end{equation*}
Since $|Q_{2,j}|>1$ for any $j\in\nn$, it follows that,
for any $j\in\nn$, $m_{2,j}$ is a local-$(X,\infty,d,\tau)$-molecule
for any $\tau\in(0,\infty)$. Then, similarly to \eqref{10.31.1}, we find that,
for any $x\in\rn$,
\begin{align*}
&\left[\sum_{j=1}^\infty\left(\frac{\lambda_{2,j}}{\|\mathbf{1}_{Q_{2,j}}\|_X}\right)^s
\mathbf{1}_{Q_{2,j}}(x)\right]^{1/s}\\
&\hs\le\sup_{y\in\rn,|y-x|<c}\int_{\rn}|(\cf^{-1}\eta)(y-z)|
\frac{|(\psi(D)f(z)|}{(1+|z-x|)^b}(1+|z-x|)^b\,dz\lesssim(\psi_1^\ast f)_b(x),
\end{align*}
which, combined with Theorem \ref{mech}, further implies that
\begin{equation}\label{11.5b}
\lf\|\eta(D)\psi(D)(f)\r\|_{h_X(\rn)}\lesssim\left\|\left[\sum_{j=1}^\infty
\left(\frac{\lambda_{2,j}}{\|\mathbf{1}_{Q_{2,j}}\|_X}\right)^s
\mathbf{1}_{Q_{2,j}}\right]^{1/s}\right\|_X
\lesssim\lf\|(\psi_1^\ast f)_b\r\|_X\lesssim\|S_l(f)\|_X.
\end{equation}

Furthermore, by $\|S_l(f)\|_X<\infty$, we conclude that
$$\left\|\left\{\int_0^1\int_{B(\cdot,t)}\lf|\varphi(tD)(f)(y)\r|^2\,
\frac{dydt}{t^{n+1}}\right\}^{\frac{1}{2}}\right\|_X<\infty,$$
which implies that $\varphi(tD)(f)\mathbf{1}_{\{0<t<1\}}\in T^1_X(\mathbb{R}_+^{n+1})$.
From Lemma \ref{tad}, we deduce that there exist
a sequence $\{\lambda_j\}_{j\in\nn}\subset[0,\infty)$ and a sequence $\{a_j\}_{j\in\nn}$
of $(T_X^1,\infty)$-atoms supported, respectively,  in $\{T(B_j)\}_{j\in\nn}$
such that, for almost every $(x,t)\in\mathbb{R}^{n+1}_+$,
\begin{equation}\label{tsd1}
\varphi(tD)(f)(x)\mathbf{1}_{\{0<t<1\}}(t)=\sum_{j=1}^\infty\lambda_j a_j(x,t)
\end{equation}
and
\begin{equation}\label{tsd2}
|\varphi(tD)(f)(x)|\mathbf{1}_{\{0<t<1\}}(t)=\sum_{j=1}^\infty\lambda_j |a_j(x,t)|.
\end{equation}
Moreover,
\begin{equation}\label{lmnes}
\left\|\left[\sum_{j=1}^\infty
\left(\frac{\lambda_j}{\|\mathbf{1}_{B_j}\|_X}\right)^s
\mathbf{1}_{B_j}\right]^{1/s}\right\|_X\sim\|\varphi(tD)(f)
\mathbf{1}_{\{0<t<1\}}\|_{T^1_X(\mathbb{R}^{n+1}_+)}\lesssim\|S_l(f)\|_X.
\end{equation}
We now claim that
\begin{equation}\label{exor}
\int_0^1\phi(tD)\varphi(tD)(f)\,\frac{dt}{t}=\sum_{j=1}^\infty
\lambda_j\int_0^1\phi(tD)[a_j(\cdot,t)]\,\frac{dt}{t} \quad\text{in}\quad \cs'(\rn).
\end{equation}
To this end, for any $f\in\cs'(\rn)$, it follows that there exist positive integers $k$ and $m$,
depending only on $f$, such that, for any $x\in\rn$,
\begin{align}\label{scf}
|\varphi(tD)(f)(x)|&=\lf|\lf\langle f,\lf(\cf^{-1}\varphi\r)_t(x-\cdot)\r\rangle\r|\notag\\
&\lesssim \sum_{|\alpha|\leq k,|\beta|\leq m} \sup_{y\in\rn}\lf|y^\alpha
\partial^\beta\lf[\lf(\cf^{-1}\varphi\r)_t\r](x-y)\r|\notag\\
&\sim t^{-n}\sup_{|\alpha|\leq k,|\beta|\leq m}\sup_{y\in\rn}
\left|(x+y)^\alpha\partial^\beta\lf(\cf^{-1}\varphi\r)\left(\frac{y}{t}\right)\right|\notag\\
&\lesssim t^{-n}\sup_{|\alpha|\leq k,|\beta|\leq m}
\sup_{y\in\rn}\sum_{0\leq l\leq|\alpha|}|x|^l|y|^{|\alpha|-l}
\left|\partial^\beta\lf(\cf^{-1}\varphi\r)\left(\frac{y}{t}\right)\right|
\lesssim t^{-n-k} (1+|x|)^k.
\end{align}
By \cite[Lemma 2.4]{ysy} and the fact that, for any $\alpha \in \zz_+^n$,
$\int_{\rn} x^\alpha(\cf^{-1}\phi)(x)dx=0$, we find that,
for any $M\in\nn$ with $M>n+k+1$ and $x\in\rn$,
$$\lf|\lf(\cf^{-1}\phi\r)_t\ast\gamma(x)\r|\lesssim t^M(1+|x|)^{-n-M},$$
which, together with \eqref{tsd2} and \eqref{scf}, implies that
\begin{align*}
\int_0^1 \sum_{j=1}^\infty\lambda_j\lf|\lf\langle a_j(\cdot,t),(\cf^{-1}\phi)_t\ast
\gamma\r\rangle\r|\,\frac{dt}{t}
&\leq\int_0^1 \sum_{j=1}^\infty\lambda_j\int_{\rn}|a_j(x,t)|\left|(\cf^{-1}\phi)_t\ast
\gamma(x)\r|\,dx\,\frac{dt}{t}\\
&=\int_0^1 \int_{\rn}\left[\sum_{j=1}^\infty\lambda_j|a_j(x,t)|\right]\left|(\cf^{-1}\phi)_t\ast
\gamma(x)\r|\,dx\,\frac{dt}{t}\\
&=\int_0^1\int_{\rn}|\varphi(tD)(f)(x)|\lf|(\cf^{-1}\phi)_t\ast\gamma(x)\r|
\,dx\,\frac{dt}{t}\\
&\lesssim \int_0^1\int_{\rn}t^{-n-k} (1+|x|)^kt^M(1+|x|)^{-n-M}\,dx\frac{dt}{t}
\lesssim 1.
\end{align*}
From this, \eqref{tsd1} and the Fubini theorem, we deduce that, for any $\gamma\in\cs(\rn)$,
\begin{align*}
\left\langle \int_0^1\phi(tD)\varphi(tD)(f)\,\frac{dt}{t}, \gamma\right\rangle
&=\lim_{\varepsilon\to0^+}\left\langle \int_\varepsilon^1\phi(tD)\varphi(tD)(f)\,
\frac{dt}{t}, \gamma\right\rangle\\
&=\lim_{\varepsilon\to0^+}
\int_\varepsilon^1\left\langle\phi(tD)\varphi(tD)(f),\gamma\right\rangle\,\frac{dt}{t}\\
&=\lim_{\varepsilon\to0^+}\int_\varepsilon^1\left\langle \varphi(tD)(f), (\cf^{-1}\phi)_t\ast\gamma\right\rangle\,\frac{dt}{t}\\
&=\int_0^1 \sum_{j=1}^\infty\lambda_j\lf\langle a_j(\cdot,t),(\cf^{-1}\phi)_t\ast
\gamma\r\rangle\,\frac{dt}{t}\\
&=\sum_{j=1}^\infty\lambda_j\int_0^1
\lf\langle a_j(\cdot,t),(\cf^{-1}\phi)_t\ast\gamma\r\rangle\,\frac{dt}{t}\\
&=\left\langle\sum_{j=1}^\infty\lambda_j\int_0^1
\phi(tD)[a_j(\cdot,t)]\,\frac{dt}{t},\gamma\right\rangle,
\end{align*}
which proves \eqref{exor}, where $\varepsilon\to0^+$ means $\varepsilon\in(0,\fz)$ and $\varepsilon\to0$.
By \eqref{tsd2}, we conclude that, for any $t\in[1,\fz)$ and $x\in\rn$,
$a_j(x,t)=0$. Therefore, for any $t\in[1,\fz)$,
$$\phi(tD)(a_j(\cdot,t))=\int_{\rn}(\cf^{-1}\phi)_t(y)a_j(\cdot-y,t)\,dy=0,$$
which further implies that
$$\int_0^1\phi(tD)(a_j(\cdot,t))\,\frac{dt}{t}
=\int_0^\infty\phi(tD)(a_j(\cdot,t))\,\frac{dt}{t}.$$
Let $q\in(1,\infty)$, $d\in[d_X,\infty)\cap\zz_+$ and $\tau\in (n(1/\theta-1/q),\infty)$.
Then an argument similar to that used in the proof of \cite[Lemma 4.8]{hyy}
shows that, for any $j\in\nn$, $\int_0^1\phi(tD)[a_j(\cdot,t)]\,\frac{dt}{t}$ is
an $(X,q,d,\tau)$-molecule up to a harmless constant multiple.
From Theorem \ref{mech} and \eqref{lmnes}, it follows that
$$\left\|\int_0^1\phi(tD)\varphi(tD)(f)\,\frac{dt}{t}\right\|_{h_X(\rn)}
\sim\left\|\left[\sum_{j=1}^\infty
\left(\frac{\lambda_j}{\|\mathbf{1}_{B_j}\|_X}\right)^s
\mathbf{1}_{B_j}\right]^{1/s}\right\|_X\lesssim\|S_l(f)\|_X,$$
which, combined with \eqref{11.5a} and \eqref{11.5b},
further implies that $f\in h_X(\rn)$
and $\|f\|_{h_X(\rn)}\ls\|S_l(f)\|_X$. This finishes the proof of Lemma \ref{10.29.2}.
\end{proof}

The following lemma is the converse case of Lemma \ref{10.29.2}.

\begin{lemma}\label{10.9.5}
Let all the notation be the same as in Theorem \ref{lpl}. Then there exists a positive constant $C$ such that,
for any $f\in h_X(\rn)$,
$\|S_l(f)\|_X\le C\|f\|_{h_X(\rn)}$.
\end{lemma}

\begin{proof}
For any $f\in h_X(\rn)$, by Lemma  \ref{rhh}, we know that
$(\psi_1^\ast f)_b \in X$ and
\begin{equation}\label{11.2.1}
\lf\|(\psi_1^\ast f)_b\r\|_X\lesssim \|f\|_{h_X(\rn)}.
\end{equation}
For  any $x\in\rn$, let
$$\widetilde{S}_l(f)(x):=\left\{\int_0^1\int_{B(x,t)}\lf|\varphi(tD)(f)(y)\r|^2\,
\frac{dydt}{t^{n+1}}\right\}^{\frac{1}{2}}.$$
From Theorem \ref{ah}, we deduce that there exist a sequence $\{a_j\}_{j=1}^\infty$
of local-$(X,\infty,d)$-atoms supported,  respectively, in cubes
$\{Q_j\}_{j=1}^\infty$ and a sequence
$\{\lambda_j\}_{j=1}^\infty$ of non-negative numbers such that
\begin{equation}\label{10.9.1}
f=\sum_{j=1}^\infty \lambda_j a_j \quad\quad \text{in}\quad \mathcal{S}'(\rn)
\end{equation}
and
\begin{equation}\label{10.9.2}
\left\|\left[\sum_{j=1}^\infty\left(\frac{\lambda_j}{\|\mathbf{1}_{Q_j}\|_X}\right)^s
\mathbf{1}_{Q_j}\right]^{1/s}\right\|_X \lesssim\|f\|_{h_X(\rn)}.
\end{equation}
Let $\cq_1:=\{j\in\nn:\ \ell(Q_j)\in(0,1)\}$ and $\cq_2:=\{j\in\nn:\ \ell(Q_j)\in[1,\fz)\}$.
Then, by \eqref{10.9.1}, we find that, for any $x\in\rn$,
\begin{equation}\label{11.5c}
\widetilde{S}_l(f)(x)\le\sum_{j=1}^\fz\lambda_j\widetilde{S}_l\lf(a_j\r)(x)
=\sum_{j\in\cq_1}\lambda_{j}\widetilde{S}_l\lf(a_{j}\r)(x)+
\sum_{j\in\cq_2}\lambda_{j}\widetilde{S}_l\lf(a_j\r)(x)=:\mathrm{I}_1+\mathrm{I}_2.
\end{equation}
Using the fact that, for any $j\in\cq_1$, $a_{j}$ is an $(X,\infty,d)$-atom,
by an argument similar to that used in the proof of Lemma \ref{la}, we conclude that
\begin{equation}\label{11.5d}
\|\mathrm{I}_1\|_X\lesssim\left\|\left[\sum_{j\in\cq_1}\left(\frac{\lambda_{j}}
{\|\mathbf{1}_{Q_{j}}\|_X}\right)^s
\mathbf{1}_{Q_{j}}\right]^{1/s}\right\|_X.
\end{equation}

Now we  deal with $\mathrm{I}_2$. Let $q\in(1,\infty)$ be as in
Lemma \ref{r}. Since $\widetilde{S}_l$ is bounded on $L^q(\rn)$,
it follows that, for any $j\in\cq_2$,
$$\lf\|\mathbf{1}_{4Q_{j}}\widetilde{S}_l\lf(a_{j}\r)\r\|_X\lesssim
\|a_{j}\|_{L^q(\rn)}
\lesssim|Q_{j}|^{1/q}\|\mathbf{1}_{Q_{j}}\|_X^{-1},$$
which, combined with Lemma \ref{r}, implies that
\begin{equation}\label{10.9.3}
\left\|\sum_{j\in\cq_2}\lambda_{j}\mathbf{1}_{4Q_{j}}\widetilde{S}_l\lf(a_{j}\r)
\right\|_X\lesssim\left\|\left[\sum_{j\in\cq_2}^\infty\left(\frac{\lambda_{j}}
{\|\mathbf{1}_{Q_{j}}\|_X}\right)^s
\mathbf{1}_{Q_{j}}\right]^{1/s}\right\|_X.
\end{equation}
Moreover, by the fact that, for any $j\in\cq_2$,
$\ell(Q_{j})\ge1$, we know that, for any $t\in(0,1)$, $x\notin 4Q_{j}$,
$z\in Q_{j}$ and $y\in B(x,t)$,
$1\lesssim|y-z|\sim |x-x_{Q_{j}}|$, where
$x_{Q_{j}}$ denotes the center of the cube $Q_{j}$ and $\ell(Q_{j})$ its side length, which, together with
$\varphi \in \cs(\rn)$, further implies that, for any $j\in\nn$, $N\in\nn$, $t\in(0,1)$,
$x\notin 4Q_{j}$ and $y\in B(x,t)$,
\begin{align*}
|\varphi(tD)a_{j}(y)|&=\left|\int_{\rn}(\cf^{-1}(\varphi))_t(y-z)a_{j}(z)\,dz\right|
\le\int_{Q_{j}}|(\cf^{-1}(\varphi))_t(y-z)||a_{j}(z)|\,dz\\
&\lesssim\|a_{j}\|_{L^\infty(\rn)}t^{-n}\int_{Q_{j}}
\left(\frac{t}{|y-z|}\right)^N\,dz
\lesssim\|a_{j}\|_{L^\infty(\rn)}t^{N-n}\frac{[\ell(Q_{j})]^n}{|x-x_{Q_{j}}|^N}.
\end{align*}
From this, we deduce that, for any $j\in\cq_2$, $x\notin 4Q_{j}$
and $N\in(n/\theta,\infty)$,
\begin{align*}
\lf|\widetilde{S}_l\lf(a_{j}\r)(x)\r|&\lesssim \|a_{j}\|_{L^\infty(\rn)}
\frac{[\ell(Q_{j})]^n}{|x-x_{Q_{j}}|^N}\left(\int_0^1t^{2N-2n-1}\,dt\right)^{1/2}\\
&\lesssim \frac{1}{\|\mathbf{1}_{Q_{j}}\|_X}\frac{[\ell(Q_{j})]^N}
{|x-x_{Q_{j}}|^N}\lesssim\frac{1}{\|\mathbf{1}_{Q_{j}}\|_X}M^{(\theta)}
\lf(\mathbf{1}_{Q_{j}}\r)(x),
\end{align*}
which, combined with \eqref{ma}, implies that
\begin{equation}\label{10.9.4}
\left\|\sum_{j\in\cq_2}\lambda_{j}\mathbf{1}_{(4Q_{j})^\complement}
\widetilde{S}_l\lf(a_{j}\r)\right\|_X\lesssim
\left\|\left[\sum_{j\in\cq_2}\left(\frac{\lambda_{j}}
{\|\mathbf{1}_{Q_{j}}\|_X}\right)^s
\mathbf{1}_{Q_{j}}\right]^{1/s}\right\|_X.
\end{equation}
By \eqref{10.9.2}, \eqref{11.5c}, \eqref{11.5d}, \eqref{10.9.3} and \eqref{10.9.4},
we conclude that
$\|\widetilde{S}_l(f)\|_X\lesssim\|f\|_{h_X(\rn)}$,
which, together with \eqref{11.2.1}, implies that
$\|S_l(f)\|_X\ls\|f\|_{h_X(\rn)}$. This finishes the proof of Lemma \ref{10.9.1}.
\end{proof}

For any $\alpha \in [1,\infty)$, $f\in\cs'(\rn)$ and $x\in\rn$, let
\begin{equation*}
(\widetilde{S}_\alpha)_l(f)(x):=\left\{\int_0^1\int_{B(x,\alpha t)}
\lf|\varphi(tD)(f)(y)\r|^2\,
\frac{dydt}{t^{n+1}}\right\}^{\frac{1}{2}}.
\end{equation*}
Using Lemma \ref{af}, we obtain the aperture estimate of the local Lusin-area function
as follows.

\begin{lemma}\label{10.10.1}
Let $\az\in[1,\fz)$. Assume that $X$ is a ball quasi-Banach
function space satisfying both \eqref{ma} with $0<\theta<s\le 1$ and Assumption \ref{a2}
with some $q\in(1,\fz]$ and the same $s$ as in \eqref{ma}.
Then there exists a positive constant $C$, independent of $\alpha$, such that, for any  $f \in\cs'(\rn)$,
$$\lf\|(\widetilde{S}_\alpha)_l(f)\r\|_X\le C\max\lf\{\alpha^{(\frac{1}{2}-\frac1q)n},1\r\}
\alpha^{\frac n\theta}\lf\|\widetilde{S}_l(f)\r\|_X.$$
\end{lemma}

\begin{proof}
By Lemma \ref{af}, we find that, for any $\alpha\in[1,\infty)$,
\begin{align*}
\lf\|(\widetilde{S}_\alpha)_l(f)\r\|_X&=\lf\|\ca_\alpha\lf(|\varphi(tD)(f)
|\mathbf{1}_{(0,1)}\r)\r\|_X\\
&\lesssim \max\lf\{\alpha^{(\frac{1}{2}-\frac1q)n},1\r\}
\alpha^{\frac n\theta}\lf\|\ca_1(|\varphi(tD)(f)|\mathbf{1}_{(0,1)})\r\|_X\\
&\sim\max\lf\{\alpha^{(\frac{1}{2}-\frac1q)n},1\r\}
\alpha^{\frac n\theta}\lf\|\widetilde{S}_l(f)\r\|_X,
\end{align*}
which completes the proof of Lemma \ref{10.10.1}.
\end{proof}

Now we prove Theorem \ref{lpl} by using Lemmas \ref{10.29.2}
through \ref{10.10.1}.

\begin{proof}[Proof of Theorem \ref{lpl}]
(i) of this theorem is obtained directly by Lemmas \ref{10.29.2} and \ref{10.9.5}.
Moreover, via replacing Lemma \ref{af} by Lemma \ref{10.10.1} and then repeating
the proof of Theorem \ref{g}, we can prove (ii) and (iii) of this theorem,
and the details are omitted here. This finishes the proof of Theorem \ref{lpl}.
\end{proof}

\noindent\textbf{Acknowledgements}.
The first and the third authors would like to thank
Ciqiang Zhuo, Ziyi He and Yangyang Zhang for many helpful discussions
on the subject of this article. Moreover, the authors would
like to thank Yoshihiro Sawano and Kwok Pun Ho
for several helpful discussions on the subject of this article and they
would also like to thank the referee for her/his very careful reading and useful
comments which indeed improved the quality of this article and, particularly, motivated the authors
to obtain Theorem \ref{ccz-new}, Remark \ref{morr-lw}, Theorems \ref{gcz-new} and \ref{pdo-new}.

\bigskip

\noindent Fan Wang and Dachun Yang (Corresponding author)

\medskip

\noindent Laboratory of Mathematics and Complex Systems (Ministry of Education of China),
School of Mathematical Sciences, Beijing Normal University,
Beijing 100875, People's Republic of China

\smallskip

\noindent{\it E-mails:} \texttt{fanwang@mail.bnu.edu.cn} (F. Wang)

\noindent\phantom{{\it E-mails:} }\texttt{dcyang@bnu.edu.cn} (D. Yang)

\bigskip

\noindent Sibei Yang

\medskip

\noindent School of Mathematics and Statistics, Gansu Key Laboratory
of Applied Mathematics and Complex Systems, Lanzhou University,
Lanzhou 730000, People's Republic of China

\smallskip

\noindent{\it E-mail:} \texttt{yangsb@lzu.edu.cn}

\end{document}